\documentclass[nohyperref]{article}

\usepackage{microtype}
\usepackage{graphicx}
\usepackage{subfigure}
\usepackage{booktabs} 

\usepackage{hyperref}

\usepackage[preprint]{icml2022}

\usepackage{amsmath}
\usepackage{amssymb}
\usepackage{mathtools}
\usepackage{amsthm}
\usepackage{comment}

\usepackage{enumitem}

\usepackage[capitalize,noabbrev]{cleveref}

\theoremstyle{plain}
\newtheorem{theorem}{Theorem}[section]

\newtheorem{lemma}[theorem]{Lemma}
\newtheorem{corollary}[theorem]{Corollary}
\newtheorem{assumption}[theorem]{Assumption}
\newtheorem{definition}[theorem]{Definition}
\theoremstyle{definition}
\theoremstyle{remark}

\usepackage{caption}

\let\originalleft\left
\let\originalright\right
\renewcommand{\left}{\mathopen{}\mathclose\bgroup\originalleft}
\renewcommand{\right}{\aftergroup\egroup\originalright}

\newcommand{\norm}[1]{\left\lVert#1\right\rVert}
\newcommand{\abs}[1]{\left\lvert#1\right\rvert}

\newcommand{\bR}{\mathbb{R}}
\newcommand{\bP}[2][]{\Pr\ifthenelse{\isempty{#1}}{}{_{#1}}\left[#2\right]}
\newcommand{\bE}[2][]{\mathop\mathbb{E}\ifthenelse{\isempty{#1}}{}{_{#1}}\left[#2\right]}
\newcommand{\bI}[2][]{\mathop\mathbb{I}\ifthenelse{\isempty{#1}}{}{_{#1}}\left[#2\right]}
\newcommand{\Var}[2][]{\mathbf{Var}\ifthenelse{\isempty{#1}}{}{_{#1}}\left[#2\right]}

\newcommand{\horizonlength}{H}
\newcommand{\minimizer}{\theta}
\newcommand{\dist}{d_{\mathcal{G}}}

\DeclareMathOperator*{\argmin}{arg\,min}

\newcommand{\yiheng}[1]{\textcolor{purple}{[Yiheng says: #1]}}
\newcommand{\judy}[1]{\textcolor{blue}{[Judy says: #1]}}

\usepackage[textsize=tiny]{todonotes}

\icmltitlerunning{Decentralized Online Convex Optimization in Networked Systems}

\begin{document}

\twocolumn[
\icmltitle{Decentralized Online Convex Optimization in Networked Systems}



\icmlsetsymbol{equal}{*}

\begin{icmlauthorlist}
\icmlauthor{Yiheng Lin}{equal,Caltech}
\icmlauthor{Judy Gan}{equal,Columbia}
\icmlauthor{Guannan Qu}{CMU}
\icmlauthor{Yash kanoria}{Columbia}
\icmlauthor{Adam Wierman}{Caltech}
\end{icmlauthorlist}

\icmlaffiliation{Caltech}{Department of Computing and Mathematical Sciences, California Institute of Technology}
\icmlaffiliation{Columbia}{Decision, Risk, and Operations division, Columbia Business School}
\icmlaffiliation{CMU}{Department of Electrical and Computer Engineering, Carnegie Mellon University}

\icmlcorrespondingauthor{Yiheng Lin}{yihengl@caltech.edu}

\icmlkeywords{Machine Learning, ICML}

\vskip 0.3in
]



\printAffiliationsAndNotice{\icmlEqualContribution} 

\begin{abstract}
We study the problem of networked online convex optimization, where each agent individually decides on an action at every time step and agents cooperatively seek to minimize the total global cost over a finite horizon. The global cost is made up of three types of local costs: convex node costs, temporal interaction costs, and spatial interaction costs. In deciding their individual action at each time, an agent has access to predictions of local cost functions for the next $k$ time steps in an $r$-hop neighborhood. Our work proposes a novel online algorithm, Localized Predictive Control (LPC), which generalizes predictive control to multi-agent systems. We show that LPC achieves a competitive ratio of $1 + \tilde{O}(\rho_T^k) + \tilde{O}(\rho_S^r)$ in an adversarial setting, where $\rho_T$ and $\rho_S$ are constants in $(0, 1)$ that increase with the relative strength of temporal and spatial interaction costs, respectively. This is the first competitive ratio bound on decentralized predictive control for networked online convex optimization. Further, we show that the dependence on $k$ and $r$ in our results is near optimal by lower bounding the competitive ratio of any decentralized online algorithm.

\end{abstract}

\section{Introduction}



A wide variety of multi-agent systems can be modeled as optimization tasks in which individual agents must select actions based on local information with the goal of cooperatively learning to minimize a global objective in an uncertain, time varying environment.  This general setting emerges in applications such as formation control   \cite{chen2005, Oh2015}, power systems control \cite{molzahn2017survey, shi2021stability}, and multiproduct price optimization \cite{caro2012clearance, candogan2012optimal}. 
In all these cases, it is key that the algorithms used by agents use only local information due to the computational burden created by the size of the systems, the information constraints in the systems, and the need for fast and/or interpretable decisions.


At this point, there is a mature literature focused on decentralized optimization, e.g.  \citet{bertsekas1989parallel,boyd2011distributed, shi2015extra, nedic2018network}, see \citet{xin2020general} for a survey; however, the design of learning policies for uncertain, time-varying environments requires decentralized \textit{online} optimization. The literature studying decentralized online optimization is still nascent (see the related work section for a discussion of recent papers, e.g. \citet{Li2021coupled,Yuan2021,Yi2020}) and many challenging open questions remain. 
%
%


Three issues of particular importance for real-world applications are the following.

\textit{First,} temporal coupling in actions is often of first-order importance to applications.  For example, startup costs, ramping costs, and switching costs are prominent in settings such as power systems and cloud computing, and lead to penalties for changing actions dramatically over time.  The design of online algorithms to address such temporal interaction costs has received significant attention in the single-agent case recently, e.g, smoothed online optimization \cite{goel2019beyond, lin2020online}, convex body chasing \cite{argue2019chasing, sellke2019chasing}, online optimization with memory \cite{agarwal2019online, shi2020online}, and dynamic pricing \cite{besbes2015intertemporal,chen2018robust}. 

\textit{Second,} spatial interaction costs are of broad importance in practical applications.  Such costs arise because of the need for actions of nearby agents to be aligned with one another, and are prominent in settings such as economic team theory \cite{Marschak1955, marschak1972economic}, combinatorial optimization over graphs \cite{hochba1997approximation, Gamarnik2010}, and statistical inference \cite{wainwright2008graphical}. 
An example is (dynamic) multiproduct pricing, where the price of a product can impact the demand of other related products \cite{song2006measuring}. 

\textit{Third,} leveraging predictions of future costs has long been recognized as a promising way to improve the performance of online agents \cite{morari1999model, lin2012online, badiei2015online, chen2016using, shi2019value, li2020online}. As learning tools become more prominent, the role of predictions is growing.  By collecting data from repeated trials, data-driven learning tools make it possible to provide accurate predictions for near future costs. 
For example, in multiproduct pricing, good demand forecasts can be constructed up to a certain time horizon and are invaluable in setting prices \cite{caro2012clearance}.

In addition to the three issues above, we would like to highlight that existing results for decentralized online optimization focus on designing algorithms with low (static) regret \cite{Hosseini2016, Li2021coupled} 
, i.e., algorithms that (nearly) match the performance of the best static action in hindsight.  In a time-varying environment, it is desirable to instead obtain stronger bounds, such as those on the dynamic regret or competitive ratio, which compare to the dynamic optimal actions instead of the best static action in hindsight, e.g., see results in the centralized setting such as \citet{lin2020online,li2020online,shi2020online}. 


This paper aims to address decentralized online optimization with the three features described above.   In particular, we are motivated by the open question: \textit{Can a decentralized algorithm make use of predictions to be competitive for networked online convex optimization in an adversarial environment when spatial and temporal costs are considered?}

\textbf{Contributions.} 
This paper provides the first competitive algorithm for decentralized learning in networked online convex optimization. Agents in a network must each make a decision at each time step, to minimize a global cost which is the sum of convex node costs, spatial interaction costs and temporal interaction costs. We propose a predictive control framework called Localized Predictive Control (LPC, Algorithm \ref{alg:LPC}) and prove that it achieves a competitive ratio of $1 + \tilde{O}(\rho_T^k) + \tilde{O}(\rho_S^r)$, which approaches 1 exponentially fast as the prediction horizon $k$ and communication radius $r$ increase simultaneously. 
Our results quantify the improvement in competitive ratio from increasing the communication radius $r$ (which also increases the computational requirements) versus increasing the prediction horizon $k$, and imply that -- as a function of problem parameters -- one of the two ``resources'' $k$ and $r$ emerges as the bottleneck to algorithmic performance. 
Given any target competitive ratio, we find the minimum required prediction horizon $k$ and communication radius $r$ as functions of the temporal interaction strength and the spatial interaction strength, resp. 

Further, we show that LPC is order-optimal in terms of $k$ and $r$ by proving a lower bound on the competitive ratio for any online algorithm. We formalize the near optimality of our algorithm by showing that a resource augmentation bound follows from our upper and lower bounds: our algorithm with given $k$ and $r$ performs at least as well as the best possible algorithm that is forced to work with $k'$ and $r'$ which are a constant factor smaller than $k$ and $r$ respectively. 


The algorithm we propose, LPC, is inspired by Model Predictive Control (MPC). After fixing the prediction horizon $k$ and the communication radius $r$, each agent makes an individual decision by solving a $k$-time-step optimization problem, on a local neighborhood centered at itself and with radius $r$. In doing so, the algorithm utilizes all available information and makes a ``greedy'' decision. One benefit of this algorithm is its simplicity and interpretability, which is often important for practical applications. Moreover, since the algorithm is local, the computation needed for each agent is independent of the network size. 


Our main results are enabled by a new analysis methodology which obtains two separate decay factors for the propagation of decision errors (a temporal decaying factor $\rho_T$ and spatial decaying factor $\rho_S$) through a novel perturbation analysis. Specifically, the perturbation analysis seeks to answer the following question: If we perturb the boundary condition of an agent $v$'s $r$-hop neighborhood at the time step which is $\tau$-th later than the present, how does that affect $v$'s current decision, in terms of spatial distance $r$ and temporal distance $\tau$? With our analysis, we are able to bound the impact on $v$'s current decision by $O(\rho_T^{\tau}\rho_S^{r})$, where the decay factors $\rho_T$ and $\rho_S$ increase with the strength of temporal/spatial interactions among individual decisions. This novel analysis is critical for deriving a competitive ratio that distinguishes the decay rate for temporal and spatial distances.

To illustrate the use of our results in a concrete application, \Cref{apx:application} provides a detailed discussion of dynamic multiproduct pricing, which is a central problem in revenue management.  The resulting revenue maximization problem fits into our theoretical framework, and we deduce from our results that LPC guarantees near optimal revenue, in addition to reducing the computational burden \cite{schlosser2016stochastic} and providing interpretable prices \cite{biggs2021model} for products.

\textbf{Related Work.} 
This paper contributes to the literature in three related areas, each of which we describe below.

\textit{Distributed Online Convex Optimization.} Our work relates to a growing literature on distributed online convex optimization with time-varying cost functions over multi-agent networks. Many recent advances have been made including distributed OCO with delayed feedback \cite{Cao2021}, coordinating behaviors among agents \cite{Li2021, Cao2021relative}, and distributed OCO with a time-varying communication graph \cite{Hosseini2016, Akbari2017,Yuan2021, Li2021coupled,Yi2020}. 
A common theme of the previous literature is the idea that agents can only access partial information of time-varying global loss functions, thus requiring local information exchange between neighboring agents. 
To the best of our knowledge, our paper is the first in this literature to provide competitive ratio bounds or consider spatial and temporal costs, e.g., switching costs.  

\textit{Online Convex Optimization (OCO) with Switching Costs.} Online convex optimization with switching costs was first introduced in \citet{lin2012online} to model dynamic power management in data centers. Different forms of cost functions have been studied since then, e.g.,  \citet{chen2018smoothed, shi2020online, lin2020online}, in order to fit a variety of applications from video streaming \cite{joseph2012jointly} to energy markets \cite{kim2014real}. The quadratic form of switching cost was first proposed in \citet{goel2018smoothed} and yields connections to optimal control, which were further explored in \citet{lin2021perturbation}. The literature has focused entirely on the centralized, single-agent setting.  Our paper contributes to this literature by providing the first analysis of switching costs in a networked setting with a decentralized algorithm.

\textit{Perturbation Analysis of Online Algorithms.} Sensitivity analysis of convex optimization problems studying the properties of the optimal solutions as a function of the problem parameters has a long and rich history (see \citet{fiacco1990sensitivity} for a survey). The works that are most related to ours consider the specific class of problems where the decision variables are located on a horizon axis, or consider a general network and aim to show the impact of a perturbation on a decision variable is exponentially decaying in the graph distance from that variable, e.g.,  \citet{shin2021exponential, shin2021controllability, lin2021perturbation}. The idea of using exponentially decaying perturbation bounds to analyze an online algorithm is first proposed in \citet{lin2021perturbation}, where only the temporal dimension is considered. This style of perturbation analysis is key to the proof of our competitive bounds and, to prove our competitive bounds, we provide new perturbation results that separate the impact of spatial and temporal costs in a network for the first time. 
Additionally, our analysis is enabled by new results on the decay rate of a product of exponential decay matrices, which may be of independent interest.


\textbf{Notation.}  A complete notation table can be found in Appendix \ref{apx:notation-table}.  Here we describe the most commonly used notation.  In a graph $\mathcal{G} = (\mathcal{V}, \mathcal{E})$, we use $d_\mathcal{G}(v, u)$ to denote the distance (i.e. the length of the shortest path) between two vertices $v$ and $u$. 
$N_v^r$ denotes the $r$-hop neighborhood of vertex $v$, i.e., $N_v^r := \{u \in \mathcal{V}\mid d_\mathcal{G}(u, v) \leq r\}$.  $\partial N_v^r$ denotes the boundary of $N_v^r$, i.e., $\partial N_v^r = N_v^r\setminus N_v^{r-1}$. We  generalize these notations to temporal-spatial graphs as follows. Let $\times$ denote the Cartesian product of sets, and
$N_{(t, v)}^{(k, r)} :={} \{\tau \in \mathbb{Z}\mid t \leq \tau < t+k\} \times N_v^r,$ $\partial N_{(t, v)}^{(k, r)} :={} N_{(t, v)}^{(k, r)} \setminus N_{(t, v)}^{(k-1, r-1)}.$
For any subset of vertices $S$, we use $\mathcal{E}(S)$ to denote the set of all edges that have both endpoints in $S$, and define $S_+ = \{u \in \mathcal{V} \mid \exists v \in S \text{ s.t. } \dist(u, v) \leq 1\}$ (i.e., $S$ and its 1-hop neighbors). Let $\Delta$ denote the maximum degree of any vertex in $\mathcal{G}$; $h(r) := \sup_v \abs{\partial N_v^r}$. We say a function is in $C^2$ if it is twice continuously differentiable. We  use $\norm{\cdot}$ to denote the (Euclidean) 2-norm for vectors and the induced 2-norm for matrices.

\section{Problem Setting}\label{sec:setting}


We consider a set of agents in a networked system where each agent individually decides on an action at each time step and the agents cooperatively seek to minimize a global cost over a finite time horizon $\horizonlength$. Specifically, we consider a graph $\mathcal{G} = \left(\mathcal{V}, \mathcal{E}\right)$ of agents. Each vertex $v \in \mathcal{V}$ denotes an individual agent, and two agents $v$ and $u$ interact with each other if and only if they are connected by an undirected edge $(v, u) \in \mathcal{E}$. At each time step $t= 1, 2, \dots, \horizonlength$, each agent $v$ picks an $n$-dimensional local action $x_t^v \in D_t^v$, where $n$ is a positive integer and $D_t^v \subset \mathbb{R}^n$ is a convex set of feasible actions. The global action at time $t$ is the vector of agent actions
$x_t = \{x_t^v\}_{v \in \mathcal{V}},$ and incurs a global state cost, which is the sum of three types of local cost functions:
\vspace{-.1in}\begin{itemize}[nolistsep, leftmargin=*]
    \item \textbf{Node costs}: Each agent $v$ incurs a time-varying node cost $f_t^v(x_t^v)$, which characterizes agent $v$'s local preference for its local action $x_t^v$.
    \item \textbf{Temporal interaction costs}: Each agent $v$ incurs a time-varying temporal interaction cost $c_t^v(x_t^v, x_{t-1}^v)$, that characterizes how agent $v$'s previous local action $x_{t-1}^v$ interacts with its current local action $x_{t}^v$.
    \item \textbf{Spatial interaction costs}: Each edge $e = (v, u)$ incurs a time-varying spatial interaction cost\footnote{Since $e$ is an undirected edge, the order we write the two inputs (the action of $v$ and the action of $u$) does not matter. Note that $s_t^e$ can be asymmetric for agents $v$ and $u$, e.g., $s_t^e(x_t^v, x_t^u) = s_t^e(x_t^u, x_t^v) = \norm{x_t^v + 2 x_t^u}^2$.} $s_t^e(x_t^v, x_t^u)$. This characterizes how agents $v$ and $u$'s current local actions affect each other.
\end{itemize}

In our model, the node cost is the part of the cost that only depends the agent's current local action. If the other two types of costs are zero functions, each agent will trivially pick the minimizer of its node cost. Temporal interaction costs encourage agents to choose a local action that is ``compatible'' with their previous local action.  For example, a temporal interaction  could be a switching cost which penalizes large deviations from the previous action, in order to make the trajectory of local actions ``smooth''.  Such switching costs can be found in work on single-agent online convex optimization, e.g., \citet{chen2018smoothed, goel2019beyond, lin2020online}. Spatial interaction costs, on the other hand, can be used to enforce some collective behavior among the agents. For example, spatial interaction can model the probability that one agent's actions affects its neighbor's actions in diffusion processes on social networks \cite{Kempe2015}; or model interactions between complement/substitute products in multiproduct pricing \cite{candogan2012optimal}.



Our analysis is based on standard smoothness and convexity assumptions on the local cost functions (see Appendix \ref{apx:notation-table} for definitions of smoothness and strong convexity):  


\begin{assumption}\label{assump:costs-and-feasible-sets}
For $\mu > 0, \ell_f < \infty, \ell_T < \infty, \ell_S < \infty$, the local cost functions and feasible sets for all $t, v, e$ satisfy:
\begin{itemize}[nolistsep, leftmargin=*]
    \item $f_t^v: \mathbb{R}^n \to \mathbb{R}_{\geq0}$ is $\mu$-strongly convex, $\ell_f$-smooth, and in $C^2$;
    \item $c_t^v: \mathbb{R}^n \times \mathbb{R}^n \to \mathbb{R}_{\geq0}$ is convex, $\ell_T$-smooth, and in $C^2$;
    \item $s_t^e: \mathbb{R}^n \times \mathbb{R}^n \to \mathbb{R}_{\geq0}$ is convex, $\ell_S$-smooth, and in $C^2$;
    \item $D_t^v \subseteq \mathbb{R}^n$ satisfies $int(D_t^v) \not = \emptyset$ and can be written as
    $D_t^v := \{x_t^v \in \mathbb{R}^n \mid (g_t^v)_i(x_t^v) \leq 0, \forall 1 \leq i \leq m_t^v\},$
    where each  $(g_t^v)_i: \mathbb{R}^n \to \mathbb{R}$ is a convex function in $C^2$.
\end{itemize}
\end{assumption}

Note that the assumptions above are common, even in the case of single-agent online convex optimization, e.g., see \citet{li2020online, shi2020online, lin2021perturbation}.

It is useful to separate the global stage costs into two parts based on whether the cost term depends only on the current global action or whether it also depends on the previous action. Specifically, the part that only depends on the current global action $x_t$ is the sum of all node costs and spatial interaction costs. We refer to this component as the \textit{(global) hitting cost} and denote it as
\[f_t(x_t) := \sum_{v \in \mathcal{V}} f_t^v(x_t^v) + \sum_{(v, u) \in \mathcal{E}} s_t^{(v, u)}(x_t^v, x_t^u).\]
The rest of the global stage cost involves the current global action $x_t$ and the previous global action $x_{t-1}$. We refer to it as the \textit{(global) switching cost} and denote it as
\[c_t(x_t, x_{t-1}) = \sum_{v \in \mathcal{V}} c_t^v(x_t^v, x_{t-1}^v).\]
Combining the global hitting and switching costs, the networked agents work cooperatively to minimize the total global stage costs in a finite horizon $\horizonlength$ starting from a given initial global action $x_0$ at time step $0$:
$cost(ALG) := \sum_{t=1}^\horizonlength \left(f_t(x_t) + c_t(x_t, x_{t-1})\right),$
where $ALG$ denotes the decentralized online algorithm used by the agents.
The offline optimal cost is the clairvoyant minimum cost one can incur on the same sequence of cost functions and the initial global action $x_0$ at time step $0$, i.e.,
$cost(OPT) := \min_{x_{1:\horizonlength}} \sum_{t=1}^\horizonlength \left(f_t(x_t) + c_t(x_t, x_{t-1})\right).$

We measure the performance of any online algorithm ALG by the competitive ratio (CR), which is a widely-used  metric in the literature of online optimization, e.g., \citet{chen2018smoothed, goel2019beyond, argue2020dimension}. 
\begin{definition}
The competitive ratio of online algorithm $ALG$ is the supremum of $cost(ALG)/cost(OPT)$ over all possible problem instances, i.e.,
$CR(ALG) \coloneqq \sup_{\mathcal{G}, \horizonlength, x_0, \{f_t^v, c_t^v, s_t^e, D_t^v\}} cost(ALG)/cost(OPT).$
\end{definition}

Finally, we define the partial hitting and switching costs over subsets of the agents.  In particular, for a subset of agents $S \subseteq \mathcal{V}$, we denote the joint action over $S$ as $x_t^S := \{x_t^v \mid v \in S\}$ and define the partial hitting cost and partial switching cost over $S$ as
\begin{align}
    f_t^{S}(x_t^{S_+}) :={}& \sum_{v \in S} f_t^v(x_t^v) + \sum_{(v, u) \in \mathcal{E}(S_+)} s_t^{(v, u)}(x_t^v, x_t^u),\nonumber\\
    c_t^{S}(x_t^S, x_{t-1}^S) :={}& \sum_{v \in S} c_t^v(x_t^v, x_{t-1}^v),
    \label{eq:partial-costs}
\end{align}
This notation is useful for presenting decentralized online algorithms where the optimizations are performed over the $r$-hop neighborhood of each agent.


\subsection{Information Availability Model}\label{sec:look-ahead-and-communication}

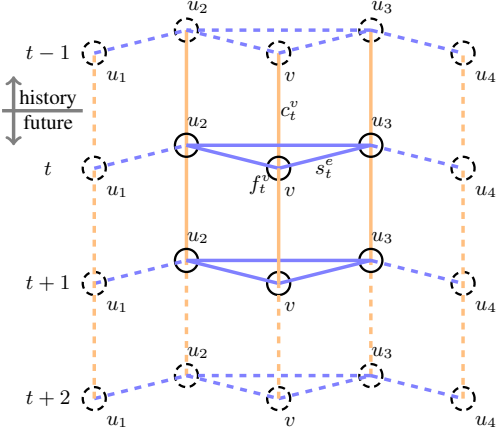
\begin{figure}
    \centering
    \resizebox{.4\textwidth}{!}{%
    \begin{tikzpicture}
\draw[draw=orange!50!white, line width=1.5pt, dashed] (40pt, 10pt) -- (40pt, 60pt);
\draw[draw=orange!50!white, line width=1.5pt] (40pt, 60pt) -- (40pt, 110pt);
\draw[draw=orange!50!white, line width=1.5pt] (40pt, 110pt) -- (40pt, 160pt);
\draw[draw=orange!50!white, line width=1.5pt, dashed] (120pt, 10pt) -- (120pt, 60pt);
\draw[draw=orange!50!white, line width=1.5pt] (120pt, 60pt) -- (120pt, 110pt);
\draw[draw=orange!50!white, line width=1.5pt] (120pt, 110pt) -- (120pt, 160pt);
\draw[draw=orange!50!white, line width=1.5pt] (80pt, 100pt) -- (80pt, 150pt);

\draw[draw=orange!50!white, line width=1.5pt, dashed] (0pt, 110pt) -- (0pt, 150pt);
\draw[draw=orange!50!white, line width=1.5pt, dashed] (160pt, 110pt) -- (160pt, 150pt);

\begin{scope}
    \draw[draw=black, line width=1pt, densely dashed] (0pt, 0pt) circle (5pt); 
    \node at (10pt, -10pt) {\footnotesize $u_1$};
    
    \draw[draw=black, line width=1pt, densely dashed] (40pt, 10pt) circle (5pt); 
    \node at (45pt, 20pt) {\footnotesize $u_2$};
    
    \draw[draw=black, line width=1pt, densely dashed] (80pt, 0pt) circle (5pt); 
    \node at (85pt, -10pt) {\footnotesize $v$};
    
    \draw[draw=black, line width=1pt, densely dashed] (120pt, 10pt) circle (5pt); 
    \node at (125pt, 20pt) {\footnotesize $u_3$};
    
    \draw[draw=black, line width=1pt, densely dashed] (160pt, 0pt) circle (5pt); 
    \node at (170pt, -10pt) {\footnotesize $u_4$};
    
    \draw[draw=blue!50!white, line width=1.5pt, dashed] (0pt, 0pt) -- (40pt, 10pt);
    \draw[draw=blue!50!white, line width=1.5pt, dashed] (40pt, 10pt) -- (80pt, 0pt);
    \draw[draw=blue!50!white, line width=1.5pt, dashed] (120pt, 10pt) -- (40pt, 10pt);
    \draw[draw=blue!50!white, line width=1.5pt, dashed] (80pt, 0pt) -- (120pt, 10pt);
    \draw[draw=blue!50!white, line width=1.5pt, dashed] (120pt, 10pt) -- (160pt, 0pt);
    
    \node at (-20pt, 0pt) {\footnotesize $t + 2$};
\end{scope}

\begin{scope}[yshift=50pt]
    \draw[draw=black, line width=1pt, densely dashed] (0pt, 0pt) circle (5pt); 
    \node at (10pt, -10pt) {\footnotesize $u_1$};
    
    \draw[draw=black, line width=1pt] (40pt, 10pt) circle (5pt); 
    \node at (45pt, 20pt) {\footnotesize $u_2$};
    
    \draw[draw=black, line width=1pt] (80pt, 0pt) circle (5pt); 
    \node at (85pt, -10pt) {\footnotesize $v$};
    
    \draw[draw=black, line width=1pt] (120pt, 10pt) circle (5pt); 
    \node at (125pt, 20pt) {\footnotesize $u_3$};
    
    \draw[draw=black, line width=1pt, densely dashed] (160pt, 0pt) circle (5pt); 
    \node at (170pt, -10pt) {\footnotesize $u_4$};
    
    \draw[draw=blue!50!white, line width=1.5pt, dashed] (0pt, 0pt) -- (40pt, 10pt);
    \draw[draw=blue!50!white, line width=1.5pt] (40pt, 10pt) -- (80pt, 0pt);
    \draw[draw=blue!50!white, line width=1.5pt] (120pt, 10pt) -- (40pt, 10pt);
    \draw[draw=blue!50!white, line width=1.5pt] (80pt, 0pt) -- (120pt, 10pt);
    \draw[draw=blue!50!white, line width=1.5pt, dashed] (120pt, 10pt) -- (160pt, 0pt);
    
    \node at (-20pt, 0pt) {\footnotesize $t + 1$};
\end{scope}

\begin{scope}[yshift=100pt]
    \draw[draw=black, line width=1pt, densely dashed] (0pt, 0pt) circle (5pt); 
    \node at (10pt, -10pt) {\footnotesize $u_1$};
    
    \draw[draw=black, line width=1pt] (40pt, 10pt) circle (5pt); 
    \node at (45pt, 20pt) {\footnotesize $u_2$};
    
    \draw[draw=black, line width=1pt] (80pt, 0pt) circle (5pt); 
    \node at (85pt, -10pt) {\footnotesize $v$};
    
    \draw[draw=black, line width=1pt] (120pt, 10pt) circle (5pt); 
    \node at (125pt, 20pt) {\footnotesize $u_3$};
    
    \draw[draw=black, line width=1pt, densely dashed] (160pt, 0pt) circle (5pt); 
    \node at (170pt, -10pt) {\footnotesize $u_4$};
    
    \draw[draw=blue!50!white, line width=1.5pt, dashed] (0pt, 0pt) -- (40pt, 10pt);
    \draw[draw=blue!50!white, line width=1.5pt] (40pt, 10pt) -- (80pt, 0pt);
    \draw[draw=blue!50!white, line width=1.5pt] (120pt, 10pt) -- (40pt, 10pt);
    \draw[draw=blue!50!white, line width=1.5pt] (80pt, 0pt) -- (120pt, 10pt);
    \draw[draw=blue!50!white, line width=1.5pt, dashed] (120pt, 10pt) -- (160pt, 0pt);
    
    \node at (-20pt, 0pt) {\footnotesize $t$};
    
    \node at (72pt, -7pt) {\footnotesize $f_t^v$};
    
    \node at (85pt, 25pt) {\footnotesize $c_t^v$};
    
    \node at (100pt, 0pt) {\footnotesize $s_t^e$};
\end{scope}

\begin{scope}[yshift=150pt]
    \draw[draw=black, line width=1pt, densely dashed] (0pt, 0pt) circle (5pt); 
    \node at (10pt, -10pt) {\footnotesize $u_1$};
    
    \draw[draw=black, line width=1pt, densely dashed] (40pt, 10pt) circle (5pt); 
    \node at (45pt, 20pt) {\footnotesize $u_2$};
    
    \draw[draw=black, line width=1pt, densely dashed] (80pt, 0pt) circle (5pt); 
    \node at (85pt, -10pt) {\footnotesize $v$};
    
    \draw[draw=black, line width=1pt, densely dashed] (120pt, 10pt) circle (5pt); 
    \node at (125pt, 20pt) {\footnotesize $u_3$};
    
    \draw[draw=black, line width=1pt, densely dashed] (160pt, 0pt) circle (5pt); 
    \node at (170pt, -10pt) {\footnotesize $u_4$};
    
    \draw[draw=blue!50!white, line width=1.5pt, dashed] (0pt, 0pt) -- (40pt, 10pt);
    \draw[draw=blue!50!white, line width=1.5pt, dashed] (40pt, 10pt) -- (80pt, 0pt);
    \draw[draw=blue!50!white, line width=1.5pt, dashed] (120pt, 10pt) -- (40pt, 10pt);
    \draw[draw=blue!50!white, line width=1.5pt, dashed] (80pt, 0pt) -- (120pt, 10pt);
    \draw[draw=blue!50!white, line width=1.5pt, dashed] (120pt, 10pt) -- (160pt, 0pt);
    
    \node at (-20pt, 0pt) {\footnotesize $t-1$};
    
    \node at (-20pt, -20pt) {\footnotesize history};
    \draw[draw=black!50!white, line width=1.5pt] (-40pt, -25pt) -- (-5pt, -25pt);
    \node at (-20pt, -30pt) {\footnotesize future};
    \draw[draw=black!50!white, line width=1.5pt, <->] (-35pt, -10pt) -- (-35pt, -40pt);
\end{scope}

\draw[draw=orange!50!white, line width=1.5pt, dashed] (0pt, 0pt) -- (0pt, 100pt);
\draw[draw=orange!50!white, line width=1.5pt, dashed] (80pt, 0pt) -- (80pt, 50pt);
\draw[draw=orange!50!white, line width=1.5pt] (80pt, 50pt) -- (80pt, 100pt);
\draw[draw=orange!50!white, line width=1.5pt, dashed] (160pt, 0pt) -- (160pt, 100pt);
\end{tikzpicture}
}
    \caption{Illustration of available information for agent $v$ at time $t$ in networked online convex optimization with $k = 2$ and $r = 1$, for the network with $\mathcal{V} = \{u_1, u_2, v, u_3, u_4\}$ and $\mathcal{E} = \{(u_1,u_2), (u_2,v), (v,u_3), (u_2, u_3), (u_3, u_4) \})$.}
    \label{fig:information}
\end{figure}


We assume that each agent has access to local cost functions up to a \emph{prediction horizon} $k$ into the future,  for themselves and their neighborhood up to a \emph{communication radius} $r$. In more detail, recall that $N_v^r$ denotes the $r$-hop neighborhood of agent $v$, i.e., $N_v^r := \{u \in \mathcal{V}\mid d_\mathcal{G}(u, v) \leq r\}$. Before picking a local action $x_t^v$ at time $t$, agent $v$ can observe $k$ steps of future node costs, temporal interaction costs, and spatial interaction costs within its $r$-hop neighborhood, 
    $\{\{(f_\tau^u, c_\tau^u)\mid u \in N_v^r\}, \{s_\tau^e\mid e \in \mathcal{E}(N_v^r)\}\}_{t \leq \tau < t+k}\, ,$
and the previous local actions in $N_v^r$: $\{x_{t-1}^u\mid u \in N_v^r\}$. 

We provide an illustration of the local cost functions known to agent $v$ at time $t$ in Figure \ref{fig:information}. In the figure, the black circles, blue lines, and orange lines denote the node costs, temporal interaction costs, and spatial interaction costs respectively. The known functions are marked by solid lines. Note that, in addition to the local cost functions, agent $v$ also knows the local actions in $N_v^r$ at time $t-1$, which are not illustrated in the figure. 

To simplify notation, in cases when the prediction horizon exceeds the whole horizon length $\horizonlength$, we adopt the convention that $f_t^v(x_t^v) = \frac{\mu}{2}\norm{x_t^v}^2$, $c_t^v \equiv s_t^e \equiv 0$ and $D_t^v = \mathbb{R}^n$ for $t > \horizonlength$. These extended definitions do not affect our original problem with horizon $\horizonlength$.

As in many previous works  studying the power of predictions in the online optimization literature, e.g., \citet{yu2020power, lin2020online, li2020online, lin2021perturbation}, we assume the $k$-step predictions of cost functions are \emph{exact} and leave the case of \emph{inexact} predictions for future work. This model is reasonable in the case where the predictors can be trained to be very accurate for the near future. Although we focus on exact predictions throughout this paper, we also discuss how to extend this available information model to include inexact predictions in Appendix \ref{sec:roadmap-to-inexact}.




\section{Algorithm and Main Results}\label{sec:main}

\begin{algorithm}[htb]
   \caption{Localized Predictive Control (for agent $v$)}
   \label{alg:LPC}
\begin{algorithmic}
   \STATE {\bfseries Parameters:} $k$ and $r$.
   \FOR{$t=1$ {\bfseries to} $\horizonlength$}
   \STATE Receive information $\{x_{t-1}^u\mid u \in N_v^r\}$ and
   \[\{\{(f_\tau^u, c_\tau^u)\mid u \in N_v^r\}, \{s_\tau^e\mid e \in \mathcal{E}(N_v^r)\}\}_{t \leq \tau < t+k}.\]
   \STATE Choose local action $x_t^v$ to be the $(t, v)$-th element in
   \[\psi_{(t, v)}^{(k, r)}\!\left(\!\{x_{t-1}^u\!\mid\! u\!\in\! N_v^r\}, \big\{\minimizer_\tau^u\!\mid\! (\tau, u) \!\in\! \partial N_{(t, v)}^{(k, r)}\big\}\!\right)\]
   the solution of \eqref{equ:LPC-opt-problem}, where $\minimizer_\tau^u \coloneqq \arg\min_{y \in D_\tau^u} f_\tau^u(y)$.
   \ENDFOR
\end{algorithmic}
\end{algorithm}

In this section we present our main results, which show that our simple and practical LPC algorithm can achieve an order-optimal competitive ratio for the networked online convex optimization problem.  We first introduce LPC  in Section \ref{sec:main:alg}. Then, we present the key idea that leads to our  competitive ratio bound: a novel perturbation-based analysis (Section \ref{sec:main:perturb}).  Next, we use our perturbation analysis to derive bounds on the competitive ratio in Section \ref{sec:main:CR}. Finally, we show that the competitive ratio of LPC is order-optimal in Section \ref{sec:main:lower}. An outline that highlights the major novelties in our proofs can be found in Appendix \ref{apx:proof-outline}.

\subsection{Localized Predictive Control (LPC)}\label{sec:main:alg}
The design of LPC is inspired by the classical model predictive control (MPC) framework \cite{garcia1989model}, which leverages all available information at the current time step to decide the current local action ``greedily''. In our context, when an agent $v$ wants to decide its action $x_t^v$ at time $t$, the available information includes previous local actions in the $r$-hop neighborhood and $k$-step predictions of all local node costs, temporal/spatial interaction costs. The boundaries of all available information, which are formed by $\{t-1\} \times N_v^r$ and $\partial N_{(t, v)}^{(k, r)}$, are illustrated in Figure \ref{fig:decision_boundary}. 

The pseudocode for LPC is presented in Algorithm \ref{alg:LPC}. For each agent $v$ at time step $t$, LPC fixes the actions on the boundaries of available information and then solves for the optimal actions inside the boundaries. Specifically, define $\psi_{(t, v)}^{(k, r)}\left(\{y_{t-1}^u\mid u \in N_v^r\}, \{z_\tau^u\mid (\tau, u) \in \partial N_{(t, v)}^{(k, r)}\}\right)$ as the optimal solution of the problem
\begin{align}\label{equ:LPC-opt-problem}
     \min & \sum_{\tau = t}^{t+k-1} \left(f_\tau^{(N_v^{r-1})}\left(x_\tau^{(N_v^r)}\right) + c_\tau^{(N_v^r)}\left(x_\tau^{(N_v^r)}, x_{\tau - 1}^{(N_v^r)}\right) \right)\nonumber\\*
     \text{ s.t. }& x_{t-1}^u = y_{t-1}^u, \forall u \in N_v^r,\nonumber\\*
     & x_\tau^u = z_\tau^u, \forall (\tau, u) \in \partial N_{(t, v)}^{(k, r)},\\*
     & x_\tau^u \in D_\tau^u, \forall (\tau, u) \in N_{(t, v)}^{(k-1, r-1)},\nonumber
\end{align}
where the partial hitting cost and partial switching cost $f_\tau^S$ and $c_\tau^S$ for a subset $S$ of agents 
were defined in \eqref{eq:partial-costs}.
Note that $\psi_{(t, v)}^{(k, r)}\left(\{y_{t-1}^u\}, \{z_\tau^u\}\right)$ is a matrix of actions (in $\bR^n$)
indexed by $(\tau, u) \in N_{(t, v)}^{(k-1, r-1)}$.
(When the context is clear, we use the shorthand $\psi_{(t, v)}^{(k, r)}\left(\{y_{t-1}^u\}, \{z_\tau^u\}\right)$.) 
Once the parameters $\{y_{t-1}^u\}$ and $\{z_\tau^u\}$ are fixed, the agent $v$ can leverage its knowledge of the local cost functions to solve \eqref{equ:LPC-opt-problem}. 

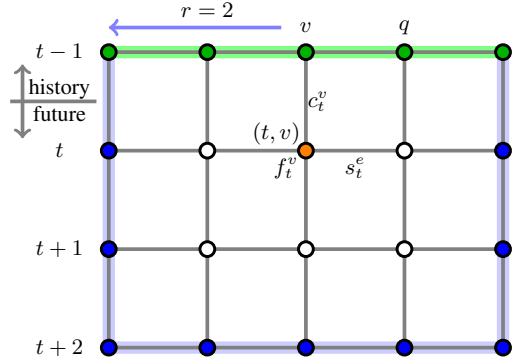
\begin{figure}[htb]
    \centering
    \resizebox{.4\textwidth}{!}{%
    \begin{tikzpicture}

\draw[draw=blue!20!white, line width=5pt] (0pt, 0pt) -- (160pt, 0pt);
\draw[draw=blue!20!white, line width=5pt] (0pt, 0pt) -- (0pt, 120pt);
\draw[draw=blue!20!white, line width=5pt] (160pt, 0pt) -- (160pt, 120pt);
\draw[draw=green!50!white, line width=5pt] (0pt, 120pt) -- (160pt, 120pt);

\draw[draw=black!50!white, line width=1.5pt] (0pt, 0pt) -- (0pt, 120pt);
\draw[draw=black!50!white, line width=1.5pt] (40pt, 0pt) -- (40pt, 120pt);
\draw[draw=black!50!white, line width=1.5pt] (80pt, 0pt) -- (80pt, 120pt);
\draw[draw=black!50!white, line width=1.5pt] (120pt, 0pt) -- (120pt, 120pt);
\draw[draw=black!50!white, line width=1.5pt] (160pt, 0pt) -- (160pt, 120pt);

\begin{scope}
    \draw[draw=black!50!white, line width=1.5pt] (0pt, 0pt) -- (160pt, 0pt);

    \filldraw[color=black, fill=blue, very thick](0pt,0pt) circle (3pt);
    
    \filldraw[color=black, fill=blue, very thick](40pt,0pt) circle (3pt);
    
    \filldraw[color=black, fill=blue, very thick](80pt,0pt) circle (3pt);
    
    \filldraw[color=black, fill=blue, very thick](120pt,0pt) circle (3pt);
    
    \filldraw[color=black, fill=blue, very thick](160pt,0pt) circle (3pt);
    
    \node at (-20pt, 0pt) {\footnotesize $t + 2$};
\end{scope}

\begin{scope}[yshift=40pt]
    \draw[draw=black!50!white, line width=1.5pt] (0pt, 0pt) -- (160pt, 0pt);

    \filldraw[color=black, fill=blue, very thick](0pt,0pt) circle (3pt);
    
    \filldraw[color=black, fill=white, very thick](40pt,0pt) circle (3pt);
    
    \filldraw[color=black, fill=white, very thick](80pt,0pt) circle (3pt);
    
    \filldraw[color=black, fill=white, very thick](120pt,0pt) circle (3pt);
    
    \filldraw[color=black, fill=blue, very thick](160pt,0pt) circle (3pt);
    
    \node at (-20pt, 0pt) {\footnotesize $t + 1$};
\end{scope}

\begin{scope}[yshift=80pt]
    \draw[draw=black!50!white, line width=1.5pt] (0pt, 0pt) -- (160pt, 0pt);

    \filldraw[color=black, fill=blue, very thick](0pt,0pt) circle (3pt);
    
    \filldraw[color=black, fill=white, very thick](40pt,0pt) circle (3pt);
    
    \filldraw[color=black, fill=orange, very thick](80pt,0pt) circle (3pt);
    
    \filldraw[color=black, fill=white, very thick](120pt,0pt) circle (3pt);
    
    \filldraw[color=black, fill=blue, very thick](160pt,0pt) circle (3pt);
    
    \node at (-20pt, 0pt) {\footnotesize $t$};
    
    \node at (68pt, 7pt) {\footnotesize $(t, v)$};
    
    \node at (72pt, -7pt) {\footnotesize $f_t^v$};
    
    \node at (85pt, 20pt) {\footnotesize $c_t^v$};
    
    \node at (100pt, -7pt) {\footnotesize $s_t^e$};
\end{scope}

\begin{scope}[yshift=120pt]
    \draw[draw=black!50!white, line width=1.5pt] (0pt, 0pt) -- (160pt, 0pt);

    \filldraw[color=black, fill=green!70!black, very thick](0pt,0pt) circle (3pt);
    
    \filldraw[color=black, fill=green!70!black, very thick](40pt,0pt) circle (3pt);
    
    \filldraw[color=black, fill=green!70!black, very thick](80pt,0pt) circle (3pt);
    
    \filldraw[color=black, fill=green!70!black, very thick](120pt,0pt) circle (3pt);
    
    \filldraw[color=black, fill=green!70!black, very thick](160pt,0pt) circle (3pt);
    
    \node at (-20pt, 0pt) {\footnotesize $t -1$};
    
    \node at (80pt, 10pt) {\footnotesize $v$};
    
    \node at (120pt, 10pt) {\footnotesize $q$};
    
    \draw[draw=blue!50!white, line width=1.5pt, ->] (70pt, 10pt) -- (0pt, 10pt);
    
    \node at (40pt, 17pt) {\footnotesize $r = 2$};
    
    \node at (-20pt, -15pt) {\footnotesize history};
    \draw[draw=black!50!white, line width=1.5pt] (-40pt, -20pt) -- (-5pt, -20pt);
    \node at (-20pt, -25pt) {\footnotesize future};
    \draw[draw=black!50!white, line width=1.5pt, <->] (-35pt, -5pt) -- (-35pt, -35pt);
\end{scope}

\end{tikzpicture}
}
    \caption{
    Illustration of LPC with $k = 3, r = 2$ on a line graph (the underlying graph is replicated over the time dimension). The orange node marks the decision variable at $(t, v)$. The green part denotes the decisions in $N_v^r$ at time $(t-1)$. The blue ``U'' shape denotes the boundary of available information for node $v$ at time $t$. Edge $e \coloneqq (v, q)$.
    }
    \label{fig:decision_boundary}
\end{figure}

LPC fixes the parameters $\{y_{t-1}^u\}$ to be $\{x_{t-1}^u\}$, which are the previous local actions in $N_v^r$, and fixes the parameters $\{z_{\tau}^u\}$ to be the minimizers of local node cost functions at nodes in $\partial N_{(t, v)}^{(k, r)}$. The selection of the parameters at nodes in $\partial N_{(t, v)}^{(k, r)}$ plays a similar role as the terminal cost of classical MPC in centralized settings. 

For a single-agent system, MPC-style algorithms are perhaps the most prominent approach for optimization-based control \cite{garcia1989model} because of their simplicity and excellent performance in practice. 
LPC extends the ideas of MPC to a multi-agent setting in a networked system by leveraging available information in both the temporal and spatial dimensions, whereas classical MPC focuses only on the temporal dimension. This change leads to significant technical challenges in the analysis. 



\subsection{Perturbation Analysis}\label{sec:main:perturb}


The key idea underlying our analysis of LPC is that the impact of perturbations to the actions at the boundaries of the available information of an agent decay quickly, in fact exponentially fast, in the distance of the boundary from the agent. This quick decay means that small errors cannot build up to hurt algorithm performance. 

In this section, we formally study such perturbations by deriving several new results which generalize perturbation bounds for online convex optimization problems on networks. Our bounds capture both the effect of temporal interactions as well as spatial interactions between agent actions, which is a more challenging problem compared to previous literature which considers either temporal interactions \cite{lin2021perturbation} or spatial interactions \cite{shin2021exponential} but not both simultaneously. 

More specifically, recall from Section \ref{sec:main:alg} that for each agent $v$ at time $t$, LPC solves an optimization problem $\psi_{(t, v)}^{(k, r)}$ where actions on the boundaries of available information (i.e., $\{t-1\} \times N_v^r$ and $\partial N_{(t, v)}^{(k, r)}$) are fixed. By the principle of optimality, we know that if the actions on the boundaries are selected to be identical with the offline optimal actions, the agent can decide its current action optimally by solving $\psi_{(t, v)}^{(k, r)}$. However, due to the limits on the prediction horizon and communication radius, LPC can only approximate the offline optimal actions on the boundaries (we do this by using the minimizer of node cost functions). The key idea to our analysis of the optimality gap of LPC is by first asking: \textit{If we perturb the parameters of $\psi_{(t, v)}^{(k, r)}$, i.e., the actions on the information boundaries, how large is the resulting change in a local action $x_{t_0}^{v_0}$ for $(t_0, v_0) \in N_{(t, v)}^{(k, r)}\setminus \partial N_{(t, v)}^{(k, r)}$ (in the optimal solution to \eqref{equ:LPC-opt-problem})?}

Ideally, we would like the above impact to decay exponentially fast in the graph distance between node $v_0$ and the communication boundary for node $v$ (i.e., $r$ minus the graph distance between $v_0$ and $v$), as well as in the temporal distance between $t_0$ and $t$. 
We formalize this goal as \textit{exponentially decaying local perturbation bound} in \Cref{thm:networked-exp-decay:meta}. We then show in \Cref{thm:networked-exp-decay} and \Cref{thm:networked-exp-decay-tight} that such bounds hold under appropriate assumptions.


\begin{definition}
\label{thm:networked-exp-decay:meta}
Define
$x_{t_0}^{v_0} \coloneqq \psi_{(t, v)}^{(k, r)}\left(\{y_{t-1}^u\}, \{z_\tau^u\}\right)_{(t_0, v_0)},$ and $(x_{t_0}^{v_0})' \coloneqq \psi_{(t, v)}^{(k, r)}\left(\{(y_{t-1}^u)'\}, \{(z_\tau^u)'\}\right)_{(t_0, v_0)}$ for arbitrary boundary parameters 
$\{(y_{t-1}^u)\}, \{(z_\tau^u)\}$ and $\{(y_{t-1}^u)'\}, \{(z_\tau^u)'\}$.
We say an \textbf{exponentially decaying local perturbation bound} holds if for non-negative constants
\begin{align*}
    C_1 ={}& C_1(\ell_T/\mu, (\Delta\ell_S)/\mu) < \infty,\\
    C_2 ={}& C_2(\ell_T/\mu, (\Delta\ell_S)/\mu) < \infty,\\
    \rho_T ={}& \rho_T(\ell_T/\mu) < 1, \rho_S = \rho_S((\Delta\ell_S)/\mu) < 1,
\end{align*}
for any $(t_0,v_0)$ 
and arbitrary boundary parameters $\{(y_{t-1}^u)'\}, \{(z_\tau^u)'\}, \{(y_{t-1}^u)\}, \{(z_\tau^u)\}$, we have: 
\begin{align*}
    &\norm{x_{t_0}^{v_0} - (x_{t_0}^{v_0})'}\\* \leq{}& C_1\sum_{(u, \tau) \in \partial N_{(t, v)}^{(k,r)} }  \rho_T^{|t_0-\tau|} \rho_S^{\dist(v_0, u)}
\norm{z_\tau^u - (z_\tau^u)'}\\*
&+ C_2\sum_{u \in N_v^r} \rho_T^{t_0 - (t-1)}\rho_S^{\dist(v_0, u)} \norm{y_{t-1}^u - (y_{t-1}^u)'}.
\end{align*}
\end{definition}
%
%
%
Perturbation bounds were recently found to be a promising tool for the analysis of adaptive control and online optimization models\cite{lin2021perturbation}.  
The exponentially decaying local perturbation bound defined above is similar in spirit to two recent results, i.e.,  \citet{lin2021perturbation} derives a similar perturbation bound for line graphs and \citet{shin2021exponential} for general graphs with local perturbations. In fact, one may attempt to derive such a 
bound by applying these results directly; however, a major weakness of the direct approach is that it will yield $\rho_T = \rho_S$, i.e., it cannot distinguish between spatial and temporal dependencies, and the bound deteriorates as $\max\{\ell_T/\mu, \ell_S/\mu\}$ increases. For instance, even if the temporal interactions are weak (i.e., $\ell_T/\mu \approx 0$), $\rho_T = \rho_S$ can still be close to $1$ if $\ell_S/\mu$ is large, leading to a large slack in the perturbation bound for small prediction horizons $k$. 

We overcome this limitation by redefining the action variables. Specifically, to focus on the temporal decay effect, we regroup all local actions in $\{\tau\}\times N_v^r$ as a ``large'' decision variable for time $\tau$ (in Figure \ref{fig:information} we would group each horizontal blue plane in $N_v^r$ to create a new variable). After regrouping, we have $(k+1)$ ``large'' decision variables located on a line graph, where the strength of the interactions between consecutive variables is upper bounded by $\ell_T$. On the other hand, to focus on spatial decay, we regroup all local actions in $\{\tau\mid t-1\leq \tau < t+k\} \times \{v\}$ as a decision variable (in Figure \ref{fig:information} we would group each vertical orange line connecting from $t-1$ to $t+k-1$ to create a new variable). After regrouping, we have $\abs{\mathcal{V}}$ ``large'' decision variables located on $\mathcal{G}$, where the strength of the interactions between two neighbors is upper bounded by $\ell_S$. Averaging over the two perturbation bounds (since we have two valid bounds, their average is also a valid bound) provides the following exponentially decaying local perturbation bound (see \eqref{thm:networked-exp-decay:e3} in Appendix \ref{apx:networked-exp-decay} for details of the proof).

\begin{theorem}\label{thm:networked-exp-decay}
Under Assumption \ref{assump:costs-and-feasible-sets}, the exponentially decaying local perturbation bound (\Cref{thm:networked-exp-decay:meta}) holds with $C_1 = C_2 = \frac{2\sqrt{\Delta \ell_S \ell_T}}{\mu}$, and
\begin{align*}
    \rho_T ={}& \sqrt{1 - 2\left(\sqrt{1 + ({2\ell_T}/{\mu})} + 1\right)^{-1}}, \\
    \rho_S ={}& \sqrt{1 - 2\left(\sqrt{1 + ({\Delta \ell_S}/{\mu})} + 1\right)^{-1}}.
\end{align*}
\end{theorem}

Note that, as $\ell_T/\mu$ (respectively $\ell_S/\mu$) tends to zero, $\rho_T$ (respectively $\rho_S$) in \Cref{thm:networked-exp-decay} also tends to zero 
with the scaling $\rho_T = \Theta(\sqrt{\ell_T/\mu})$ (resp. $\rho_S = \Theta(\sqrt{ \ell_S/\mu})$). 

Next, we provide a tighter bound (through a refined analysis) for the regime where $\mu$ is much larger than $\ell_T, \ell_S$. Specifically, we establish a bound with the scaling $\rho_T = \Theta(\ell_T/\mu)$ and $\rho_S = \Theta(\ell_S/\mu)$.  Again, it is not possible to obtain this result from previous perturbation bounds in the literature.  
\begin{theorem}\label{thm:networked-exp-decay-tight}
Recall $h(\gamma) \coloneqq \sup_{v \in \mathcal{V}} \abs{\partial N_v^{\gamma}}$. Given any $b_1, b_2 > 0$, 
define $a = \sum_{\gamma \ge 0}(\frac{1+b_1}{1+b_1+b_2})^{\gamma} h(\gamma)$, $\tilde a = \sum_{\gamma \ge 0}(\frac{1}{1+b_1})^{\gamma} h(\gamma)$ and $\gamma_S = \frac{\sqrt{1 + \Delta \ell_S/\mu} - 1}{\sqrt{1 + \Delta \ell_S/\mu} + 1}$.  Suppose Assumption \ref{assump:costs-and-feasible-sets} holds,  $a, \tilde a < \infty$ and $\mu \geq \max\{8\tilde a\ell_T, \Delta \ell_S (b_1 + b_2)/4\}$. Then
 the exponentially decaying local perturbation bound (\Cref{thm:networked-exp-decay:meta}) holds with $C_1 = C_2 = \max\{\frac{a^2}{2\tilde a(1 - 4\tilde al_T/\mu)}, \frac{2a^2\Delta \ell_S/\mu}{\gamma_S(1 + b_1 + b_2)(1 - 4\tilde al_T/\mu)}\}$ 
\begin{align*}
    \rho_T = \frac{4\tilde a\ell_T}{\mu}, \quad
    \rho_S = (1 + b_1 +b_2)\gamma_S.
\end{align*}
\end{theorem}
Note that $\rho_T, \rho_S < 1$ 
follow from the condition on $\mu$. Also observe that $\gamma_S = \Theta(\ell_S/\mu)$ as $\ell_S/\mu \to 0$.

The main difference between this result and  \Cref{thm:networked-exp-decay} is, instead of dividing and redefining the action variables, we explicitly write down the perturbations along spatial edges and along temporal edges in the original temporal-spatial graph. We observe that per-time-step spatial interactions are characterized by a banded matrix and that the inverse of the banded matrix exhibits exponential correlation decay, which implies the exponentially decaying local perturbation bounds holds if the perturbed boundary action and the impacted local action we consider are at the same the time step. 
However, for a multi-time-step problem, to characterize the impact at a local action at some time step due to perturbation at a boundary action at a different time step is a difficult problem. The main technical contribution of this proof is to establish that a product of exponentially decaying matrices still satisfies exponential decay under the conditions in  \Cref{thm:networked-exp-decay-tight}. In addition, we obtain a tight bound on the decay rate of the product matrix (see  \Cref{thm: l-power}), which may be of independent interest.

Our condition on $a, \tilde a < \infty$ and $\mu > \max\{8\tilde a\ell_T, \Delta \ell_S (b_1 + b_2)/4\}$ characterizes a tradeoff between the allowable neighborhood boundary sizes $h(\gamma)$, and how large $\mu$ needs to be compared to the interaction cost parameters $\ell_T, \ell_S$. At one extreme, if $h(\gamma) = \Delta^{\gamma}$, then by setting $b_1 = 2\Delta - 1$ and $b_2 = 4\Delta^2 - 2\Delta$, we obtain $a = \tilde a = 2$ but must make a strong requirement on $\mu$, namely, $\mu > \max\{16\ell_T, \Delta^3\ell_S(1 - \frac{1}{4\Delta^2})\}.$ At the other extreme, if $h(\gamma) \le O(poly(\gamma))$ (as is the case if $\mathcal{G}$ is a grid), then $a, \tilde a < \infty$ holds for any $b_1, b_2 > 0$, we can impose a weaker requirement on $\mu$: for example, taking $b_1 = b_2 = 1$ yields a requirement $\mu > \max\{8\tilde a \ell_T, \Delta \ell_S/2\}$ (where $\tilde a = \sum_{\gamma \ge 0} (\frac{1}{2})^{\gamma}h(\gamma)$); which grows only linearly in $\Delta$, and compares favorably with the $\mu > \Omega(\Delta^3)$ requirement which arose earlier.


Proofs of \Cref{thm:networked-exp-decay} and \Cref{thm:networked-exp-decay-tight} are in \Cref{apx:perturbation-bounds}. 

\subsection{From Perturbations to Competitive Bounds}\label{sec:main:CR}
We now present our main result, which bounds the competitive ratio of LPC using the exponentially decaying local perturbation bounds defined in the previous section.

Before presenting the result, we first provide some intuition as to why the perturbation bounds are useful for deriving the competitive ratio bound.  Specifically, to bound the competitive ratio requires bounding the gap between LPC's trajectory and the offline optimal trajectory.  This gap comes from the following two sources: (i) the per-time-step error made by LPC due to its limited prediction horizon and communication radius; and (ii) the cumulative impact of all per-time-step errors made in the past. Intuitively, the local perturbation bounds we derive in Section \ref{sec:main:perturb} allow us to bound the per-step error made jointly by all agents in LPC, and then we use the perturbation bounds from \citet{lin2021perturbation} help us to bound the second type of cumulative errors. 

We present our main result in the following theorem and defer a proof outline to Appendix \ref{sec:outline-competitive-ratio}. A formal proof can be found in Appendix \ref{apx:coro:CR-LPC}.

\begin{theorem}\label{coro:CR-LPC}
Suppose Assumption \ref{assump:costs-and-feasible-sets} and the exponentially decaying local perturbation bound in \Cref{thm:networked-exp-decay:meta} holds. Define 
$\rho_G \coloneqq 1 - 2 \cdot \left(\sqrt{1 + ({2\ell_T}/{\mu})} + 1\right)^{-1}, $
and define $C_3(r) \coloneqq \sum_{\gamma = 0}^{r} h(\gamma) \cdot \rho_S^\gamma$. If parameters $r$ and $k$ of LPC are large enough such that
$O\left(h(r)^2\cdot \rho_S^{2r} + C_3(r)^2 \cdot \rho_T^{2k} \cdot \rho_G^{2k}\right) \leq \frac{1}{2},$
then the competitive ratio of LPC is upper bounded by
\[1 + O\left(h(r)^2 \cdot \rho_S^r\right) + O\left(C_3(r)^2 \cdot \rho_T^{k}\right).\]
Here the $O(\cdot)$ notation hides factors that depend polynomially on $\ell_f/\mu, \ell_T/\mu,$ and $(\Delta \ell_S)/\mu$; see Appendix \ref{apx:coro:CR-LPC}.
\end{theorem}

Recall that $h(r)$ denotes the size of the largest $r$-hop boundary in $\mathcal{G}$. 
The bound in \Cref{coro:CR-LPC} implies that if $h(r)$ can be upper bounded by $poly(r)\cdot \rho_S^{-\frac{(1 - \iota) r}{2}}$ for some constant $\iota > 0$, the competitive ratio of LPC can be upper bounded by $1 + O(\rho_S^{\iota r}) + O(\rho_T^k)$, because $C_3(r)$ can be upper bounded by some constant that depends on $\iota$ in this case. Therefore, the competitive ratio improves exponentially with respect to the prediction horizon $k$ and communication radius $r$. 

Note that  the assumption
$h(r) \leq poly(r)\cdot \rho_S^{-\frac{(1 - \iota) r}{2}}$
is not particularly restrictive: For commonly seen graphs like an $m$-dimensional grid, $h(r)$ is polynomial in $r$, so $\iota = 1$ works. More generally, for graphs with bounded degree $\Delta<\infty$, there exists $\delta = \delta(\Delta)>0$ such that, for any graph with node degrees bounded above by $\Delta$ and any ${\ell_S}/{\mu} \leq \delta$, we have $\rho_S$ (from either Theorem~\ref{thm:networked-exp-decay} or Theorem~\ref{thm:networked-exp-decay-tight}) will be small enough that, e.g., $h(r) \leq \Delta^r = O( \rho_S^{-\frac{r}{4}})$; i.e., $\iota = 1/2$ works. 
Thus we can eliminate the dependence on $h(r)$ and $C_3(r)$ in the competitive ratio by making additional assumptions on $\ell_S/\mu$. This result is stated in Corollary \ref{coro:pure-exp-decay} whose proof is deferred to Appendix \ref{apx:coro:pure-exp-decay}. Corollary \ref{coro:pure-exp-decay} is a corollary of \Cref{coro:CR-LPC} and \Cref{thm:networked-exp-decay-tight}. We use the bound in \Cref{thm:networked-exp-decay-tight} and not the bound in \Cref{thm:networked-exp-decay} because \Cref{thm:networked-exp-decay-tight} is tighter when $\ell_S/\mu$ is small.

\begin{corollary}\label{coro:pure-exp-decay}
Suppose \Cref{assump:costs-and-feasible-sets} and inequalities $\ell_S/\mu \leq {\Delta^{-7}}$, and $\ell_T/\mu \leq 1/16$ hold. If $r$ and $k$ satisfy that
$O\left(\rho_S^{r} + \rho_T^{2k} \cdot \rho_G^{2k}\right) \leq \frac{1}{2},$
then the competitive ratio of LPC is upper bounded by
$1 + O\left(\rho_S^{r/2}\right) + O\left(\rho_T^{k}\right),$
where $\rho_S$ and $\rho_T$ are given by \Cref{thm:networked-exp-decay-tight} with parameters $b_1 = 2\Delta -1$ and $b_2 = 4\Delta^2 - 2\Delta$.
The $O(\cdot)$ notation hides factors that depend polynomially on $\ell_f/\mu, \ell_T/\mu,$ and $(\Delta \ell_S)/\mu,$ see Appendix \ref{apx:coro:pure-exp-decay} for the full constants.
\end{corollary}

\subsection{A Lower Bound}\label{sec:main:lower}
We show that the competitive ratio in \Cref{coro:CR-LPC} is order-optimal by deriving a lower bound on the competitive ratio of any decentralized online algorithm with  prediction horizon $k$ and communication radius $r$. The specific constants and a proof of \Cref{thm:bottleneck} can be found in Appendix \ref{apx:thm:bottleneck}. 
\begin{theorem}\label{thm:bottleneck}
When $\Delta \geq 3$, the competitive ratio of any decentralized online algorithm is lower bounded by $1 + \Omega(\lambda_T^k) + \Omega(\lambda_S^r)$. Here, the decay factor $\lambda_T$ is given by $\lambda_T = \left(1 - 2\left(\sqrt{1 + (4\ell_T/\mu)} + 1\right)^{-1}\right)^2$. The decay factor $\lambda_S$ is given by $\lambda_S = \frac{(\Delta \ell_S/\mu)}{3 + 3 (\Delta \ell_S/\mu)}$ if $\Delta \ell_S/\mu < 48$; $\lambda_S = \max\left(\frac{(\Delta \ell_S/\mu)}{3 + 3 (\Delta \ell_S/\mu)}, \left(1 - 4 \sqrt{3} \cdot (\Delta \ell_S/\mu)^{-\frac{1}{2}}\right)^2\right)$ otherwise. The $\Omega(\cdot)$ notation hides factors that depend polynomially on $1/\mu, \ell_T,$ and $\ell_S$.
\end{theorem}


While \Cref{thm:bottleneck} highlights that \Cref{coro:CR-LPC} is order-optimal, the decay factors $\lambda_T, \lambda_S$ in the lower bound differ from their counterparts $\rho_T, \rho_S$ in the upper bound for LPC. To understand the magnitude of the difference, we compare the bounds on graphs with bounded degree $\Delta$. 
The decay factors are a function of the interaction strengths, which are measured by ${\ell_S}/{\mu}$ and ${\ell_T}/{\mu}$. Our lower bound on the temporal decay factor $\lambda_T$ and upper bound $\rho_T$ only differ by a constant factor in the log-scale, and the same holds for the lower/upper bound in terms of the spatial decay factor.

To formalize this comparison, we derive a resource augmentation bound that bounds the additional ``resources'' that LPC needs to outperform the optimal decentralized online algorithm.\footnote{See, e.g., \citet{roughgarden2020resource}, for an introduction to this flavor of bounds for expressing the near-optimality of an algorithm.} Here the prediction horizon $k$ and the communication radius $r$ can be viewed as the ``resources'' available to a decentralized online algorithm in our setting. We ask \textit{how large do $k$ and $r$ given to LPC need to be, to ensure that it beats the optimal decentralized online algorithm given a communication radius $r^*$ and prediction horizon $k^*$?} 

We formally state our result in the following corollary 
and provide a proof in Appendix \ref{apx:coro:resource-augmentation}.
\begin{corollary}\label{coro:resource-augmentation}
Under Assumption \ref{assump:costs-and-feasible-sets}, suppose the optimal decentralized online algorithm achieves a competitive ratio of $c(k^*, r^*)$ with prediction horizon $k^*$ and communication radius $r^*$. Additionally assume that $h(\gamma) = \tilde{O}\left(\rho_S^{-\gamma/4}\right)$ 
and $\Delta \geq 3$, where the $\tilde{O}$ notation hides a factor that depends polynomially on $\gamma$. As $k^*, r^* \to \infty$, LPC achieves a competitive ratio at least as good as that of the optimal decentralized online algorithm when LPC uses a  
    prediction horizon of $k = (4 + o(1)) k^*$
    and a communication radius of $r = (32 + o(1)) r^*$. 
\end{corollary}



Finally, note that we establish Corollary \ref{coro:resource-augmentation} based on the local perturbation bound in \Cref{thm:networked-exp-decay} rather than \Cref{thm:networked-exp-decay-tight} for simplicity, because it does not make assumptions on the relationship among $\mu, \ell_T,$ and $\ell_S$. We expect that \Cref{thm:networked-exp-decay-tight} can give better resource augmentation bounds under stronger assumptions on $\mu, \ell_T,$ and $\ell_S$.

\section{Concluding Remarks}

In this work, we introduce and study a novel form of decentralized online convex optimization in a networked system, where the local actions of each agent are coupled by temporal interactions and spatial interactions. We propose a decentralized online algorithm, LPC, which leverages all available information within a prediction horizon of length $k$ and a communication radius of $r$ to achieve a competitive ratio of $1 + \tilde{O}(\rho_T^k) + \tilde{O}(\rho_S^r)$. Our lower bound result shows that this competitive ratio is order optimal. Our results imply that the two types of resources, the prediction horizon and the communication radius, must be improved \textit{simultaneously} in order to obtain a competitive ratio that converges to $1$. That is, it is not enough to either have a large communication radius or a long prediction horizon, the combination of both is necessary to approach the hindsight optimal performance.




A limitation of this work is that we have considered only exact predictions in the available information model, the LPC algorithm, and its analysis. Considering inexact predictions is an important future goal and we are optimistic that our work can be generalized in that direction using the roadmap in Appendix \ref{sec:roadmap-to-inexact}.

\textbf{Acknowledgement}:
We would like to thank ICML reviewers and meta-reviewer, as well as Yisong Yue for their valuable feedback on this work. This work is supported by NSF grants ECCS-2154171, CNS-2146814, CPS-2136197, CNS-2106403, NGSDI-2105648, CMMI-1653477, with additional support from Amazon AWS, PIMCO, and the Resnick Sustainability Insitute. Yiheng Lin was supported by Kortschak Scholars program.

\bibliography{main}
\bibliographystyle{icml2022}

\newpage
\appendix
\onecolumn

\section{Notation Summary and Definitions}\label{apx:notation-table}

The notation we use throughout the paper is summarized in the following two tables. 

\begin{table}[H]
  \centering
  \caption{Notation related to the graph/network structures.}
  \begin{tabular}{c|l}
    \specialrule{1.5pt}{0pt}{0pt}
    \textbf{Notation} & \hspace*{12.5em}\textbf{Meaning} \\
    \specialrule{0.4pt}{0pt}{0pt}
    $\mathcal{G} = (\mathcal{V}, \mathcal{E})$ & The network of agents connected by undirected edges;  \\
    $\dist(u, v)$ & The graph distance (i.e. the length of the shortest path) between two vertices $u$ and $v$ in $\mathcal{G}$; \\
    $N_v^r$ & The $r$-hop neighborhood of vertex/agent $v$ in $\mathcal{G}$, i.e., $N_v^r := \{u \in \mathcal{V}\mid \dist(u, v) \leq r\}$;\\
    $\partial N_v^r$ & The boundary of the $r$-hop neighborhood of vertex/agent $v$, i.e., $\partial N_v^r = N_v^r \setminus N_v^{r-1}$;\\
    $h(r)$ & The size of the largest $r$-hop boundary in $\mathcal{G}$, i.e., $h(r) := \sup_v\abs{\partial N_v^r}$;\\
    $\Delta$ & The maximum degree of any vertex $v$ in $\mathcal{G}$;\\
    $\mathcal{E}(S)$ & The set of all edges that have both endpoints in $S$, where $S \subseteq \mathcal{V}$;\\
    $S_+$ & The extension of $S$ by 1-hop, i.e., $S_+ = S \cup \{u\mid \exists v \in S \text{ s.t. } (u, v) \in \mathcal{E}\}$;\\
    $N_{(t, v)}^{(k, r)}$ & $\{\tau \in \mathbb{Z}\mid t \leq \tau < t+k\} \times N_v^r$, which is a set of (time, vertex) pairs;\\
    $\partial N_{(t, v)}^{(k, r)}$ & $N_{(t,v)}^{(k, r)} \setminus N_{(t,v)}^{(k-1, r-1)}$, which is a set of (time, vertex) pairs;\\
    \specialrule{1.5pt}{0pt}{0pt}
  \end{tabular}
\end{table}

\begin{table}[H]
  \centering
  \caption{Notation related to the optimization problems.}
  \begin{tabular}{c|l}
    \specialrule{1.5pt}{0pt}{0pt}
    \textbf{Notation} & \hspace*{12.5em}\textbf{Meaning} \\
    \specialrule{0.4pt}{0pt}{0pt}
    $\norm{\cdot}$ & The (Euclidean) 2-norm for vectors and the induced 2-norm for matrices;\\
    $\horizonlength$ & The whole horizon length of Networked OCO problem;\\
    $[\horizonlength]$ & The sequence of integers $1, 2, \ldots, \horizonlength$;\\
    $\mathbb{S}^m$ & For any positive integer $m$, $\mathbb{S}^m$ denotes the set of all symmetric real $m \times m$ matrices;\\ 
    $y_{t_1 : t_2}$ & The sequence $y_{t_1}, y_{t_1 + 1}, \ldots, y_{t_2}$, for $t_2 \geq t_1$;\\
    $x_t^v$ & The individual action of agent $v$ at time step $t$. It is a vector in $\mathbb{R}^n$;\\
    $x_t^S$ & The joint action of all agents in set $S \subseteq \mathcal{V}$ at time $t$, i.e., $x_t^S = \{x_t^v\}_{v \in S}$;\\
    $x_t$ & The joint action of all agents in $\mathcal{V}$ at time $t$, i.e., $x_t = \{x_t^v\}_{v \in \mathcal{V}}$. It is a shorthand of $x_t^\mathcal{V}$;\\
    $x_t^*$ & The offline optimal joint action of all agents at time $t$;\\
    $x_{\tau\mid t}^*$ & The clairvoyant joint decision of all agents at time $\tau$ given that the joint decision is $x_t$ at time $t$;\\
    $f_t^v(x_t^v)$ & The node cost function for agent $v \in \mathcal{V}$ at time step $t$;\\
    $c_t^v(x_t^v, x_{t-1}^v)$ & The temporal interaction cost function for agent $v \in \mathcal{V}$ at time step $t$;\\
    $s_t^e(x_t^v, x_t^u)$ & The spatial interaction cost for edge $e = (v, u) \in \mathcal{E}$ at time step $t$;\\
    $\mu$ & The strong convexity constant of node costs $f_t^v$;\\
    $\ell_f, \ell_T, \ell_S$ & The smoothness constant of node costs, temporal interaction costs, and spatial interaction costs;\\
    $D_t^v$ & The feasible set of $x_t^v$ for agent $v$ at time $t$. It is a convex subset of $\mathbb{R}^n$;\\
    $\minimizer_t^v$ & The minimizer of node cost function for $v$ at time $t$ subject to $D_t^v$, i.e., $\minimizer_t^v = \argmin_{y \in D_t^v} f_t^v(y)$;\\
    $f_t^S(x_t^{S_+})$ & The total node costs and spatial interaction costs over a subset $S \subseteq \mathcal{V}$ at time $t$, i.e.,\\
    &$f_t^{S}(x_t^S) := \sum_{v \in S} f_t^v(x_t^v) + \sum_{(v, u) \in \mathcal{E}(S)} s_t^{(v, u)}(x_t^v, x_t^u)$;\\
    $c_t^S(x_t^S)$ & The total temporal interaction costs over a subset $S \subseteq \mathcal{V}$ at time $t$, i.e., $c_t^S(x_t^S) := \sum_{v \in S} c_t^v(x_t^v, x_{t-1}^v)$;\\
    $f_t(x_t)$ & The total node costs and spatial interaction costs over a $\mathcal{V}$ at time $t$. A shorthand of $f_t^\mathcal{V}(x_t^\mathcal{V})$;\\
    $c_t(x_t)$ & The total temporal interaction costs over $\mathcal{V}$ at time $t$. A shorthand of $c_t^\mathcal{V}(x_t^{\mathcal{V}})$;\\
    $\psi_{(t, v)}^{(k, r)}(\cdot, \cdot)$ & The optimal individual decisions in $N_{(t, v)}^{(k, r)}$ when the decision boundaries formed by $\{t-1\}\times N_v^r$\\
    &and $\partial N_{(t, v)}^{(k, r)}$ are fixed as parameters;\\
    $\tilde{\psi}_t^p(\cdot, \cdot)$ & The optimal global trajectory $x_t, x_{t+1}, \ldots, x_{t+p-2}$ when $x_{t-1}$ and $x_{t+p-1}$ are fixed as parameters;\\
    $\tilde{\psi}_t(\cdot)$ & The optimal global trajectory $x_t, x_{t+1}, \ldots, x_{T}$ when $x_{t-1}$ is fixed as the parameter;\\
    \specialrule{1.5pt}{0pt}{0pt}
  \end{tabular}
\end{table}

In addition to the notation in the tables above, we make use of the concepts of strong convexity and smoothness throughout this paper.
\begin{definition}
    For a fixed dimension $m \in \mathbb{Z}_+$, let $D \subset \bR^m$ be a convex set, and suppose function $\hat{h} : D \to \bR$ is a differentiable function. Then, $\hat{h}$ is called $\ell$-smooth for some constant $\ell \in \bR_{\geq 0}$ if 
    \[ \hat{h}(y) \leq \hat{h}(x) + \langle \nabla \hat{h}(x), y-x \rangle + \frac{\ell}{2} \norm{y - x}_2^2, \forall x, y \in \bR^m,\]
    and is called $\mu$-strongly convex for some constant $\mu \in \bR_{\geq 0}$ if 
    \[ \hat{h}(y) \geq \hat{h}(x) + \langle \nabla \hat{h}(x), y-x \rangle + \frac{\mu}{2} \norm{y - x}_2^2, \forall x, y \in \bR^m.\]
    Here $\langle \cdot, \cdot \rangle$ denotes the dot product of vectors.
\end{definition}

\section{Example: Multiproduct Pricing}\label{apx:application}
The networked online convex optimization problem captures many applications where individual agents must make decisions online in a networked system. In this section, we give a concrete motivating example from multiproduct pricing problems. Many recent books of revenue management \cite{Talluri2006, Gallego2019} and papers \cite{song2006measuring,caro2012clearance,candogan2012optimal, Chen2015} describe numerous instances of this problem. In our example below, we follow a similar pricing model as in \citet{candogan2012optimal}.  

Consider a setting where a large company sells $n$ different products and wants to maximize its revenue by adjusting prices adaptively in a time-varying market. Each vertex/agent $v \in \mathcal{V}$ corresponds to a product and $x_t^v$ denotes its price at time $t$. Two products $v$ and $u$ are connected by an edge if they interact, e.g., because the products are complements or substitutes.


We assume a classical linear demand model \cite{Talluri2006, Gallego2019}, 
where the demand of $v$ at time $t$, denoted as $d_t^v$, is given by
\[d_t^v = \underbrace{a_t^v - k_t^v x_t^v}_\text{Part 1} \overbrace{ - \sum_{u \in N_v^1\setminus \{v\}} \eta_t^{(u\to v)} x_t^u}^\text{Part 2} \underbrace{+ b_t^v x_{t-1}^v}_\text{Part 3},\]
with parameters $a_t^v, k_t^v,  b_t^v > 0$ and $\eta_t^{(u\to v)} \in \mathbb{R}$.
Here, Part 1 corresponds to the nominal demand at price $x_t^v$; Part 2 adds the network externalities from $v$'s complements/substitutes, and Part 3 reflects the pent up demand of product $v$ due to high price at time $t-1$. Note that the coefficient $\eta_t^{(u \to v)}$ can be different with $\eta_t^{(v \to u)}$. To simplify the notations, for each undirected edge $e = (u, v)$, we define an aggregate coefficient $\gamma_t^{e} \coloneqq \frac{1}{2}\left(\eta_t^{(u \to v)} + \eta_t^{(v \to u)}\right)$.

The full revenue maximization problem can be written as
\begin{equation}
\begin{aligned}
\max \quad & \sum_{t = 1}^\horizonlength \sum_{v \in \mathcal{V}} x_t^v d_t^v  = \sum_{t = 1}^\horizonlength \sum_{v \in \mathcal{V}} x_t^v({a_t^v - k_t^v x_t^v}  - \sum_{u \in N_v^1\setminus \{v\}} \eta_t^{(u\to v)} x_t^u + b_t^v x_{t-1}^v)\\
\textrm{s.t.} \quad & 0 \le x_t^v \le \overline p_t^v,
\end{aligned}
\end{equation}
which is equivalent to the following:
\begin{equation}\label{opt-problem:min}
\begin{aligned}
\min \quad & -\sum_{t = 1}^\horizonlength\sum_{v \in \mathcal{V}} x_t^v d_t^v  = \sum_{t = 1}^\horizonlength\sum_{v \in \mathcal{V}}\Bigg[ x_t^v(-a_t^v + k_t^v x_t^v + \sum_{u \in N_v^1\setminus \{v\}} \eta_t^{(u\to v)} x_t^u - b_t^v x_{t-1}^v)\Bigg]\\
\textrm{s.t.} \quad & 0 \le x_t^v \le \overline p_t^v
\end{aligned}
\end{equation}

We assume that the product's own price elasticity coefficient $k_t^v$ is uniformly larger than the sum of magnitudes of cross-elasticity coefficients, i.e., exist $\mu > 0$, s.t.
\[\xi_t^v := k_t^v - \sum_{u \in N_v^1\setminus {v}}|\gamma_t^{(u, v)}| - \frac{b_{t}^v + b_{t+1}^v}{2} \geq \mu/2 > 0\]
holds for any $(t, v)$.
Further, we assume $\sup_{v \in \mathcal{V}, t \in \horizonlength}k_t^v \le \ell_f/2$, $\sup_{v \in \mathcal{V}, t \in \horizonlength}b_t^v \le b$, and $\sup_{(u, v) \in \mathcal{E}, t \in \horizonlength}|\eta_t^{(u\to v)}| \le \gamma$.

We now observe that this problem fits within our framework with node, spatial, and temporal costs which are quadratic and defined as follows.
\begin{align*}
    f_t^v(x_t^v) \coloneqq {}& \xi_t^v \left(x_t^v - \frac{a_t^v}{2\xi_t^v}\right)^2,\\
    s_t^{(u, v)}(x_t^u, x_t^v) \coloneqq {}& |\gamma_t^{(u, v)}| \left(x_t^u + sgn\left(\gamma_t^{(u, v)}\right)\cdot x_t^v\right)^2,\\
    c_t^{v}(x_t^v, x_{t-1}^v) \coloneqq {}& \frac{b_t^v}{2}\left(x_t^v - x_{t-1}^v\right)^2.
\end{align*}
Note that $f_t^v(x_t^v)$ is $\xi_t^v (x_t^v)^2 - {a_t^v}{x_t^v}$ plus a constant $(a_t^v)^2/(4 \xi_t^v)$, and the interaction functions can be rewritten as
$$s_t^{(u, v)}(x_t^u, x_t^v) = |\gamma_t^{(u, v)}| \Bigg((x_t^u)^2 +  (x_t^v)^2\Bigg) +  2\gamma_t^{(u, v)} x_t^v x_t^u, \quad c_t^{v}(x_t^v, x_{t-1}^v) = \frac{b_t^v}{2}\Bigg((x_t^v)^2 +  (x_{t-1}^v)^2\Bigg) - b_tx_{t-1}^v x_{t}^v.$$
Summing, we see that 
\begin{align}
\sum_{t = 1}^\horizonlength\sum_{v \in \mathcal{V}} f_t^v(x_t^v) + c_t^v(x_t^v, x_{t-1}^v) + \sum_{t=1}^\horizonlength \sum_{e \in \mathcal{E}}s_t^e(x_t^u, x_t^v) = \textup{(Objective in \eqref{opt-problem:min})} + \sum_{t=1}^\horizonlength \sum_{v \in \mathcal{V}} (a_t^v)^2/(4 \xi_t^v) \, .
\end{align}

Hence the optimal solution of \eqref{opt-problem:min} is the same as the following problem: 

\begin{equation}
\begin{aligned}
\min \quad & \sum_{t = 1}^\horizonlength\sum_{v \in \mathcal{V}} f_t^v(x_t^v) + c_t^v(x_t^v, x_{t-1}^v) + \sum_{t=1}^\horizonlength \sum_{e \in \mathcal{E}}s_t^e(x_t^u, x_t^v)\\
\textrm{s.t.} \quad & 0 \le x_t^v \le \overline p_t^v
\end{aligned}
\end{equation}
where the node cost function $f_t^v(x_t^v)$ is nonnegative, $\mu$-strongly convex, and $\ell_f$-smooth;
the spatial interaction function is nonnegative, convex and $(4\gamma)$-smooth; the temporal interaction function is nonegative, convex and $(2b)$-smooth.

The decentralized nature of our policy is important in this setting. Interpretable pricing algorithms \cite{biggs2021model} are attractive in practice. Our local pricing algorithm is indeed interpretable, since the current price of a given product is transparently determined by reliable predictions of demand in the near future as well as interactions  with directly related products. 

In addition, exactly solving the global multiproduct pricing problem can be computationally challenging in practice, especially when the network is large. For example, large online e-commerce companies  maintain millions of products, which makes the entire network difficult to store, let alone do computation over. Moreover, due to the ease of changing prices, e-commerce companies often use dynamic pricing and change prices on a daily (or quicker) basis, which magnifies the computational burden.


\subsection{Competitive Bound}
We end our discussion of multiproduct pricing by showing how the competitive bound in \Cref{coro:CR-LPC} can be applied to the revenue maximization problem of product networks through the use of the following lemma.   

\begin{lemma}\label{le:cr-rev}
Suppose the competitive ratio of our general cost minimization problem is $CR(k,r)$, which a function of prediction horizon $k$ and communication radius $r$. Suppose $\sup_{(u, v) \in \mathcal{E}, t \in [\horizonlength]}a_t^u/a_t^v \le \tilde b$, $\sup_{v \in \mathcal{V}, t \in [\horizonlength]} \frac{a_t^v}{\overline{p}_t^v} \le \tilde c $, then the competitive ratio for the corresponding revenue maximization problem, defined as $rev(ALG)/rev(OPT)$, is at least $1 - \frac{\eta}{2}(CR(k, r)-1)$, where $\Delta$ denotes the degree of the product network and $\eta \coloneqq \max\{2(\ell_f + \Delta \tilde b \gamma)/\mu, \tilde c/\mu\}$.
\end{lemma}
\begin{proof}
We define $C \coloneqq \sum_{t, v} (a_t^v)^2/(4\xi_t^v)$. 
Suppose
$$CR(k, r)\cdot cost(OPT) \ge cost(ALG),$$ then
$$CR(k, r)\cdot (-rev(OPT) + C) \ge (-rev(ALG) + C).$$

Rearranging the terms yields 
\begin{equation}\label{eq:CR-cost-revenue}
(CR(k, r)-1) \cdot C \ge CR(k, r) rev(OPT) - rev(ALG).
\end{equation}

To find a lower bound on $rev(OPT)$, we choose a pricing strategy such that $x_t^v = \frac{a_t^v}{\eta\mu} \ge 0$ where $\eta = \max\{2(\ell_f + \Delta \tilde b \gamma)/\mu, \tilde c/\mu\}$. 
We first check that the demand is always nonnegative under this strategy: 
\begin{align*}
a_t^v - k_t^v\frac{a_t^v}{\eta\mu} - \sum_{u \in N_v^1\setminus \{v\}} \eta_t^{(u\to v)} \frac{a_t^u}{\eta\mu} + b_t^v \frac{a_{t-1}^v}{\eta\mu} &\ge a_t^v - k_t^v\frac{a_t^v}{\eta\mu} - \sum_{u \in N_v^1\setminus \{v\}} \eta_t^{(u\to v)} \frac{a_t^u}{\eta\mu} \\
&\ge a_t^v - k_t^v\frac{a_t^v}{\eta\mu} - \sum_{u \in N_v^1\setminus \{v\}} \tilde b\gamma  \frac{a_t^v}{\eta\mu} \\
&\ge a_t^v (1 - \frac{\ell_f + \Delta \tilde b \gamma }{\eta \mu}) \\
&\ge \frac{a_t^v}{2}.
\end{align*}
Moreover, 
$$x_t^v \le a_t^v/\tilde c \le \overline{p}_t^v.$$
Hence this is a feasible price strategy. 

We lower bound the optimal revenue:
\begin{align*}
rev(OPT) &\ge \sum_{t = 1}^\horizonlength\sum_{v \in \mathcal{V}} \frac{a_t^v}{\eta\mu}({a_t^v - k_t^v \frac{a_t^v}{\eta\mu}} - \sum_{u \in N_v^1\setminus \{v\}} \eta_t^{(u\to v)} \frac{a_t^u}{\eta\mu} + b_t^v \frac{a_{t-1}^v}{\eta\mu})\\
& \ge \sum_{t = 1}^\horizonlength\sum_{v \in \mathcal{V}} \frac{a_t^v}{\eta\mu} \frac{a_t^v}{2}\\
& \ge \frac{2}{\eta}C.
\end{align*}
We further divide \Cref{eq:CR-cost-revenue} by $rev(OPT)$ to obtain
$$(CR(k, r) - 1) \frac{C}{rev(OPT)} \ge CR(k, r) - rev(ALG)/rev(OPT).$$

Since $CR(k, r) \ge 1$ for the cost minimization problem, 
$$rev(ALG)/rev(OPT) \ge 1 - (CR(k, r) - 1) \frac{C}{rev(OPT)}.$$


This allows us to complete the proof as follows
$$rev(ALG)/rev(OPT) \ge 1 - \frac{\eta}{2}(CR(k, r) - 1).$$
\end{proof}

\section{Proof Outline}\label{apx:proof-outline}

In this section, we outline the major novelties in our proofs for the tighter exponentially decaying local perturbation bound in \Cref{thm:networked-exp-decay-tight} and the main competitive ratio bound for LPC in \Cref{coro:CR-LPC}.  The full details of the proofs of these and other results are in the appendices following this one.

\subsection{Refined Analysis of Perturbation Bounds}\label{sec:outline-network-perturbation-bound}

We begin by outlining the four-step structure we use to prove Theorem \ref{thm:networked-exp-decay-tight}.  Our goal is to highlight the main ideas, while deferring a detailed proof to Appendix \ref{apx:networked-exp-decay-tight}.

\subsubsection*{Step 1. Establish first order equations}
We define $h$ as the objective function in \eqref{equ:LPC-opt-problem}, where actions on the boundary are fixed as $\{z_{\tau}^u | (\tau, u) \in  \partial N_{(t, v)}^{(k, r)}\}$ and the actions at time $t-1$ are fixed as $\{x_{t-1}^{u} | u \in N_v^r \}$. We denote those fixed actions as system parameter $$\zeta \coloneqq (x_{t-1}^{(N_v^r)}, \{z_{\tau}^u | (\tau, u) \in  \partial N_{(t, v)}^{(k, r)}\}).$$
To avoid writing the time index $t$ repeatedly, we use $\hat x_i$ to denote actions at time $t-1+i$ for $0 \le i \le k$.  The main lemma in for this step is the following.

\begin{lemma}\label{le:foc}
Given $\theta \in \mathbb{R}$, system parameter $\zeta$ and perturbation vector $e$, we have
\begin{equation*}\label{eq:foc}
\frac{d}{d\theta} \psi(\zeta + \theta e) = M^{-1} \Bigg(R^{(1)}e_0 + R^{(k-1)}e_k + \sum_{\tau = 1}^{k-1} K^{(\tau)} e_{\tau}\Bigg)
\end{equation*}
where $$M= \nabla_{\hat x_{1:k-1}}^2 h(\psi(\zeta + \theta e), \zeta + \theta e),$$ 
$$R^{(1)} \coloneqq -\nabla_{\hat x_0}\nabla_{\hat x_{1:k-1}}  h(\psi(\zeta + \theta e), \zeta + \theta e),$$ 
$$R^{(k-1)} \coloneqq -\nabla_{\hat x_k}\nabla_{\hat x_{1:k-1}}  h(\psi(\zeta + \theta e), \zeta + \theta e),$$ 
$$K^{(\tau)} \coloneqq  -\nabla_{\hat x_{\tau}^{(\partial N_v^r)}}\nabla_{\hat x_{1:k-1}} h(\psi(\zeta + \theta e), \zeta + \theta e).$$
\end{lemma}
The proof for Lemma \ref{le:foc} using first order conditions at the global optimal solution for convex function $h(\cdot , \zeta + \theta e)$ and then takes derivatives with respect to to $\theta$. See Appendix \ref{step1} for a proof.

\subsubsection*{Step 2: Exploit the structure of matrix $M$} $M$ is a hierarchical block matrix with the first level of dimension $(k-1) \times (k-1)$. When fixing the first level indices (i.e. time indices) in $M$, the lower level matrices are non-zero only if the difference in the time indices is $\le 1$. Hence we decompose $M$ to a block diagonal matrix $D$ and tri-diagonal block matrix $A$ with zero matrix on the diagonal. Each diagonal block in $D$ is a graph-induced banded matrix, which captures the Hessian of $h$ in a single time step. Denote each diagonal block as $D_{i,i}$ for $1 \le i \le k-1$. Further, for $1 \le i \le k-1$, $A_{i, i-1}$ (similarly $A_{i, i+1}$) captures the temporal correlation of individual's action between consecutive time steps. Given this decomposition, 
$$M^{-1} = (D + A)^{-1} = D^{-1} (I + AD^{-1})^{-1}.$$
For the ease of notation, we denote $I + AD^{-1}$ by $P.$ 
Note that $P$ is not necessarily a symmetric matrix. Nevertheless, under technical conditions on $P$'s eigenvalues, we have the following power series expansion \cite{2020}. The details are presented in the \Cref{le: H-inverse} in Appendix \ref{apx:networked-exp-decay-tight}. 
\begin{lemma}\label{le: H-inverse}
Under the condition $\mu > 2\ell_{T}$, we have 
\begin{equation}\label{eq:power-series}
    P^{-1} = \sum_{\ell \ge 0} (I - P)^{\ell}.
\end{equation}
\end{lemma}

To understand the the power series in \eqref{eq:power-series}, consider the special case where each block $A_{i, j} = \ell_{T}\cdot I$, and $D_{i, i} = Q $. Denote $P - {I}$ as $J$, which is equivalent to $AD^{-1}$. 
Then, we have $J_{i, i} = 0$, $J_{i, i-1} = J_{i, i+1} = \ell_{T}Q^{-1}$, $J_{i, j} = 0 \textup{ when $|i-j| > 1$}.$ Intuitively, $J$ captures the ``correlation over actions'' after one time step.
More generally, for $\ell \ge 0$ and any two time indices $\tau'$, $\tau$, 
$$J^{\ell}_{\tau', \tau} = 
\ell_{T}^{\ell}Q^{-\ell}b(\ell, \tau, \tau'),$$
where $b(\ell, \tau, \tau')$ is a constant depending on $\ell, \tau, \tau'$ and equal to zero if $\ell < |\tau - \tau'|$.

Given that $Q$ is a graph-induced banded matrix, $Q^{-1}$ satisfies exponential-decay properties, which makes it plausible that $Q^{-\ell}$ is an exponential decay matrix with a slower rate. 

For the general case, we need to bound terms similar to $\norm{(D_{i_1, i_1}^{-1} D_{i_2, i_2}^{-1} \cdots D_{i_{\ell}, i_{\ell}}^{-1})_{u, v}}$. This is the goal of Step 3.  

\subsubsection*{Step 3: Properties for general exponential-decay matrices}

The goal of this step is to establish that a product of exponential decay matrices still exhibits exponential decay property under technical conditions about the underlying graph. 

\begin{lemma}\label{thm: l-power}
Given any graph $\mathcal{M} = (\mathcal{V'}, \mathcal{E'})$ and integers $d, \ell \ge 1$, suppose block matrices $A_i \in \bR^{|\mathcal{V'}|d \times |\mathcal{V'}|d}$ all satisfy exponential decay properties, i.e. exists $C_i \ge 0$, and $0 \le \lambda < 1$, s.t.,
$$\norm{(A_i)_{u, q}} \le C_i \lambda^{d_{\mathcal{M}}(u, q)} \quad \textup{for any node $u, q \in \mathcal{M}$}.$$

Select some $\delta > 0$ s.t. $\lambda' = \lambda + \delta < 1$. If $\tilde a \coloneqq \sum_{k = 0}^{\infty}   (\frac{\lambda}{\lambda'})^{k}(\sup_{u \in \mathcal{V'}}|\partial N_u^k|)< \infty$,
then $\prod_{i=1}^{\ell} A_i$ satisfies exponential decay properties with decay rate $\lambda'$, i.e., 
$$\norm{(\prod_{i=1}^{\ell}A_i)_{u, v}} \le C' {(\lambda')}^{d_M(u, v)}.$$
where $C' = (\tilde a)^{\ell}\prod_{i=1}^{\ell} C_{i}$.
\end{lemma}

A proof of Lemma \ref{thm: l-power} can be found in the Appendix \ref{pf:l-power}.

\subsubsection*{Step 4: Establish correlation decay properties of matrix $M$}
The last step of the proof is to study the properties of $M$.  To accomplish this, we first show that, for time indices  $i, j \ge 1$, $J^{\ell}$ has the following properties:
\begin{itemize}
    \item $(J^{\ell})_{i, j} = 0$ if $\ell < |i-j|$ or $\ell - |i - j|$ is odd. 
    \item $(J^{\ell})_{i, j}$ is a summation of terms $\prod_{k = 1}^{\ell}A_{j_k, i_k}D_{i_k, i_k}^{-1}$ and the number of such terms is bounded by $\binom{\ell}{(\ell - |i-j|)/2}$.
\end{itemize}
We formally state and prove the above observation in Lemma \ref{le:J-properties}. We can further use Theorem \ref{thm: l-power} on block matrices $A_{j_k, i_k}D_{i_k}^{-1}$, which gives the following lemma. 

\begin{lemma}\label{le:J-bound}
Recall $\gamma_S \coloneqq  \frac{\sqrt{1 + (\Delta \ell_{S}/\mu)} - 1}{\sqrt{1 + (\Delta \ell_{S}/\mu)} + 1}$. Select $\delta > 0$ s.t. $\gamma_S' = \gamma_S + \delta < 1$ and $b \coloneqq \sum_{\gamma \ge 0}(\frac{\gamma_S}{\gamma_S'})^{\gamma} h(\gamma)$.
Given $\ell, i, j \ge 1$ and $u, q \in \mathcal{V}$, we have
$$\norm{((J^{\ell})_{i, j})_{u, q}} \le \binom{\ell}{(\ell - |i-j|)/2}(b\frac{2\ell_{T}}{\mu})^{\ell} (\gamma_S')^{\dist(u, q)}.$$
\end{lemma}

Intuitively speaking, Lemma \ref{le:J-bound} bounds the correlations over actions for node $u$ at time step $t-1+i$ and action for node $q$ at time step $t-1+j$. 
We present its proof in the Appendix \ref{apx:J-bound}.

Recall that, for $1 \le i, j \le k-1$, 
$$M^{-1}_{i, j} = D^{-1}_{i, i}\sum_{\ell \ge 0}(-J)^{\ell}_{i, j}.$$ 
With the exponential decaying bounds on matrix $J^{\ell}$, we can thus conclude Theorem \ref{thm:networked-exp-decay-tight} by following a similar procedure as in the proof of Theorem 3.1 of \cite{lin2021perturbation}. We present the details in Appendix \ref{le:final-step}.

\subsection{From Perturbation to Competitive Ratio}\label{sec:outline-competitive-ratio}

We now show how to use the result proven in the previous section to prove our competitive ratio bounds in \Cref{coro:CR-LPC}.  Our starting point is the assumption that the exponentially decaying local perturbation bound in \Cref{thm:networked-exp-decay:meta} holds for some $C_1, C_2 > 0$ and $\rho_S, \rho_T \in [0, 1)$, which is established using the proof approach outlined in the Section \ref{sec:main:CR}. 


As we discussed in Section \ref{sec:main:CR}, our proof contains two key parts: (i) we bound the per-time-step error of LPC (Lemma \ref{thm:per-step-err}); and (ii) show that the per-time-step error does not accumulate to be unbounded (\Cref{thm:alg-perf-bound-err-inject}). 

A key observation that enables the above analysis approach is that the aggregation of local per-time-step error made by each agent at $x_t^v$ can be viewed as a global per-time-step error in the joint global action $x_t$. Following this observation, we first introduce 
a global perturbation bound that focuses on the global action $x_t$ rather than the local actions $x_t^v$. Recall that $f_t$ denotes the global hitting cost (see Section \ref{sec:setting}). Define the optimization problem that solves the optimal trajectory from global state $x_{t-1}$ to $x_{t+p-1}$
\begin{align}\label{equ:global-opt-problem}
    \tilde{\psi}_t^p(y, z) = \argmin_{x_{t:t+p-1}}& \sum_{\tau = t}^{t+p-1} \left(f_\tau(x_\tau) + c_\tau(x_\tau, x_{\tau-1})\right)\nonumber\\*
    \text{ s.t. }& x_{t-1} = y, x_{t+p-1} = z,
\end{align}
and another one that solves the optimal trajectory from global state $x_{t-1}$ to the end of the game
\begin{align}\label{equ:global-opt-problem-until-end}
    \tilde{\psi}_t(y) = \argmin_{x_{t:T}}& \sum_{\tau = t}^{T} \left(f_\tau(x_\tau) + c_\tau(x_\tau, x_{\tau-1})\right)\nonumber\\*
    \text{ s.t. }& x_{t-1} = y.
\end{align}

The following global perturbation bound can be derived from Theorem 3.1 in \citet{lin2021perturbation}:

\begin{theorem}[Global Perturbation Bound]\label{thm:global-exp-decay}
Under Assumption \ref{assump:costs-and-feasible-sets}, the following perturbation bounds hold for optimization problems \eqref{equ:global-opt-problem} and \eqref{equ:global-opt-problem-until-end}:
\begin{align*}
    \norm{\tilde{\psi}_t^p(y, z)_{t_0} - \tilde{\psi}_t^p(y', z')_{t_0}} \leq{}& C_G \rho_G^{t_0 - t + 1} \norm{y - y'} + C_G \rho_G^{t+p-1-t_0} \norm{z - z'},\\
    \norm{\tilde{\psi}_t(y)_{t_0} - \tilde{\psi}_t(y')_{t_0}} \leq{}& C_G \rho_G^{t_0 - t + 1} \norm{y - y'},
\end{align*}
where $\rho_G = 1 - 2 \cdot \left(\sqrt{1 + \frac{2\ell_{T}}{\mu}} + 1\right)^{-1}$ and $C_G = \frac{2\ell_T}{\mu}$.
\end{theorem}

To make the concept of \textit{per-time-step error} rigorous, we formally define it as the distance between the actual next action picked by LPC and the clairvoyant optimal next action from previous action $x_{t-1}$ to the end of the game:

\begin{definition}[Per-step error magnitude]\label{def:per-step-err-magnitude}
At time step $t$, given the previous state $x_{t-1}$, the (decentralized) online algorithm $ALG$ picks $x_t \in D_{t}$. We define error $e_t$ as
\[e_t := \norm{x_t - \tilde{\psi}_{t} (x_{t-1})_t}.\]
\end{definition}

Using the local perturbation bound in \Cref{thm:networked-exp-decay:meta}, we show the per-time-step error of LPC decays exponentially with respect to prediction length $k$ and communication range $r$. This result is stated formally in Lemma \ref{thm:per-step-err}, and the proof can be found in Appendix \ref{apx:thm:per-step-err}.

\begin{lemma}\label{thm:per-step-err}
For $LPC$ with parameters $r$ and $k$, $e_t$ satisfies
\begin{align*}
    e_t^2 ={}& O\left(h(r)^2 \cdot \rho_S^{2r} + C_3(r)^2 \cdot \rho_T^{2k} \rho_G^{2k}\right) \cdot \norm{x_{t-1} - x_{t-1}^*}^2\nonumber\\
    &+ O\left(h(r)^2 \cdot \rho_S^{2r}\right) \sum_{\tau = t}^{t+k-1} \rho_T^{\tau - t} f_\tau(x_\tau^*)\\
    &+ O\left(C_3(r)^2 \cdot \rho_T^{2k}\right) f_{t+k-1}(x_{t+k-1}^*),
\end{align*}
where $C_3(r) := \sum_{\gamma = 0}^{r} h(\gamma) \cdot \rho_S^\gamma$.
\end{lemma}

Using the global perturbation bound in \Cref{thm:global-exp-decay}, we show $\sum_{t=1}^T\norm{x_t - x_t^*}^2$ can be upper bounded by the sum of per-time-step errors of LPC in \Cref{thm:alg-perf-bound-err-inject}. The proof can be found in Appendix \ref{apx:thm:alg-perf-bound-err-inject}.

\begin{theorem}\label{thm:alg-perf-bound-err-inject}
Let $x_0, x_1^*, x_2^*, \ldots, x_\horizonlength^*$ denote the offline optimal global trajectory and $x_0, x_1, x_2, \ldots, x_\horizonlength$ denote the trajectory of $ALG$. The trajectory of $ALG$ satisfies that
\[\sum_{t=1}^\horizonlength\norm{x_t - x_t^*}^2 \leq \frac{C_0^2}{(1 - \rho_G)^2} \sum_{t=1}^{\horizonlength} e_t^2,\]
where $C_0 := \max\{1, C_G\}$ and $C_G$ is defined in \Cref{thm:global-exp-decay}.
\end{theorem}



To understand the bound in \Cref{thm:alg-perf-bound-err-inject}, we can set all per-time-step error $e_t$ to be zero except a single time step $\tau$. We see the impact of $e_\tau$ on the total squared distance $\sum_{t=1}^T \norm{x_t - x_t^*}^2$ is up to some constant factor of $e_\tau$. This is because the impact of $e_\tau$ on $\norm{x_t - x_t^*}$ decays exponentially as $t$ increases from $\tau$ to $T$.

By substituting the per-time-step error bound in \Cref{thm:per-step-err} into \Cref{thm:alg-perf-bound-err-inject}, one can bound $\sum_{t=1}^T \norm{x_t - x_t^*}^2$ by the offline optimal cost, which can be converted to the competitive ratio bound in \Cref{coro:CR-LPC}.

\subsection{Roadmap to Generalize the Proof to Inexact Predictions}\label{sec:roadmap-to-inexact}
In this section, we present a roadmap to generalize our proof to the case where predictions of future cost functions are inexact. In the information availability model in Section \ref{sec:look-ahead-and-communication}, one can study inexact predictions by introducing additional disturbance parameters $\{\delta_t^v, w_t^v, w_t^e\}$ to the 3 types of cost functions and generalize them to $f_t(x_t^v, \delta_t^v), c_t^v(x_t^v, x_{t-1}^v, w_t^v), s_t^{e}(x_t^v, x_t^u, w_t^e)$ for all $t \in [H], v \in \mathcal{V}, e = (v, u) \in \mathcal{E}$. $\{\delta_t^v, w_t^v, w_t^e\}$ represents the disturbances in the cost functions that are hard to predict exactly. Before the decentralized online algorithm decides each local action, it receives the generalized cost functions $f_t^v(\cdot, \cdot), c_t^v(\cdot, \cdot, \cdot), s_t^e(\cdot, \cdot, \cdot)$ and noisy predictions of the true disturbance parameters $\{\delta_t^v, w_t^v, w_t^e\}$ within $k$ time steps and an $r$-hop neighborhood. In the LPC algorithm (Algorithm \ref{alg:LPC}), the optimization problem $\psi_{(t, v)}^{(k, r)}$ can then be solved with the noisy predictions of disturbance parameters. To analyze the performance of LPC in the presence of prediction errors, one can first generalize the exponentially decaying perturbation bounds in \Cref{thm:networked-exp-decay} and \Cref{thm:networked-exp-decay-tight} to include the perturbations on disturbance parameters similar to what we already did in \Cref{prop:exp-decay-networked-SOCO}. The prediction error on disturbance parameters will result in an additional additive term in the per-step error bound in \Cref{thm:per-step-err}. If one is willing to assume that the total sum of prediction errors is $O(cost(OPT))$, as in \citet{antoniadis2020online}, one can derive a competitive ratio for LPC by substituting the per-step error bound into \Cref{thm:alg-perf-bound-err-inject}. It is worth noting that the resulting competitive ratio will inevitably depend on the quality of predictions, and will converge to a limit larger than $1$ (under imperfect predictions) as the prediction horizon $k$ and the communication radius $r$ increase.

\section{Perturbation Bounds}\label{apx:perturbation-bounds}

This section provides the full proofs of the perturbation bounds stated in Section \ref{sec:main:perturb}.

\subsection{Proof of Theorem \ref{thm:networked-exp-decay}}\label{apx:networked-exp-decay}
We begin with a technical lemma. Recall that for any positive integer $m$, $\mathbb{S}^m$ denotes the set of all symmetric $m \times m$ real matrices.

\begin{lemma}\label{thm:graph-induced-band-mat-inverse}
For a graph $\mathcal{G}' = (\mathcal{V}', \mathcal{E}')$, suppose $A$ is a positive definite matrix in $\mathbb{S}^{\sum_{i\in \mathcal{V}'} p_i}$ formed by $\abs{\mathcal{V}'} \times \abs{\mathcal{V}'}$ blocks, where the $(i, j)$-th block has dimension $p_i \times p_j$, i.e., $A_{i,j} \in \mathbb{R}^{p_i\times p_j}$. Assume that $A$ is $q$-banded for an even positive integer $q$; i.e.,
\[A_{i,j} = 0, \forall d_{\mathcal{G}'}(i, j) > q/2.\]
Let $a_0$ denote the smallest eigenvalue value of $A$, and $b_0$ denote the largest eigenvalue value of $A$. Assume that $b_0 \geq a_0 > 0$. Suppose $D = diag(D_1, \ldots, D_{\abs{\mathcal{V}'}})$, where $D_i \in \mathbb{S}^{p_i}$ is positive semi-definite. Let $M = \left((A+D)^{-1}\right)_{S_R,S_C}$, where $S_R, S_C \subseteq \{1, \ldots, \abs{\mathcal{V}'} \}$. 
Then we have $\norm{M} \leq C \gamma^{\hat{d}}$, where
\[C = \frac{2}{a_0}, \gamma = \left(\frac{\sqrt{cond(A)} - 1}{\sqrt{cond(A)} + 1}\right)^{2/q}, \hat{d} = \min_{i \in S_R, j \in S_c} d_{\mathcal{G}'}(i, j).\]
Here $cond(A) = b_0/a_0$ denotes the condition number of matrix $A$.
\end{lemma}

We can show \Cref{thm:graph-induced-band-mat-inverse} using the same method as Lemma B.1 in \citet{lin2021perturbation}. We only need to note that even when the size of blocks are not identical, the $m$ th power of a $q$-banded matrix is a $qm$-banded matrix for any positive integer $m$.



With the help of \Cref{thm:graph-induced-band-mat-inverse}, we can proceed to show a local perturbation bound on a general $\mathcal{G}'$ in \Cref{prop:exp-decay-networked-SOCO}, where $\mathcal{G}'$ can be different from the network $\mathcal{G}$ of agents in Section \ref{sec:setting}. Compared with Theorem 3.1 in \citet{lin2021perturbation}, \Cref{prop:exp-decay-networked-SOCO} is more general because it considers a general network of decision variables while Theorem 3.1 in \citet{lin2021perturbation} only consider the special case of a line graph. Although \Cref{prop:exp-decay-networked-SOCO} does not consider the temporal dimension which features in the local perturbation bound defined in \Cref{thm:networked-exp-decay:meta}, we will use it to show \Cref{thm:networked-exp-decay} later by redefining the variables from two perspectives. 

\begin{theorem}\label{prop:exp-decay-networked-SOCO}
For a network $\mathcal{G}' = (\mathcal{V}', {\mathcal{E}'})$ with undirected edges, suppose that each node $v \in \mathcal{V}'$ is associated with a decision vector\footnote{We add a hat over the decision vector $\hat{x}_v$ to distinguish it with the local action $x_t^v$ and global action $x_t$ defined in Section \ref{sec:setting}. We assume $\hat{x}_v$ is a $p_v$ dimensional real vector.} $\hat{x}_v \in \mathbb{R}^{p_v}$ and a cost function $\hat{f}_v: \mathbb{R}^{p_v} \to \mathbb{R}_{\geq 0}$, and each edge $e = (u, v) \in {\mathcal{E}'}$ is associated with an edge cost $\hat{c}_e: \mathbb{R}^{p_v} \times \mathbb{R}^{p_u} \times \mathbb{R}^{q} \to \mathbb{R}_{\geq 0}$. Assume that $\hat{f}_v$ is $\mu$-strongly convex for all $v \in \mathcal{V}'$ and $\hat{c}_e$ is $\ell$-smooth for all $e \in {\mathcal{E}'}$. For some subset $S \subset \mathcal{V}'$, define
\begin{align*}
    E_0 :={}& \{(u, v) \in {\mathcal{E}'} \mid u, v \in \mathcal{V}'\setminus S\},\\
    E_1 :={}& \{(u, v) \in {\mathcal{E}'} \mid u \in \mathcal{V}'\setminus S, v \in S\}.
\end{align*}
For the disturbance vectors\footnote{We do not consider the disturbance vectors in the exponentially decaying local perturbation bounds defined in \Cref{thm:networked-exp-decay:meta}, but adding $w$ into the edge costs makes \Cref{prop:exp-decay-networked-SOCO} more general. For each edge $e$, $w_e$ is a $q$-dimensional real vector.} $w \in \mathbb{R}^{(\abs{E_0}+\abs{E_1})\times q}$ indexed by $e \in E_0 \cup E_1$ and $y \in \mathbb{R}^{\sum_{v \in S} p_v}$ indexed by $v \in S$, let $\psi(w, y)$ denote the optimal solution of the optimization problem
\[\psi(w, y) := \argmin_{x \in \mathbb{R}^{\abs{\mathcal{V}'\setminus S}\times d}} \sum_{v \in \mathcal{V}'\setminus S} \hat{f}_v(\hat{x}_v) + \sum_{(u, v) \in E_0} \hat{c}_{(u, v)}(\hat{x}_u, \hat{x}_v, w_{(u, v)}) + \sum_{(u, v) \in E_1} \hat{c}_{(u, v)}(\hat{x}_u, y_v, w_{(u, v)}).\]
Then, we have that for any vertex $u_0 \in \mathcal{V}'\setminus S$, the following inequality holds:
\[\norm{\psi(w, y)_{u_0} - \psi(w', y')_{u_0}} \leq C \left(\sum_{e \in E_0 \cup E_1} \lambda^{d_{\mathcal{G}'}(h, e)-1} \norm{w_e - w_e'} + \sum_{v \in S}\lambda^{d_{\mathcal{G}'}(h, v)-1} \norm{y_v - y_v'} \right), \forall w, y, w', y',\]
where $C \coloneqq {(2\ell)}/{\mu}$ and $\lambda \coloneqq 1 - 2 \cdot \left(\sqrt{1 + (\Delta' \ell/\mu)} + 1\right)^{-1}$. Here, $\Delta'$ denote the maximum degree of any vertex $v \in \mathcal{V}'$ in graph $\mathcal{G}'$. For $e = (u, v) \in {\mathcal{E}'}$, we define $d_\mathcal{G'}(u_0, e) := \min\{d_{\mathcal{G}'}(u_0, u), d_{\mathcal{G}'}(u_0, v)\}$.
\end{theorem}


\begin{proof}[Proof of Theorem \ref{prop:exp-decay-networked-SOCO}]
Let $e = [\pi^\top, \epsilon^\top]^\top$ be a vector where $\epsilon = \{\epsilon_v\}_{v \in S}$ for $\epsilon_v \in \mathbb{R}^{p_v}$ and
$\pi = \{\pi_e\}_{e \in E_0 \cup E_1},$
for $\pi_e \in \mathbb{R}^{q}$.
Let $\theta$ be an arbitrary real number. Define function $\hat{h}: \mathbb{R}^{\sum_{v \in \mathcal{V}'\setminus S} p_v} \times \mathbb{R}^{(\abs{E_0} + \abs{E_1}) \times q} \times \mathbb{R}^{\sum_{v \in S} p_v} \to \mathbb{R}_{\geq 0}$ as
\[\hat{h}(\hat{x}, w, y) = \sum_{v \in \mathcal{V}'\setminus S} \hat{f}_v(\hat{x}_v) + \sum_{(u, v) \in E_0} \hat{c}_{(u, v)}(\hat{x}_u, \hat{x}_v, w_{(u, v)}) + \sum_{(u, v) \in E_1} \hat{c}_{(u, v)}(\hat{x}_u, y_v, w_{(u, v)}).\]
To simplify the notation, we use $\zeta$ to denote the tuple of system parameters, i.e.,
\[\zeta := (w, y).\]

From our construction, we know that $\hat{h}$ is $\mu$-strongly convex in $x$, so we use the decomposition $\hat{h} = \hat{h}_a + \hat{h}_b$, where
\begin{align*}
    \hat{h}_a(\hat{x}, \zeta) &= \sum_{v \in \mathcal{V}'\setminus S} \frac{\mu}{2}\norm{\hat{x}_v}^2 + \sum_{(u, v) \in E_0} \hat{c}_{(u, v)}(\hat{x}_u, \hat{x}_v, w_{(u, v)}) + \sum_{(u, v) \in E_1} \hat{c}_{(u, v)}(\hat{x}_u, y_v, w_{(u, v)}),\\
    \hat{h}_b(\hat{x}, \zeta) &= \sum_{v \in \mathcal{V}'\setminus S} \left(\hat{f}_v(\hat{x}_v) - \frac{\mu}{2}\norm{\hat{x}_v}^2\right).
\end{align*}

Since $\psi(\zeta + \theta e)$ is the minimizer of convex function $\hat{h}(\cdot, \zeta + \theta e)$, we see that
\[\nabla_{\hat{x}} \hat{h}(\psi(\zeta + \theta e), \zeta + \theta e) = 0.\]
Taking the derivative with respect to $\theta$ gives that
\begin{align*}
    \nabla_{\hat{x}}^2 \hat{h}(\psi(\zeta + \theta e), \zeta + \theta e) \frac{d}{d\theta}\psi(\zeta + \theta e) ={}& - \sum_{v \in S} \nabla_{y_v} \nabla_{\hat{x}} \hat{h}(\psi(\zeta + \theta e), \zeta + \theta e) \epsilon_v\\
    &- \sum_{e \in E_1 \cup E_2}\nabla_{w_e} \nabla_{\hat{x}} \hat{h}(\psi(\zeta + \theta e), \zeta + \theta e) \pi_e.
\end{align*}
To simplify the notation, we define
\begin{align*}
    M &:= \nabla_{\hat{x}}^2 \hat{h}(\psi(\zeta + \theta e), \zeta + \theta e), \text{which is a }\abs{\mathcal{V}'\setminus S}\times \abs{\mathcal{V}'\setminus S} \text{ block matrix},\\
    R^{(v)} &:= - \nabla_{y_v} \nabla_{\hat{x}} \hat{h}(\psi(\zeta + \theta e), \zeta + \theta e), \forall v \in S, \text{which are }\abs{\mathcal{V}'\setminus S}\times 1 \text{ block matrix},\\
    K^{(e)} &:= - \nabla_{w_e} \nabla_{\hat{x}} \hat{h}(\psi(\zeta + \theta e), \zeta + \theta e), \forall e \in E_0 \cup E_1, \text{which are }\abs{\mathcal{V}'\setminus S} \times 1 \text{ block matrices},
\end{align*}
where in $M$, the block size is $p_u \times p_v, \forall (u, v) \in (\mathcal{V}'\setminus S)^2$; in $R^{(v)}$, the block size is $p_u \times p_v, \forall u \in \mathcal{V}'\setminus S$; in $K^{(e)}$, the block size is $p_u\times q, \forall u \in \mathcal{V}'\setminus S$. Hence we can write
\[\frac{d}{d\theta}\psi(\zeta + \theta e) = M^{-1} \left(\sum_{v \in S} R^{(v)} \epsilon_v + \sum_{e \in E_1 \cup E_2} K^{(e)} \pi_e\right).\]
Recall that $\{R^{(v)}\}_{v \in S}$ are $\abs{\mathcal{V}'\setminus S} \times 1$ block matrices with block size $p_u \times p_v, \forall u \in \mathcal{V}'\setminus S$; $\{K^{(e)}\}_{e \in E_0 \cup E_1}$ are $\abs{\mathcal{V}'\setminus S} \times 1$ block matrices with block size $p_u\times q, \forall u \in \mathcal{V}'\setminus S$. Let $N(v)$ denote the set of neighbors of vertex $v$ on $\mathcal{G}'$. For $R^{(v)}, v \in S$, the $(u, 1)$-th block can be non-zero only if $u \in (\mathcal{V}'\setminus S) \cap N(v)$. For $K^{(e)}, e \in E_0 \cup E_1$, the $(u, 1)$-th block can be non-zero only if $u \in e$ and $u \in \mathcal{V}'\setminus S$. Hence we see that
\begin{align*}
    \frac{d}{d\theta}\psi(\zeta + \theta e)_{u_0} = \sum_{v \in S} (M^{-1})_{u_0, (\mathcal{V}'\setminus S) \cap N(v)} R^{(v)}_{(\mathcal{V}'\setminus S) \cap N(v), 1} \epsilon_v
    + \sum_{e \in E_0 \cup E_1} (M^{-1})_{u_0, \{u \in e\mid u \in \mathcal{V}'\setminus S\}} K^{(\tau)}_{\{u \in e\mid u \in \mathcal{V}'\setminus S\}, 1} \pi_e.
\end{align*}
Since the switching costs $c_\tau(\cdot, \cdot, \cdot), \tau = 1, \ldots, p$ are $\ell$-strongly smooth, we know that the norms of
\[R^{(v)}_{(\mathcal{V}'\setminus S) \cap N(v), 1}, \text{ and } K^{(\tau)}_{\{u \in e\mid u \in \mathcal{V}'\setminus S\}, 1}\]
are all upper bounded by $\ell$.
Taking norms on both sides gives that
\begin{align}\label{equ:thm:SOCO-sensitivity:e1}
    \norm{\frac{d}{d\theta}\psi(\zeta + \theta e)_{u_0}} \leq{}& \sum_{v \in S} \ell\norm{(M^{-1})_{u_0, (\mathcal{V}'\setminus S) \cap N(v)}} \norm{\epsilon_v} + \sum_{e \in E_0 \cup E_1} \ell\norm{(M^{-1})_{u_0, \{u \in e\mid u \in \mathcal{V}'\setminus S\}}} \norm{\pi_e}.
\end{align}
Note that $M$ can be decomposed as $M = M_a + M_b$, where
\begin{align*}
    M_a &:= \nabla_{\hat{x}}^2 \hat{h}_a(\psi(\zeta + \theta e), \zeta + \theta e),\\
    M_b &:= \nabla_{\hat{x}}^2 \hat{h}_b(\psi(\zeta + \theta e), \zeta + \theta e).
\end{align*}
Since $M_a$ is block tri-diagonal and satisfies $(\mu + \Delta' \ell) I \succeq M_a \succeq \mu I$ 
, and $M_b$ is block diagonal and satisfies $M_b \succeq 0$, we obtain the following using Lemma \ref{thm:graph-induced-band-mat-inverse}:
\[\norm{(M^{-1})_{u_0, (\mathcal{V}'\setminus S) \cap N(v)}} \leq \frac{2}{\mu} \lambda^{d_{\mathcal{G}'}(u_0, v)-1}, \text{ and } \norm{(M^{-1})_{u_0, \{u \in e\mid u \in \mathcal{V}'\setminus S\}}} \leq \frac{2}{\mu} \lambda^{d_{\mathcal{G}'}(u_0, e) - 1},\]
where $\lambda :={(\sqrt{cond(M_a)} - 1)}/{(\sqrt{cond(M_a)} + 1)} = 1 - 2 \cdot \left(\sqrt{1 + (2\ell/\mu)} + 1\right)^{-1}$.

Substituting this into \eqref{equ:thm:SOCO-sensitivity:e1}, we see that
\[\norm{\frac{d}{d\theta}\psi(\zeta + \theta e)_{u_0}} \leq C\left(\sum_{v \in S} \lambda^{d_{\mathcal{G}'}(u_0, v)-1} \norm{\epsilon_v} + \sum_{e \in E_0 \cup E_1} \lambda^{d_{\mathcal{G}'}(u_0, e) - 1}\norm{\pi_e}\right),\]
where $C = (2\ell)/\mu$.

Finally, by integration we can complete the proof
\begin{align*}
    \norm{\psi(\zeta)_{u_0} - \psi(\zeta + e)_{u_0}} ={}& \norm{\int_0^1 \frac{d}{d\theta}\psi(\zeta + \theta e)_{u_0} d\theta}\\
    \leq{}& \int_0^1 \norm{\frac{d}{d\theta}\psi(\zeta + \theta e)_{u_0}} d\theta\\
    \leq{}& C\left(\sum_{v \in S} \lambda^{d_{\mathcal{G}'}(u_0, v)-1} \norm{\epsilon_v} + \sum_{e \in E_0 \cup E_1} \lambda^{d_{\mathcal{G}'}(u_0, e) - 1}\norm{\pi_e}\right).
\end{align*}
\end{proof}

Now we return to the proof of Theorem \ref{thm:networked-exp-decay}. For simplicity, we temporarily assume the individual decision points are unconstrained, i.e., $D_t^v = \mathbb{R}^n$. We discuss how to relax this assumption in Appendix \ref{apx:adding-constraints}.

We first consider the case when $\left(\{y_{t-1}^u\}, \{z_\tau^u\}\right)$ and $\left(\{y_{t-1}^u\}, \{z_\tau^u\}\right)$ only differ at one entry $y_{t-1}^u$ or $z_\tau^u$. If the difference is at $z_{\tau}^u$, by viewing each subset $\{\tau\} \times N_v^r$   for $\tau \in \{t-1, t, \dots,  t+k\}$ in the original problem as a vertex in the new graph $\mathcal{G}'$ and applying Theorem \ref{prop:exp-decay-networked-SOCO}, we obtain that
\begin{align}\label{thm:networked-exp-decay:e1}
    \norm{x_{t_0}^{v_0} - (x_{t_0}^{v_0})'} \leq C_1^0 \cdot (\rho_T^0)^{\abs{t_0 - \tau}}\norm{z_\tau^u - (z_\tau^u)'},
\end{align}
where $C_1^0 = (2\ell_T)/\mu$ and $\rho_T^0 = 1 - 2\cdot \left(\sqrt{1 + (2\ell_T/\mu)} + 1\right)^{-1}$. On the other hand, by viewing each subset $\{\tau \mid t-1 \leq \tau < t+k\} \times \{u\}$ for $u \in N_v^r$ in the original problem as a vertex in the new graph $\mathcal{G}'$ and applying Theorem \ref{prop:exp-decay-networked-SOCO}, we obtain that
\begin{align}\label{thm:networked-exp-decay:e2}
    \norm{x_{t_0}^{v_0} - (x_{t_0}^{v_0})'} \leq C_1^1 \cdot (\rho_S^0)^{\dist(u, v_0)}\norm{z_\tau^u - (z_\tau^u)'},
\end{align}
where $C_1^1 = (2\Delta \ell_S)/\mu$ and $\rho_S^0 = 1 - 2 \cdot \left(\sqrt{1 + (2 \Delta \ell_S/\mu)}\right)^{-1}$. Combining \eqref{thm:networked-exp-decay:e1} and \eqref{thm:networked-exp-decay:e2} gives that
\begin{align}\label{thm:networked-exp-decay:e3}
    \norm{x_{t_0}^{v_0} - (x_{t_0}^{v_0})'} \leq{}& \min \{C_1^0 \cdot (\rho_T^0)^{\abs{t_0 - \tau}}, C_1^1 \cdot (\rho_S^0)^{\dist(u, v_0)}\}\cdot \norm{z_\tau^u - (z_\tau^u)'}\nonumber\\
    \leq{}& \sqrt{C_1^0 \cdot C_1^1} \cdot (\rho_T^0)^{\abs{t_0 - \tau}/2} \cdot (\rho_S^0)^{\dist(u, v_0)/2} \cdot \norm{z_\tau^u - (z_\tau^u)'}\nonumber\\
    \leq{}& C_1 \cdot \rho_T^{|t_0-\tau|} \rho_S^{\dist(v_0, u)} \norm{z_\tau^u - (z_\tau^u)'}
\end{align}
when $\left(\{y_{t-1}^u\}, \{z_\tau^u\}\right)$ and $\left(\{y_{t-1}^u\}, \{z_\tau^u\}\right)$ only differ at one entry $z_{\tau}^u$ for $(\tau, u) \in \partial N_{(t, v)}^{(k, r)}$.

We can use the same method to show that when $\left(\{y_{t-1}^u\}, \{z_\tau^u\}\right)$ and $\left(\{y_{t-1}^u\}, \{z_\tau^u\}\right)$ only differ at one entry $y_{t-1}^u$ for $u \in N_v^r$, we have
\begin{align}\label{thm:networked-exp-decay:e4}
    \norm{x_{t_0}^{v_0} - (x_{t_0}^{v_0})'} \leq{} C_2 \rho_T^{t_0 - (t-1)}\rho_S^{\dist(v_0, u)} \norm{y_{t-1}^u - (y_{t-1}^u)'}.
\end{align}

In the general case where $\left(\{y_{t-1}^u\}, \{z_\tau^u\}\right)$ and $\left(\{y_{t-1}^u\}, \{z_\tau^u\}\right)$ differ not only at one entry, we can perturb the entries of parameters one at a time and apply the triangle inequality. Then, the conclusion of \Cref{thm:networked-exp-decay} follows from \eqref{thm:networked-exp-decay:e3} and \eqref{thm:networked-exp-decay:e4}.

\subsection{Proof of Theorem \ref{thm:networked-exp-decay-tight}}\label{apx:networked-exp-decay-tight}
The proof follows a four step structure outlined in \Cref{sec:outline-network-perturbation-bound}.

\subsubsection*{Step 1. Establish first order equations}\label{step1}
Given any system parameter $\zeta = (x_{t-1}^{(N_v^r)},  \{z_{\tau}^u | (\tau, u) \in \partial N_{(t, v)}^{(k, r)}\})$, we can define function $\hat{h}$ as follows:
$$\hat{h}(\hat x_{[k-1]}, \zeta) = \sum_{i = 1}^{k-1}\sum_{u \in N_v^{r-1}} f^u_{t-1+i}(\hat x_{i}^u) + \sum_{i = 1}^{k-1}\sum_{(u, u') \in \mathcal{E}(N_v^r)} s_{t-1+i}^{(u, u')}(\hat x_{i}^u, \hat x_{i}^{u'}) + \sum_{i = 1}^k \sum_{u \in N_v^{r}} c_{t-1+i}^u(\hat x_{i}^u, \hat x_{i-1}^u).$$
$\hat x_0$ coincides with $x_{t-1}$ on every node in $N_v^r.$ 
$\hat x_k$ coincides with $z_{t-1+k}$ on every node in $N_v^r.$
For $1 \le i \le k-1$, 
$\hat x_i^u$ coincides with 
$z_{t-1+i}^u$ on the boundary, i.e., $u \in \partial N_{ v}^{r}.$ 

Let perturbation vector $e = [e_0^T, e_1^T, \cdots, e_{k-1}^T, e_k^T]^T$ where $e_0, e_k \in \bR^{|N_v^r|\times n}$ and $e_i = \bR^{|\partial N_v^r|\times n}$ for $1 \le i \le k-1$.

Given $\theta \in \mathbb{R}$, $\psi(\zeta + \theta e)$ is the global minimizer of convex function $\hat{h}(\cdot, \zeta + \theta e)$, and hence we have 
$$\nabla_{\hat x_{1:k-1}^{(N_v^{r-1})}} \hat{h}(\psi(\zeta + \theta e), \zeta + \theta e) = 0.$$ 
Taking the derivative with respect to $\theta$, we establish the following set of equations: 
\begin{align}\label{eq:derivatives-theta}
\begin{split}
\nabla_{\hat x_{1:k-1}^{(N_v^{r-1})}}^2 \hat{h}(\psi(\zeta + \theta e), \zeta + \theta e) \frac{d}{d\theta} \psi(\zeta + \theta e) &= -\nabla_{\hat x_0^{(N_v^r)}}\nabla_{\hat x_{1:k-1}^{(N_v^{r-1})}} \hat{h}(\psi(\zeta + \theta e), \zeta + \theta e)e_0 \\
& -\nabla_{\hat x_k^{(N_v^r)}}\nabla_{\hat x_{1:k-1}^{(N_v^{r-1})}} \hat{h}(\psi(\zeta + \theta e), \zeta + \theta e)e_k \\
& -\sum_{\tau = 1}^{k-1} \nabla_{\hat x_{\tau}^{(\partial N_v^r)}}\nabla_{\hat x_{1:k-1}^{(N_v^{r-1})}} \hat{h}(\psi(\zeta + \theta e), \zeta + \theta e)e_{\tau}.
\end{split}
\end{align}

\noindent We adopt the following short-hand notation:
\begin{itemize}
    \item $M \coloneqq \nabla_{\hat x_{1:k-1}^{(N_v^{r-1})}}^2 \hat{h}(\psi(\zeta + \theta e), \zeta + \theta e)$, which is a hierarchical block matrix with the first level of dimension $(k-1) \times (k-1)$, the second level of dimension $|N_v^{r-1}| \times |N_v^{r-1}|$ and the third level of dimension $n \times n$.
    \item $R^{(1)} \coloneqq -\nabla_{\hat x_0^{(N_v^r)}}\nabla_{\hat x_{1:k-1}^{(N_v^{r-1})}} \hat{h}(\psi(\zeta + \theta e), \zeta + \theta e)$, which is  also a hierarchical block matrix with the first level of dimension $(k-1) \times 1$, the second level of dimension $|N_v^{r-1}| \times |N_v^{r}|$ and the third level of dimension $n \times n$. 
    \item $R^{(k-1)} \coloneqq -\nabla_{\hat x_k^{(N_v^r)}}\nabla_{\hat x_{1:k-1}^{(N_v^{r-1})}} \hat{h}(\psi(\zeta + \theta e), \zeta + \theta e)$, which is also a hierarchical block matrix with the first level of dimension $(k-1) \times 1$, the second level of dimension $|N_v^{r-1}| \times |N_v^{r}|$ and the third level of dimension $n \times n$. 
    \item $K^{(\tau)} \coloneqq -\nabla_{\hat x_{\tau}^{(\partial N_v^r)}}\nabla_{\hat x_{1:k-1}^{(N_v^{r-1})}} \hat{h}(\psi(\zeta + \theta e), \zeta + \theta e)$, which is also a hierarchical block matrix with the first level of dimension $(k-1) \times 1$. the second level of dimension $|N_v^{r-1}| \times |\partial N_v^r|$ and the third level of dimension $n \times n$. 
\end{itemize}
Using the above, we can rewrite \eqref{eq:derivatives-theta} as follows:
$$\frac{d}{d\theta} \psi(\zeta + \theta e) = M^{-1} \Bigg(R^{(1)}e_0 + R^{(k-1)}e_k + \sum_{\tau = 1}^{k-1} K^{(\tau)} y_{\tau}\Bigg).$$
Due to the structure of temporal interaction cost functions, for $R^{(1)}$(resp. $R^{(k-1)}$), only when the first level index is $1$ (resp. $k-1$), the lower level block matrix is non-zero; due to the structure of spatial interaction cost functions, for $K^{(\tau)}$, only when the first level index is $\tau$, the lower level block matrix is non-zero. Hence, for $1 \le \tau' \le k-1$, we have
\begin{align}\label{eq:time-index}
\begin{split}
    (\frac{d}{d\theta} \psi(\zeta + \theta e))_{\tau'} &= M^{-1}_{\tau', 1} R^{(1)}_1e_0 + M^{-1}_{\tau', k-1}R^{(k-1)}_{k-1}e_k + \sum_{\tau = 1}^{k-1} M^{-1}_{\tau', \tau}K^{(\tau)}_{\tau} y_{\tau},
\end{split}    
\end{align}
where the subscripts on the right hand side denote the first level index of hierarchical block matrices $M$, $R^{(1)}$, $R^{(k-1)}$ and $K^{(\tau)}$.

\subsubsection*{Step 2. Exploit the structure of matrix $M$}
We decompose $M$ to block diagonal matrix $D$ and tri-diagonal block matrix $A$ such that $M = D + A$. We denote each diagonal block in $D$ as $D_{i,i}$ for $1 \le i \le k-1$. Other blocks in $D$ are zero matrices.
\begin{equation*}
D \coloneqq \begin{bmatrix}
        \begin{matrix}
        * & 0 & \cdots & *\\
        0 & * & & 0\\
        \vdots & & \ddots & \\
        * & 0 & \cdots & * 
        \end{matrix}& & & \\
        &\begin{matrix}
        * & 0 & \cdots & *\\
        0 & * & & 0\\
        \vdots & & \ddots & \\
        * & 0 & \cdots & * 
        \end{matrix}\\
        & & \ddots& \\
        & & &\begin{matrix}
        * & 0 & \cdots & *\\
        0 & * & & 0\\
        \vdots & & \ddots & \\
        * & 0 & \cdots & * 
        \end{matrix}
        \end{bmatrix}
\end{equation*}
Each non-zero block in $A$ is a diagonal block matrix, which captures the Hessian of temporal interaction cost between consecutive time steps. Denote each block as $A_{i,j}$ for $1 \le i, j \le k-1$. 
\begin{equation*}
A \coloneqq \begin{bmatrix}
         &\begin{matrix}
        * & 0 & \cdots & 0\\
        0 & * & & 0\\
        \vdots & & \ddots & \\
        0 & 0 & \cdots & *
        \end{matrix} & & \\
        \begin{matrix}
        * & 0 & \cdots & 0\\
        0 & * & & 0\\
        \vdots & & \ddots & \\
        0 & 0 & \cdots & *
        \end{matrix}
        & &\begin{matrix}
        * & 0 & \cdots & 0\\
        0 & * & & 0\\
        \vdots & & \ddots & \\
        0 & 0 & \cdots & *
        \end{matrix}\\
        & &\ddots &\begin{matrix}
        * & 0 & \cdots & 0\\
        0 & * & & 0\\
        \vdots & & \ddots & \\
        0 & 0 & \cdots & *
        \end{matrix} \\
        & & \begin{matrix}
        * & 0 & \cdots & 0\\
        0 & * & & 0\\
        \vdots & & \ddots & \\
        0 & 0 & \cdots & *
        \end{matrix}
        &
        \end{bmatrix}
\end{equation*}

We rewrite the inverse of $M$ as follows:
$$M^{-1} = (D + A)^{-1} = D^{-1} (I + AD^{-1})^{-1} = D^{-1}P^{-1}.$$




Next we show the proof for Lemma \ref{le: H-inverse}.
\begin{proof}[Proof of Lemma \ref{le: H-inverse}]
We claim the eigenvalues of $P$ are in $\{\lambda \in \mathbb{C} | |\lambda - z| \le R \}$ for some $R \in \bR_{>0}$ and $z \in \mathbb{C}\setminus \{0\}$ such that $R < |z|.$ We first establish Lemma \ref{le: H-inverse} based on the claim and then prove the claim. 

We follow the argument as in the proof of Thm 4 in \citet{2020}. Since any eigenvalue $\lambda$ of $P$ satisfies $|\lambda - z| \le R$, $|\lambda/z - 1| \le R/|z| < 1.$ Thus, the eigenvalues of $I - (1/z)P$ lie on $\{\tilde \lambda \in \mathbb{C}: |\tilde \lambda| \le R/|z|\}$, which guarantees $\rho(I - (1/z)P) < 1.$ Therefore,
$$P^{-1} = \frac{1}{z}\Bigg(I - (I - \frac{1}{z}P)\Bigg)^{-1} = \frac{1}{z}\sum_{q \ge 0}(I - \frac{1}{z}P)^q.$$

We let $z = 1$ and $R = \frac{2\ell_T}{\mu}$ and prove the above claim by utilizing Gershgorin circle theorem for block matrices.

By Theorem 1.13.1 and Remark 1.13.2 of \citet{Tretter2008}, the following holds: Consider $\mathcal{A} = (A_{ij}) \in \mathbb{R}^{dn \times dn}$ ($d, n \ge 1$) where $A_{ij} \in \mathbb{R}^{d\times d}$ and $A_{ii}$ is symmetric. Suppose $\sigma(\cdot)$ is the spectrum of a matrix. Define set
$$G_i \coloneqq \sigma(A_{ii}) \cup \Bigg\{\cup_{k=1}^d B\Bigg(\lambda_k(A_{ii}), \sum_{j\neq i}\norm{A_{ij}}\Bigg)\Bigg\}$$
where $B(\cdot, \cdot)$ denotes a disk
$B(c, r) = \{\lambda: \norm{\lambda - c} \le r\}$ and $\lambda_k$ is the $k$-th smallest eigenvalues of $A_{ii}$.
Then,
$$\sigma(\mathcal{A}) \in \cup_{i=1}^n G_i.$$

Next, we use the above fact to find a superset of $\sigma(P)$. 
Every diagonal block of $P$ is ${I}$. Moreover, $P_{i, j} = 0$ for $|i - j| > 1$, $P_{i, i-1} = A_{i, i-1}D_{i-1, i-1}^{-1}$, $P_{i, i+1} = A_{i, i+1}D_{i+1, i+1}^{-1}$. Hence we have
\begin{align*}
\sum_{j \neq i} \norm{P_{i, j}} &\le \norm{A_{i, i-1}}\norm{D_{i-1, i-1}^{-1}} + \norm{A_{i, i+1}}\norm{D_{i+1, i+1}^{-1}} \\
&\le \frac{2\ell_T}{\mu}.
\end{align*}
The last inequality is by Assumptions \ref{assump:costs-and-feasible-sets}. Therefore, 
$G_i = B(1, \frac{2\ell_T}{\mu}).$
This implies all eigenvalues of $P$ are in $B(1, \frac{2\ell_T}{\mu}).$
\end{proof}

To further simplify the notation in the power series expansion, we define $J \coloneqq AD^{-1} = P - I$. Given any time indices $\tau'$ and $\tau$, we have 
\begin{align}\label{ineq:M-inverse}
\begin{split}
(M^{-1})_{\tau', \tau}
&= (D^{-1})_{\tau', \tau'}(P^{-1})_{\tau', \tau}  \\
&= (D^{-1})_{\tau', \tau'} \times \sum_{\ell\ge 0} (-J)^{\ell}_{\tau', \tau},
\end{split}
\end{align}
where the first equality is 
since $D^{-1}$ is a diagonal block matrix,
the second equality is due to Lemma \ref{le: H-inverse}.



\subsubsection*{Step 3: Property for general exponential-decay matrices}

This step simply requires proving \Cref{thm: l-power}.

\begin{proof}[Proof of Lemma \ref{thm: l-power}]\label{pf:l-power}
Under the assumptions, we see that
\begin{align}
\begin{split}
\sum_{q} (\frac{1}{\lambda'})^{d_{\mathcal{M}}(u, q)} \norm{(A_1A_2\cdots A_{\ell})_{u, q}}
&= \sum_{q} (\frac{1}{\lambda'})^{d_{\mathcal{M}}(u, q)} \norm{\sum_{s_1, \cdots, s_{\ell-1}} (A_1)_{u, s_1} (A_2)_{s_1, s_2} \cdots (A_{\ell})_{s_{\ell-1}, q}}\\
&\le \sum_{q} (\frac{1}{\lambda'})^{d_{\mathcal{M}}(u, q)} \sum_{s_1, \cdots, s_{\ell-1}} (C_1\lambda^{d_{\mathcal{M}}(u, s_1)}) (C_2\lambda^{d_{\mathcal{M}}(s_1, s_2)}) \cdots (C_{\ell}\lambda^{d_{\mathcal{M}}(s_{\ell-1}, q)})\\
&\le \sum_{q}\sum_{s_1, \cdots, s_{\ell-1}} \prod_{i=1}^{\ell} C_{i}(\frac{\lambda}{\lambda'})^{d_{\mathcal{M}}(u, s_1) + d_{\mathcal{M}}(s_1, s_2) + \cdots + d_{\mathcal{M}}(s_{\ell-1}, q)}\\
&\le (\tilde a)^{\ell}\prod_{i=1}^{\ell} C_{i}.
\end{split}
\end{align}

Hence, we obtain that 
$$\norm{(\prod_{i=1}^{\ell} A_i)_{u, q}} \le C'(\lambda')^{d_M(u, q)}.$$
\end{proof}

\subsubsection*{Step 4: Establish correlation decay properties of matrix $M$}

In this step, we use the property developed for general exponential-decay matrices on $M$ and derive the perturbation bound in the Theorem \ref{thm:networked-exp-decay-tight}.

\begin{lemma}\label{le:J-properties}
For $\ell \ge 1$, time index $i, j \ge 1$, $J^{\ell}$ has the following properties:
\begin{itemize}
    \item $(J^{\ell})_{i, j} = 0$ if $\ell < |i-j|$ or $\ell - |i - j|$ is odd. 
    \item $(J^{\ell})_{i, j}$ is a summation of terms $\prod_{k = 1}^{\ell}A_{j_k, i_k}D_{i_k, i_k}^{-1}$ and the number of such terms is bounded by ${\ell \choose (\ell - |i-j|)/2}$.
\end{itemize}
\end{lemma}
Note for integers $m, k \ge 1$, we define ${m \choose k/2} = 0$  if $k$ is odd. 
\begin{proof}\label{pf:J-properties}
Since $J$ is a tri-diagonal banded matrix, $J^{\ell}_{i, j} = 0$ for $\ell < |i - j|$. We prove the rest of properties of $J$ by induction on $\ell$. When $\ell = 1$, 
$$J_{i, i} = 0, \quad J_{i, i-1} = A_{i, i-1}D_{i-1, i-1}^{-1}, \quad J_{i, i+1} = A_{i, i+1}D_{i+1, i+1}^{-1}.$$
Lemma \ref{le:J-properties} holds for the base case. Suppose Lemma \ref{le:J-properties} holds for $J^{q}$ for $q \le \ell-1$. Let $q = \ell$, then 
\begin{align*}
J^{\ell}_{i, j} &= \sum_{k}J^{\ell-1}_{i, k}J_{k, j} = J^{\ell-1}_{i, j-1}A_{j-1, j}D^{-1}_{j, j} + J^{\ell-1}_{i, j+1}A_{j+1, j}D_{j, j}^{-1}.
\end{align*}
By induction hypothesis, $J^{\ell-1}_{i, j}$ is a summation of terms $\prod_{k = 1}^{\ell-1}A_{j_k, i_k}D_{i_k, i_k}^{-1}$. Moreover, the number of such terms is bounded by ${\ell-1 \choose (\ell-1-|i-j-1|)/2} + {\ell-1 \choose (\ell-1-|i-j+1|)/2}$. Next we will show 
${\ell-1 \choose (\ell-1-|i-j-1|)/2} + {\ell-1 \choose (\ell-1-|i-j+1|)/2} = {\ell \choose (\ell - |i-j|)/2}$ case by case. 

\textbf{Case 1: }$\ell - |i-j|$ is odd.

If $\ell - |i-j|$ is odd, then $\ell - 1 - |i-j -1|$ and $\ell - 1 - |i-j + 1|$ are both odd. Under this case, 
$${\ell-1 \choose (\ell-1-|i-j-1|)/2} + {\ell-1 \choose (\ell-1-|i-j+1|)/2} = 0, $$
which is equal to
${\ell \choose (\ell-|i-j|)/2}.$

\textbf{Case 2: }$\ell - |i-j|$ is even and $i = j$. Under this case, 
we have
$${\ell-1 \choose (\ell-1-|i-j-1|)/2} + {\ell-1 \choose (\ell-1-|i-j+1|)/2} = {\ell-1 \choose \ell/2-1} + {\ell-1 \choose \ell/2-1}.$$
Since $\ell$ is even, ${\ell-1 \choose \ell/2-1} + {\ell-1 \choose \ell/2-1} = {\ell \choose \ell/2} = {\ell \choose (\ell - |i-j|)/2}.$

\textbf{Case 3: }$\ell - |i-j|$ is even and $i \neq j$.

If $\ell - |i-j|$ is even, then $\ell - 1 - |i-j -1|$ and $\ell - 1 - |i-j + 1|$ are both even. We denote $(\ell - |i-j|)/2$ as $k_0$. By triangle inequality,  $(\ell-1-|i-j-1|)/2$ and $(\ell-1-|i-j+1|)/2$ are
in $\{k_0-1, k_0\}$. Since $i \neq j$, 
$${\ell-1 \choose (\ell-1-|i-j-1|)/2} + {\ell-1 \choose (\ell-1-|i-j+1|)/2} = {\ell-1 \choose k_0 - 1} + {\ell-1 \choose k_0}, $$
which sums to ${\ell \choose k_0}$ by Pascal's triangle. 

\end{proof}

Next we present the proof of Lemma \ref{le:J-bound}. 

\begin{proof}[Proof of Lemma \ref{le:J-bound}]\label{apx:J-bound}
By Lemma \ref{le:J-properties}, $(J^{\ell})_{i, j}$ is a summation of terms $\prod_{k = 1}^{\ell}A_{j_k, i_k}D_{i_k, i_k}^{-1}$ and the number of such terms is bounded by ${\ell \choose (\ell - |i-j|)/2}$.

Define $B_{k} \coloneqq A_{j_k, i_k}D_{i_k, i_k}^{-1}$. Recall $A_{j_k, i_k}$ is a diagonal matrix and $D_{i_k, i_k}$ is a graph-induced banded matrix. 

$$\norm{(B_k)_{u, q}} = \norm{(A_{j_k, i_k}D_{i_k, i_k}^{-1})_{u, q}} = \norm{(A_{j_k, i_k})_{u,u}(D_{i_k, i_k}^{-1})_{u, q}} \le \ell_T\norm{(D_{i_k, i_k}^{-1})_{u, v}} \le \frac{2\ell_T}{\mu}\gamma_S^{\dist(u, v)}.$$
where the last inequality is by using Lemma
\ref{thm:graph-induced-band-mat-inverse} on $D_{i_k, i_k}$.

Under the condition $b < \infty$, we can use Lemma \ref{thm: l-power} to obtain the following bound, 
$$\norm{(\prod_{k = 1}^{\ell}A_{j_k, i_k}D_{i_k}^{-1})_{u, v}} \le (b\frac{2\ell_T}{\mu})^{\ell} (\gamma_S')^{\dist(u, v)}.$$

Since the number of such terms is bounded by ${\ell \choose (\ell - |i-j|)/2}$, we have 
$$\norm{((J^{\ell})_{i, j})_{u, q}} \le {\ell \choose (\ell - |i-j|)/2}(b\frac{2\ell_T}{\mu})^{\ell} (\gamma_S')^{\dist(u, v)}.$$ 
\end{proof}

\begin{lemma}\label{le:final-step}
Given $1 \le \tau', \tau \le k-1$, $y \in \bR^{|\partial N_v^r| \times n}$ and $v_0 \in N_v^{r-1}$, we have
$$\norm{\Bigg((M)^{-1}_{\tau', \tau}K_{\tau}^{(\tau)} y\Bigg)_{v_0}} \le C_1\rho_T^{|\tau' - \tau|}\sum_{u \in \partial N_v^r}\rho_S^{\dist(v_0, u) - 1}\norm{y_{u}},$$
and for $i \in \{1, k-1\}$, $e \in \bR^{|N_v^r|\times n}$,
$$\norm{\Bigg((M^{-1})_{\tau', i})R^{(i)}_i e\Bigg)_{v_0}} 
\le C_2\rho_T^{|\tau' - i| + 1}\sum_{u \in N_v^r} \rho_S^{\dist( v_0, u)}\norm{e_{u}},$$
where $\rho_T = \frac{4\tilde a\ell_T}{\mu}$ and $\rho_S =  (1+b_1+b_2)\gamma_S.$ We let $C_1 = C_2 = \max\{\frac{a^2}{2\tilde a(1 - 4\tilde a\ell_T/\mu)}, \frac{2a^2\Delta \ell_S/\mu}{\gamma_S(1 + b_1 + b_2)(1 - 4\tilde a\ell_T/\mu)}\}$. 
\end{lemma}
\begin{proof}\label{pf:M-vector-bound}
Given $1 \le \tau, \tau' \le k-1$ and $v_0 \in N_v^{r-1}$, since $M^{-1} = D^{-1}\sum_{\ell \ge 0}(-J)^{\ell}$, we have 
\begin{align}\label{eq:19}
\norm{\Bigg((M)^{-1}_{\tau', \tau}K_{\tau}^{(\tau)} y\Bigg)_{v_0}} 
&= \norm{\Bigg(D^{-1}_{\tau', \tau'}\sum_{\ell \ge 0}(-J)^{\ell}_{\tau', \tau}K_{\tau}^{(\tau)} y\Bigg)_{v_0}}.
\end{align}

With slight abuse of notation, we use $K$ to denote $K_{\tau}^{(\tau)}$, and $Q^{-1}$ to denote $D^{-1}_{\tau', \tau'}$ in this proof from now.
We can rewrite the right hand side of \eqref{eq:19} using the new notation as follows:
\begin{align}\label{ineq: M-vector-bound}
\begin{split}
\norm{\Bigg(Q^{-1}\sum_{\ell \ge 0}(-J)^{\ell}_{\tau', \tau}K y\Bigg)_{v_0}}
&\le \sum_{\ell \ge 0}\norm{\Bigg(Q^{-1}(-J)^{\ell}_{\tau', \tau}K y\Bigg)_{v_0}}\\
&= \sum_{\ell \ge 0}\norm{\sum_{q \in N_v^{r-1}}\Bigg(Q^{-1} (-J)^{\ell}_{\tau', \tau}\Bigg)_{v_0, q} (Ky)_q}\\
&\le \sum_{\ell \ge 0}\sum_{q \in N_v^{r-1}} \norm{\Bigg(Q^{-1} (-J)^{\ell}_{\tau', \tau}\Bigg)_{v_0, q}} \norm{(Ky)_q}.
\end{split}
\end{align}

For a given $q \in N_v^{r-1}$ and $y \in \bR^{|\partial N_v^r|d}$, 
$$\norm{(Ky)_{q}} = \norm{\sum_{u \in \partial N_v^r}K_{q, u}y_{u}} = \norm{\sum_{u \in \partial N_v^r \cap N_{q}^1}K_{q, u}y_{u}}.$$
where the last equality is since spacial interaction costs are only among neighboring nodes.

For a given $u \in \partial N_v^r$, since the spacial interaction cost for each edge is $\ell_S$ smooth, 
$$\norm{K_{q, u}y_u} \le \norm{K_{q, u}}\norm{y_u} \le \ell_S\norm{y_u},$$
which gives
$$\norm{(Ky)_{q}} \le \sum_{u \in \partial N_v^r \cap N_{q}^1}\ell_S\norm{y_u}.$$
Therefore, 
\begin{align}
\begin{split}
\norm{\Bigg(Q^{-1}\sum_{\ell \ge 0}(-J)^{\ell}_{\tau', \tau}K y\Bigg)_{v_0}}
&\le \ell_S\sum_{\ell \ge 0}\sum_{q \in N_v^{r-1}} \norm{\Bigg(Q^{-1} (-J)^{\ell}_{\tau', \tau}\Bigg)_{v_0, q}} \sum_{u \in \partial N_v^r \cap N_{q}^1}\norm{y_u}.
\end{split}
\end{align}

By Lemma \ref{le:J-bound}, $(-J)^{\ell}_{\tau', \tau}$ satisfies the following exponential decay properties: for any $u, q \in N_v^{r-1}$,
$$\norm{((J^{\ell})_{\tau', \tau})_{u, q}} \le {\ell \choose (\ell - |\tau' - \tau|)/2}(\tilde a\frac{2\ell_T}{\mu})^{\ell} (\gamma_S')^{\dist(u, q)},$$
where we choose $\delta = b_1\cdot\gamma_S$, $\gamma_S' = (1 + b_1)\gamma_S$ and $\tilde a = \sum_{\gamma \ge 0}(\frac{1}{1 + b_1})^{\gamma} h(\gamma)$.

Moreover, $Q^{-1}$ (which denotes $D_{\tau', \tau'}^{-1}$) is the inverse of a graph-induced banded matrix. $Q^{-1}$ satisfies: for any $u, q \in N_v^{r-1}$,
$$\norm{(Q^{-1})_{u, q}} \le \frac{2}{\mu}\gamma_S^{\dist(u, q)} <  \frac{2}{\mu}(\gamma_S')^{\dist(u, q)},$$
where the first inequality is again by using Lemma
\ref{thm:graph-induced-band-mat-inverse} on $D_{\tau', \tau'}$.

Applying Lemma \ref{thm: l-power} on $Q^{-1}$ and $\norm{((J^{\ell})_{\tau', \tau})}$, we have for any $u, q \in N_v^{r-1}$, and $\ell \ge 1$, 

$$\norm{\Bigg(Q^{-1} (-J)^{\ell}_{\tau', \tau}\Bigg)_{u, q}}  \le a^2\frac{2}{\mu}{\ell \choose (\ell - |\tau' - \tau|)/2}(\tilde a\frac{2\ell_T}{\mu})^{\ell} (\lambda')^{\dist(u, q)},$$
where $\lambda' \coloneqq  \gamma_S' + b_2\cdot \gamma_S < 1$ and $a \coloneqq  \sum_{\gamma \ge 0}(\frac{1+b_1}{1 + b_1 + b_2})^{\gamma} h(\gamma).$ 
Note that $J^{0} \coloneqq I$, it is straightforward to verify that the above inequality holds when $\ell = 0$. 

With the exponential decay properties of $Q^{-1} (-J)^{\ell}_{\tau', \tau}$, we have
\begin{align}
\begin{split}
\norm{\Bigg(Q^{-1}\sum_{\ell \ge 0}(-J)^{\ell}_{\tau', \tau}K y\Bigg)_{v_0}}
&\le \ell_Sa^2\frac{2}{\mu}\sum_{\ell \ge 0}{\ell \choose (\ell - |\tau' -\tau|)/2}(\tilde a\frac{2\ell_T}{\mu})^{\ell}\sum_{q \in N_v^{r-1}} (\lambda')^{\dist(v_0, q)} \sum_{u \in \partial N_v^r \cup N_q^1} \norm{y_u}\\
&\le \ell_Sa^2\frac{2}{\mu}\sum_{\ell \ge |\tau' - \tau|}{\ell \choose (\ell - |\tau' - \tau|)/2}(\tilde a\frac{2\ell_T}{\mu})^{\ell} \sum_{u \in \partial N_v^r }\Delta(\lambda')^{\dist(v_0, u)-1} \norm{y_u}\\
&\le \Delta \ell_Sa^2\frac{2}{\mu}\sum_{\ell \ge |\tau' - \tau|}(\frac{4\tilde a\ell_T}{\mu})^{\ell} \sum_{u \in \partial N_v^r }(\lambda')^{\dist(v_0, u)-1} \norm{y_u}\\
&\le \frac{2\Delta \ell_Sa^2}{\mu - 4\tilde a\ell_T}(\frac{4\tilde a\ell_T}{\mu})^{|\tau' - \tau|} \sum_{u \in \partial N_v^r }(\lambda')^{\dist(v_0, u)-1} \norm{y_u}\\
&= \frac{2\Delta \ell_Sa^2}{\lambda'(\mu - 4\tilde a\ell_T)}(\frac{4\tilde a\ell_T}{\mu})^{|\tau' - \tau|} \sum_{u \in \partial N_v^r }(\lambda')^{\dist(v_0, u)} \norm{y_u}.\\
\end{split}
\end{align}
The third inequality uses ${\ell \choose (\ell - |\tau' - \tau|)/2} \le 2^{\ell}$, which can be proved using the following version of Stirling's approximation:
For all $n \ge 1$, $e$ denotes the natural number,
$$\sqrt{2\pi n} (n/e)^{n} e^{1/(12n+1)} < n! < \sqrt{2\pi n} (n/e)^{n} e^{1/(12n)}.$$

Similarly, consider $\norm{((M^{-1})_{\tau', i})R^{(i)}_i e)_{v_0}}$ for $i \in \{1, k-1\}$. With slight abuse of notation, in this proof, we use $R$ to denote $R_{i}^{(i)}$ and use the notation $Q^{-1}$ to denote $D^{-1}_{\tau', \tau'}.$ Following the same steps as before, we have  
\begin{align}
\begin{split}
\norm{\Bigg((M^{-1})_{\tau', i})R^{(i)}_i e\Bigg)_{v_0}} 
&\le \sum_{\ell \ge 0} \sum_{q \in N_v^{r}} \norm{\Bigg(Q^{-1} (-J)^{\ell}_{\tau', i}\Bigg)_{v_0, q}}\norm{(Re)_q}.
\end{split}
\end{align}

Since temporal interactions occurs for the same node under consecutive time steps,  $R$ is a diagonal block matrix. Hence, 
$$\norm{(R e)_{q}} = \norm{R_{q, q} e_q} \le  \ell_T\norm{e_q}.$$ 
Moreover, using the exponential decay properties of $Q^{-1} (-J)^{\ell}_{\tau', i}$, we have for $u, q \in N_v^{r-1}$,  
$$\norm{\Bigg(Q^{-1} (-J)^{\ell}_{\tau', i}\Bigg)_{u, q}} \le a^2\frac{2}{\mu}{\ell \choose (\ell - |\tau' - i|)/2}(\tilde a\frac{2\ell_T}{\mu})^{\ell} (\lambda')^{\dist(u, q)}.$$

Therefore, 
\begin{align}
\begin{split}
\norm{((M^{-1})_{\tau', i})R^{(i)}_i e)_{v_0}} 
&\le \sum_{\ell \ge 0} \sum_{q \in N_v^{r}} a^2\frac{2}{\mu}{\ell \choose (\ell - |\tau' - i|)/2}(\tilde a\frac{2\ell_T}{\mu})^{\ell} (\lambda')^{\dist(v_0, q)}\ell_T\norm{e_q}\\
&\le \sum_{\ell \ge |\tau' - i|} \sum_{q \in N_v^r} a^2\frac{2}{\mu}{\ell \choose (\ell - |\tau' - i|)/2}(\tilde a\frac{2\ell_T}{\mu})^{\ell} (\lambda')^{\dist(v_0, q)}\ell_T\norm{e_q}\\
&\le \frac{2\ell_T a^2}{\mu}\sum_{\ell \ge |\tau' - i|}(\frac{4\tilde a \ell_T}{\mu})^{\ell}  \sum_{q \in N_v^r} {(\lambda')}^{\dist(v_0, q)} \norm{e_q}\\
&\le \frac{2\ell_T a^2}{\mu - 4\tilde a\ell_T} (\frac{4\tilde a\ell_T}{\mu})^{|\tau' - i|} \sum_{q \in N_v^r} (\lambda')^{\dist(v_0, q)} \norm{e_q}\\
&= \frac{a^2 \mu}{2\tilde a(\mu - 4\tilde a\ell_T)} (\frac{4\tilde a\ell_T}{\mu})^{|\tau' - i| + 1} \sum_{q \in N_v^r} (\lambda')^{\dist(v_0, q)} \norm{e_q}.\\
\end{split}
\end{align}
\end{proof}
Given time index $1 \le \tau' \le k-1$, node $v_0 \in N_v^{r-1}$, and perturbation vector $e = (e_0, e_1, \cdots, e_k)$, 
\begin{align*}
\norm{(\frac{d}{d\theta} \psi(\zeta + \theta e))_{\tau', v_0}} &\le \norm{\Bigg(M^{-1}_{\tau', 1} R^{(1)}_1e_0\Bigg)_{v_0}} + \norm{\Bigg(M^{-1}_{\tau', k-1}R^{(k-1)}_{k-1}e_k\Bigg)_{v_0}} + \sum_{\tau = 1}^{k-1} \norm{\Bigg(M^{-1}_{\tau', \tau}K^{(\tau)}_{\tau} e_{\tau}\Bigg)_{v_0}}\\
&\le \frac{a^2 \mu}{2\tilde a(\mu - 4\tilde a\ell_T)}\Bigg[\rho_T^{\tau'}\sum_{q \in N_v^r} \rho_S^{\dist(v_0, q)}\norm{(e_{0})_q} + \rho_T^{k- \tau'}\sum_{q \in N_v^r} \rho_S^{\dist(v_0, q)}\norm{(e_{k})_q}\Bigg] \\
&+ \sum_{\tau = 1}^{k-1} \frac{2\Delta \ell_S a^2}{\lambda'(\mu - 4\tilde a\ell_T)}\rho_T^{|\tau' - \tau|} \sum_{u \in \partial N_v^r}(\rho_S)^{\dist(v_0, u)}\norm{(e_{\tau})_u}
\end{align*}
where $\rho_T = \frac{4\tilde a\ell_T}{\mu}$ and $\rho_S = \lambda' = (1+b_1+b_2)\gamma_S.$ We let $C = \max\{\frac{a^2}{2\tilde a(1 - 4\tilde a\ell_T/\mu)}, \frac{2a^2\Delta \ell_S/\mu}{\gamma_S(1 + b_1 + b_2)(1 - 4\tilde a\ell_T/\mu)}\}$. 
Under the condition $\mu \ge \max\{8\tilde a\ell_T, \Delta \ell_S (b_1 + b_2)/4\}$, $\rho_T < 1$ and $\rho_S  < 1$. 

Then, 
\begin{align*}
    &\norm{(\frac{d}{d\theta} \psi(\zeta + \theta e))_{\tau', v_0}}\\
\le{}& C\Bigg[\rho_T^{\tau'}\sum_{q \in N_v^r} \rho_S^{\dist(v_0, q)}\norm{(e_{0})_q} + \rho_T^{k- \tau'}\sum_{q \in N_v^r} \rho_S^{\dist(v_0, q)}\norm{(e_{k})_q}+ \sum_{\tau = 1}^{k-1} \rho_T^{|\tau' - \tau|} \sum_{u \in \partial N_v^r}(\rho_S)^{\dist(v_0, u)}\norm{(e_{\tau})_u}\Bigg].
\end{align*}
Finally, let $\zeta = \{y_{t-1}^u, z_\tau^u | (\tau, u) \in \partial N_{(v, t)}^{(k, r)}\}$ and $e = \{(y_{t-1}^u)' - y_{t-1}^u, (z_\tau^u)' - z_\tau^u\}$. By integration, 
\begin{align*}
\norm{\psi_{(t, v)}^{(k, r)}\left(\{y_{t-1}^u\}, \{z_\tau^u\}\right)_{(t_0, v_0)} - (\psi_{(t, v)}^{(k, r)}\left(\{(y_{t-1}^u)'\}, \{(z_\tau^u)'\}\right)_{(t_0, v_0)}}
&\le \int_{0}^1 \norm{(\frac{d}{d\theta} \psi(\zeta + \theta e))_{t_0, v_0}} d\theta,
\end{align*}
which is bounded by 
$$C\sum_{u \in N_v^r} \rho_T^{t_0 - (t-1)}\rho_S^{\dist(v_0, u)} \norm{y_{t-1}^u - (y_{t-1}^u)'} + C\sum_{(u, \tau) \in \partial N_{(t, v)}^{(k,r)} }  \rho_T^{|t_0-\tau|} \rho_S^{\dist(v_0, u)}
\norm{z_\tau^u - (z_\tau^u)'}.$$

\subsection{Adding Constraints to Perturbation Bounds}\label{apx:adding-constraints}
Recall that in Appendix \ref{apx:networked-exp-decay} and \ref{apx:networked-exp-decay-tight}, we showed \Cref{thm:networked-exp-decay} and \Cref{thm:networked-exp-decay-tight} under the assumption that the individual decisions are unconstrained to simplify the analysis. In this section, we present a general way to relax this assumption by incorporating logarithm barrier functions, which also applies for \Cref{thm:global-exp-decay}. 

Recall that in Assumption \ref{assump:costs-and-feasible-sets}, we assume that $D_t^v$ is convex with a non-empty interior, and can be expressed as
\[D_t^v := \{x_t^v \in \mathbb{R}^n \mid (g_t^v)_i(x_t^v) \leq 0, \forall 1 \leq i \leq m_t^v\},\]
where the $i$ th constraint $(g_t^v)_i: \mathbb{R}^n \to \mathbb{R}$ is a convex function in $C^2$. For any time-vertex pair $(\tau, v)$, we can approximate the individual constraints
\[(g_\tau^v)_i(x_\tau^v) \leq 0, \forall 1 \leq i \leq m_\tau^v,\]
by adding the \textit{logarithmic barrier function} $- \mu \sum_{i=1}^{m_\tau^v} \ln{(- (g_\tau^v)_i(x_\tau^v))}$ to the original node cost function $f_\tau^v$. Here, parameter $\mu$ is a positive real number that controls how ``good'' the barrier function approximates the indicator function
\begin{align*}
    \mathbf{I}_{D_\tau^v}(x_\tau^v) = \begin{cases}
    0 & \text{ if }(g_\tau^v)_i(x_\tau^v) \leq 0, \forall 1 \leq i \leq m_\tau^v,\\
    +\infty & \text{ otherwise.}
    \end{cases}
\end{align*}
The approximation improves as parameter $\mu$ becomes closer to $0$. Thus, the new node cost function will be
\[B_\tau^v(x_\tau^v; \mu) := f_\tau^v (x_\tau^v) - \mu \sum_{i=1}^{m_\tau^v} \ln{(- (g_\tau^v)_i(x_\tau^v))}.\]
As an extension of the original notation, we use $\psi_{(t, v)}^{(k, r)}(\{y_{t-1}^u\}, \{z_\tau^u\}; \mu)$ denote the optimal solution of the following optimization problem
\begin{align*}
    \argmin_{\{x_\tau^u\mid (\tau, v) \in N_{(t, v)}^{(k-1, r-1)}\}}& \sum_{\tau = t}^{t+k-1} \left(\sum_{u \in N_v^r} B_\tau^u(x_\tau^u; \mu) + \sum_{u \in N_v^r} c_\tau^u(x_\tau^u, x_{\tau-1}^u) + \sum_{(u, q) \in {\mathcal{E}}(N_v^r)} g_t^{(u, q)}(x_t^u, x_t^q)\right)\nonumber\\*
     \text{ s.t. }& x_{t-1}^u = y_{t-1}^u, \forall u \in N_v^r,\\*
     & x_\tau^u = z_\tau^u, \forall (\tau, u) \in \partial N_{(t, v)}^{(k, r)}.\nonumber
\end{align*}
Compared with $\psi_{(t, v)}^{(k, r)}(\{y_{t-1}^u\}, \{z_\tau^u\})$ defined in Section \ref{sec:main:alg}, the constraints $x_\tau^u \in D_\tau^u$ are removed and the node costs $f_\tau^u(x_\tau^u)$ are replaced with $B_\tau^u(x_\tau; \mu)$.

A key observation we need to point out is that the perturbation bounds we have shown in Appendix \ref{apx:networked-exp-decay} and \ref{apx:networked-exp-decay-tight} do not depend on the smoothness constant $\ell_f$ of node cost functions. That means the perturbation bound
\begin{align*}
    &\norm{\psi_{(t, v)}^{(k, r)}\left(\{y_{t-1}^u\}, \{z_\tau^u\}; \mu\right)_{(t_0, v_0)} - \psi_{(t, v)}^{(k, r)}\left(\{(y_{t-1}^u)'\}, \{(z_\tau^u)'\}; \mu\right)_{(t_0, v_0)}}\\* \leq{}& C_1\sum_{(u, \tau) \in \partial N_{(t, v)}^{(k,r)} }  \rho_T^{|t_0-\tau|} \rho_S^{\dist(v_0, u)}
\norm{z_\tau^u - (z_\tau^u)'} + C_2\sum_{u \in N_v^r} \rho_T^{t_0 - (t-1)}\rho_S^{\dist(v_0, u)} \norm{y_{t-1}^u - (y_{t-1}^u)'}
\end{align*}
holds for arbitrary $\mu$, where $C_1, C_2, \rho_S, \rho_T$ are specified in \Cref{thm:networked-exp-decay} or \Cref{thm:networked-exp-decay-tight} and are independent of $\mu$. Theorem 3.10 in \citet{forsgren2002interior} guarantees that $\psi_{(t, v)}^{(k, r)}(\{y_{t-1}^u\}, \{z_\tau^u\}; \mu_k)$ converge to $\psi_{(t, v)}^{(k, r)}(\{y_{t-1}^u\}, \{z_\tau^u\})$ for any positive sequence $\{\mu_k\}_{k=1}^\infty$ that tends to zero. Thus the above perturbation bound also holds for $\psi_{(t, v)}^{(k, r)}(\{y_{t-1}^u\}, \{z_\tau^u\})$ which includes the constraints on individual decisions.

Note that the argument we present in this section also works for \Cref{thm:global-exp-decay}.

\section{Competitive Bounds}

This appendix includes the proofs of the competitive bounds presented in \Cref{sec:main:CR}.

\subsection{Proof of Theorem \ref{thm:alg-perf-bound-err-inject}}\label{apx:thm:alg-perf-bound-err-inject}
We first derive an upper bound on the distance between $x_t$ and $x_t^*$.

Note that for any time step $t$, we have
\begin{align}\label{thm:oracle-perf-bound:e0}
    \norm{x_{t} - \tilde{\psi}_t\left(x_{t-1}\right)_{t}} \leq e_t.
\end{align}
Thus we see that
\begin{subequations}\label{thm:oracle-perf-bound:e1}
\begin{align}
    \norm{x_t - x_t^*}
    ={}& \norm{x_t - \tilde{\psi}_1(x_0)_{t}}\nonumber{}\\
    \leq{}& \norm{x_t - \tilde{\psi}_{t}(x_{t-1})_{t}} + \sum_{i=1}^{t-1} \norm{\tilde{\psi}_{t-i+1}(x_{t-i})_{t} - \tilde{\psi}_{t-i}(x_{t-i-1})_{t}}\nonumber\\
    \leq{}& \norm{x_t - \tilde{\psi}_{t}(x_{t-1})_{t}} + \sum_{i = 1}^{t-1} C_G \rho_{G}^{i} \norm{x_{t-i} - \tilde{\psi}_{t-i}(x_{t-i-1})_{t-i}}\label{thm:oracle-perf-bound:e1:s1}\\
    \leq{}& \sum_{i = 0}^{t-1} C_0 \rho_G^{i} \norm{x_{t-i} - \tilde{\psi}_{t-i}(x_{t-i-1})_{t-i}}\label{thm:oracle-perf-bound:e1:s2}\\
    \leq{}& \sum_{i=1}^{t} C_0 \rho_G^{t-i} e_i,\label{thm:oracle-perf-bound:e1:s3}
\end{align}
\end{subequations}
where in \eqref{thm:oracle-perf-bound:e1:s1}, we used Theorem \ref{thm:global-exp-decay} and the fact that $\tilde{\psi}_{t-i}(x_{t-i-1})_{t}$ can be written as
\[\tilde{\psi}_{t-i}(x_{t-i-1})_{t} = \tilde{\psi}_{t-i+1}\left(\tilde{\psi}_{t-i}(x_{t-i-1})_{t-i}\right)_{t}.\]
We also used $C_0 = \max\{1, C_G\}$ in \eqref{thm:oracle-perf-bound:e1:s2} and \eqref{thm:oracle-perf-bound:e0} in \eqref{thm:oracle-perf-bound:e1:s3}.

By \eqref{thm:oracle-perf-bound:e1} and the Cauchy-Schwarz Inequality, we see that
\[\norm{x_t - x_t^*}^2 \leq C_0^2 \left(\sum_{i=1}^{t} \rho_G^{t-i} e_i\right)^2 \leq C_0^2 \left(\sum_{i=1}^{t} \rho_G^{t-i}\right)\cdot \left(\sum_{i=1}^{t} \rho_G^{t-i} e_i^2\right) \leq \frac{C_0^2}{1 - \rho_G} \cdot \left(\sum_{i=1}^{t} \rho_G^{t-i} e_i^2\right).\]
Summing up over $t$ gives that
\[\sum_{t=1}^{\horizonlength} \norm{x_t - x_t^*}^2 \leq \frac{C_0^2}{(1 - \rho_G)^2}\cdot \sum_{t=1}^{\horizonlength} e_t^2.\]

\subsection{Proof of Lemma \ref{thm:per-step-err}}\label{apx:thm:per-step-err}
In this section, we show \Cref{thm:per-step-err} holds with following specific constants:
\begin{align}\label{thm:per-step-err:e0}
    e_{t}^2 :={}& \norm{x_t - x_{t\mid t-1}^*}^2\nonumber\\
    \leq{}& 4 C_1^2 C_0^2 \left(\frac{h(r)^2 \rho_G^2}{(1 - \rho_T)(1 - \rho_G^2 \rho_T)}\cdot \rho_S^{2r} + C_3(r)^2 \cdot \rho_T^{2(k-1)} \cdot \rho_G^{2k}\right) \norm{x_{t-1} - x_{t-1}^*}^2\nonumber\\
    &+ \frac{8 C_1^2}{\mu} \left(\frac{ h(r)^2}{1 - \rho_T} \cdot \rho_S^{2r} \sum_{\tau = t}^{t+k-1} \rho_T^{\tau - t} f_\tau(x_\tau^*) + C_3(r)^2 \cdot \rho_T^{2(k-1)} f_{t+k-1}(x_{t+k-1}^*)\right)
\end{align}

Note that, by the principle of optimality, we have
\begin{align*}
    x_t^v &= \psi_{(t, v)}^{(k, r)}\left(\{x_{t-1}^u\}, \{\minimizer_\tau^u\}\right)_{(t, v)},\\
    (x_{t\mid t-1}^v)^* &= \psi_{(t, v)}^{(k, r)}\left(\{x_{t-1}^u\}, \{(x_{\tau\mid t-1}^u)^*\}\right)_{(t, v)}.
\end{align*}
Recall that we define the quantity
$C_3(r) := \sum_{\gamma = 0}^{r} h(\gamma) \cdot \rho_S^\gamma $
to simplify the notation. 

Since the exponentially decaying local perturbation bound holds in \Cref{thm:networked-exp-decay:meta}, we see that
\begin{align}
    \norm{x_t^v - (x_{t \mid t-1}^v)^*}
    \leq{}& C_1 \rho_S^{r} \sum_{\tau = t}^{t+k-1} \rho_T^{\tau - t} \sum_{u \in \partial N_{v}^{r} } \norm{(x_{\tau\mid t-1}^u)^* - \minimizer_\tau^u}\nonumber\\
    &+ C_1\rho_T^{k-1} \sum_{u \in N_v^r} \rho_S^{d_\mathcal{G}(u, v)} \norm{(x_{t+k-1\mid t-1}^u)^* - \minimizer_{t+k-1}^u} ,
\end{align}
which implies that
\begin{subequations}\label{thm:per-step-err:e1}
\begin{align}
    \norm{x_t^v - (x_{t \mid t-1}^v)^*}^2
    \leq{}& 2C_1^2 \rho_S^{2r} \left(\sum_{\tau = t}^{t+k-1} \rho_T^{\tau - t} \sum_{u \in \partial N_{v}^{r} } \norm{(x_{\tau\mid t-1}^u)^* - \minimizer_\tau^u}\right)^2\nonumber\\
    &+ 2C_1^2 \rho_T^{2(k-1)} \left(\sum_{u \in N_v^r} \rho_S^{d_\mathcal{G}(u, v)} \norm{(x_{t+k-1\mid t-1}^u)^* - \minimizer_{t+k-1}^u}\right)^2\label{thm:per-step-err:e1:s1}\\
    \leq{}& 2C_1^2 \rho_S^{2r} \left(\sum_{\tau = t}^{t+k-1} \rho_T^{\tau - t} \sum_{u \in \partial N_{v}^{r} } 1\right) \left(\sum_{\tau = t}^{t+k-1} \rho_T^{\tau - t} \sum_{u \in \partial N_{v}^{r} } \norm{(x_{\tau\mid t-1}^u)^* - \minimizer_\tau^u}^2\right)\nonumber\\
    &+ 2C_1^2 \rho_T^{2(k-1)} \left(\sum_{u \in N_v^r} \rho_S^{d_\mathcal{G}(u, v)} \right) \left(\sum_{u \in N_v^r} \rho_S^{d_\mathcal{G}(u, v)} \norm{(x_{t+k-1\mid t-1}^u)^* - \minimizer_{t+k-1}^u}^2\right)\label{thm:per-step-err:e1:s2}\\
    \leq{}& \frac{2 C_1^2 h(r)}{1 - \rho_T} \cdot \rho_S^{2r} \left(\sum_{\tau = t}^{t+k-1} \rho_T^{\tau - t} \sum_{u \in \partial N_{v}^{r} } \norm{(x_{\tau\mid t-1}^u)^* - \minimizer_\tau^u}^2\right)\nonumber\\
    &+ 2C_1^2 C_3(r) \cdot \rho_T^{2(k-1)} \left(\sum_{u \in N_v^r} \rho_S^{d_\mathcal{G}(u, v)} \norm{(x_{t+k-1\mid t-1}^u)^* - \minimizer_{t+k-1}^u}^2\right),\label{thm:per-step-err:e1:s3}
\end{align}
\end{subequations}
where we used the AM-GM Inequality in \eqref{thm:per-step-err:e1:s1}; we used the Cauchy-Schwarz Inequality in \eqref{thm:per-step-err:e1:s2}; we used the definitions of functions $h(r)$ and $C_3(r)$ in \eqref{thm:per-step-err:e1:s3}.

Summing up \eqref{thm:per-step-err:e1} over all $v \in \mathcal{V}$ and reorganizing terms gives
\begin{align}\label{thm:per-step-err:e2}
    &\sum_{v \in \mathcal{V}} \norm{x_t^v - (x_{t \mid t-1}^v)^*}^2\nonumber\\
    \leq{}& \frac{2 C_1^2 h(r)}{1 - \rho_T} \cdot \rho_S^{2r} \sum_{v \in \mathcal{V}} \left(\sum_{\tau = t}^{t+k-1} \rho_T^{\tau - t} \sum_{u \in \partial N_{v}^{r} } \norm{(x_{\tau\mid t-1}^u)^* - \minimizer_\tau^u}^2\right)\nonumber\\
    &+ 2C_1^2 C_3(r) \cdot \rho_T^{2(k-1)} \sum_{v \in \mathcal{V}} \left(\sum_{u \in N_v^r} \rho_S^{d_\mathcal{G}(u, v)} \norm{(x_{t+k-1\mid t-1}^u)^* - \minimizer_{t+k-1}^u}^2\right)\nonumber\\
    \leq{}& \frac{2 C_1^2 h(r)^2}{1 - \rho_T} \cdot \rho_S^{2r} \sum_{\tau = t}^{t+k-1} \rho_T^{\tau - t} \norm{x_{\tau \mid t-1}^* - \minimizer_\tau}^2 + 2C_1^2 C_3(r)^2 \cdot \rho_T^{2(k-1)} \norm{x_{t+k-1 \mid t-1}^* - \minimizer_{t+k-1}}^2,
\end{align}
where we used the facts that
\begin{align*}
    \sum_{v \in \mathcal{V}}\sum_{u \in \partial N_v^r} \norm{(x_{\tau\mid t-1}^u)^* - \minimizer_\tau^u}^2 \leq{} h(r) \sum_{v \in \mathcal{V}}\norm{(x_{\tau\mid t-1}^v)^* - \minimizer_\tau^v}^2 = h(r)\cdot \norm{x_{\tau\mid t-1}^* - \theta_\tau}^2,
\end{align*}
and
\begin{align*}
    \sum_{v \in \mathcal{V}}\sum_{u \in \partial N_v^r} \rho_S^{d_\mathcal{G}(u, v)} \norm{(x_{t+k-1\mid t-1}^u)^* - \minimizer_{t+k-1}^u}^2 \leq{}& C_3(r) \sum_{v \in \mathcal{V}} \norm{(x_{t+k-1\mid t-1}^v)^* - \minimizer_{t+k-1}^v}^2\\
    ={}& C_3(r)\cdot \norm{x_{t+k-1 \mid t-1}^* - \minimizer_{t+k-1}}^2.
\end{align*}

We also note that by the principle of optimality, the following equations hold for all $\tau \geq t$:
\begin{align*}
    x_{\tau\mid t-1}^* &= \tilde{\psi}_{t}\left(x_{t-1} \right)_{\tau},\\
    x_\tau^* &= \tilde{\psi}_{t}\left(x_{t-1}^*\right)_{\tau}.
\end{align*}
Recall that $C_0 := \max\{1, C_G\}$. By \Cref{thm:global-exp-decay}, we see that
\begin{align}\label{thm:per-step-err:e2-1}
    \norm{x_{\tau\mid t-1}^* - x_\tau^*} \leq C_0 \rho_G^{\tau - t + 1} \norm{x_{t-1} - x_{t-1}^*},
\end{align}
which implies
\begin{subequations}\label{thm:per-step-err:e2-2}
\begin{align}
    \norm{x_{\tau \mid t-1}^* - \minimizer_\tau}^2 \leq{}& 2\norm{x_{\tau \mid t-1}^* - x_\tau^*}^2 + 2 \norm{x_\tau^* - \minimizer_\tau}^2\label{thm:per-step-err:e2-2:s1}\\
    \leq{}& 2 C_0^2 \rho_G^{2(\tau - t + 1)} \norm{x_{t-1} - x_{t-1}^*}^2 + 2 \norm{x_\tau^* - \minimizer_\tau}^2, \label{thm:per-step-err:e2-2:s2}
\end{align}
\end{subequations}
where we used the triangle inequality and the AM-GM inequality in \eqref{thm:per-step-err:e2-2:s1}; we used \eqref{thm:per-step-err:e2-1} in \eqref{thm:per-step-err:e2-2:s2}.

Substituting \eqref{thm:per-step-err:e2-2} into \eqref{thm:per-step-err:e2} gives
\begin{align}\label{thm:per-step-err:e3}
    &\sum_{v \in \mathcal{V}} \norm{x_t^v - (x_{t \mid t-1}^v)^*}^2\nonumber\\
    \leq{}& 4 C_1^2 C_0^2 \left(\frac{h(r)^2 \rho_G^2}{(1 - \rho_T)(1 - \rho_G^2 \rho_T)}\cdot \rho_S^{2r} + C_3(r)^2 \cdot \rho_T^{2(k-1)} \cdot \rho_G^{2k}\right) \norm{x_{t-1} - x_{t-1}^*}^2\nonumber\\
    &+ 4 C_1^2 \left(\frac{ h(r)^2}{1 - \rho_T} \cdot \rho_S^{2r} \sum_{\tau = t}^{t+k-1} \rho_T^{\tau - t} \norm{x_{\tau }^* - \minimizer_\tau}^2 + C_3(r)^2 \cdot \rho_T^{2(k-1)} \norm{x_{t+k-1}^* - \minimizer_{t+k-1}}^2\right)\nonumber\\
    \leq{}& 4 C_1^2 C_0^2 \left(\frac{h(r)^2 \rho_G^2}{(1 - \rho_T)(1 - \rho_G^2 \rho_T)}\cdot \rho_S^{2r} + C_3(r)^2 \cdot \rho_T^{2(k-1)} \cdot \rho_G^{2k}\right) \norm{x_{t-1} - x_{t-1}^*}^2\nonumber\\
    &+ \frac{8 C_1^2}{\mu} \left(\frac{ h(r)^2}{1 - \rho_T} \cdot \rho_S^{2r} \sum_{\tau = t}^{t+k-1} \rho_T^{\tau - t} f_\tau(x_\tau^*) + C_3(r)^2 \cdot \rho_T^{2(k-1)} f_{t+k-1}(x_{t+k-1}^*)\right),
\end{align}
where we used the fact that the node cost function $f_\tau^v$ is non-negative and $\mu$-strongly convex for all $\tau, v$, thus
\[f_\tau(x_\tau^*) \geq \sum_{v \in \mathcal{V}} f_\tau^v((x_\tau^v)^*) \geq \frac{\mu}{2} \sum_{v \in \mathcal{V}}\norm{(x_\tau^v)^* - \theta_\tau^v}^2 = \frac{\mu}{2}\norm{x_\tau^* - \theta_\tau}^2.\]

Note that $\sum_{v \in \mathcal{V}} \norm{x_t^v - (x_{t \mid t-1}^v)^*}^2 = \norm{x_t - x_{t\mid t-1}^*}^2 = e_t^2$. Thus we have finished the proof of \eqref{thm:per-step-err:e0}.

\subsection{Proof of Theorem \ref{coro:CR-LPC}}\label{apx:coro:CR-LPC}
In this section, we show \Cref{coro:CR-LPC} holds with the following specific constants:
\begin{align}\label{coro:CR-LPC:e1}
    1 + \left(1 + \frac{32 C_0^2 C_1^2 (\ell_f + \Delta \ell_S + 2\ell_T)\cdot h(r)^2}{\mu (1 - \rho_G)^2 (1 - \rho_T)^2}\right) \cdot \rho_S^r + \left(1 + \frac{32 C_0^2 C_1^2 (\ell_f + \Delta \ell_S + 2\ell_T) C_3(r)^2}{\mu (1 - \rho_G)^2}\right) \rho_T^{k-1}.
\end{align}
under the assumption that
\begin{align}\label{coro:CR-LPC:e2-0}
    \frac{4 C_1^2 C_0^4}{(1 - \rho_G)^2}\left(\frac{h(r)^2 \rho_G^2}{(1 - \rho_T)(1 - \rho_G^2 \rho_T)}\cdot \rho_S^{2r} + C_3(r)^2 \cdot \rho_T^{2(k-1)} \cdot \rho_G^{2k}\right) \leq \frac{1}{2}.
\end{align}

Recall that $C_0$ is defined in \Cref{thm:alg-perf-bound-err-inject}. Note that \Cref{thm:networked-exp-decay} and \Cref{thm:global-exp-decay} hold under Assumption \ref{assump:costs-and-feasible-sets}. One can check that $C_0, C_1, (1 - \rho_G)^{-1},$ and $(1 - \rho_T)^{-1}$ are bounded by polynomials of $\ell_f/\mu, \ell_T/\mu,$ and $(\Delta \ell_S)/\mu$.

In the proof, we need to use Lemma F.2 in \citet{lin2021perturbation} to bound LPC's total cost by a weighted sum of the offline optimal cost and the sum of squared distances between their trajectories. For completeness, we present Lemma F.2 in \citet{lin2021perturbation} below:

\begin{lemma}\label{lemma:smooth-difference}
For a fixed dimension $m \in \mathbb{Z}_+$, assume a function $h: \mathbb{R}^m \to \mathbb{R}_{\geq 0}$ is convex, $\ell$-smooth and continuously differentiable. For all $x, y \in \mathbb{R}^m$, for all $\eta > 0$, we have
\[h(x) \leq (1 + \eta) h(y) + \frac{\ell}{2}\left(1 + \frac{1}{\eta}\right)\norm{x - y}^2.\]
\end{lemma}

Now we come back to the proof of \Cref{coro:CR-LPC}. We first bound the sum of squared distances between LPC's trajectory and the offline optimal trajectory:
\begin{subequations}\label{coro:CR-LPC:e2}
\begin{align}
    \sum_{t=1}^\horizonlength \norm{x_t - x_t^*}^2 \leq{}& \frac{C_0^2}{(1 - \rho_G)^2} \sum_{t=1}^{\horizonlength} e_t^2\label{coro:CR-LPC:e2:s1}\\
    \leq{}& \frac{4 C_1^2 C_0^4}{(1 - \rho_G)^2}\left(\frac{h(r)^2 \rho_G^2}{(1 - \rho_T)(1 - \rho_G^2 \rho_T)}\cdot \rho_S^{2r} + C_3(r)^2 \cdot \rho_T^{2(k-1)} \cdot \rho_G^{2k}\right) \sum_{t=1}^\horizonlength \norm{x_{t-1} - x_{t-1}^*}^2\nonumber\\
    &+ \frac{8 C_0^2 C_1^2}{\mu (1 - \rho_G)^2} \sum_{t=1}^\horizonlength \left(\frac{ h(r)^2}{1 - \rho_T} \cdot \rho_S^{2r} \sum_{\tau = t}^{t+k-1} \rho_T^{\tau - t} f_\tau(x_\tau^*) + C_3(r)^2 \cdot \rho_T^{2(k-1)} f_{t+k-1}(x_{t+k-1}^*)\right),\label{coro:CR-LPC:e2:s2}
\end{align}
\end{subequations}
where we used \Cref{thm:alg-perf-bound-err-inject} in \eqref{coro:CR-LPC:e2:s1}; we used \Cref{thm:per-step-err} with the specific constants given in Appendix \ref{apx:thm:per-step-err} in \eqref{coro:CR-LPC:e2:s2}.

Recall that in \eqref{coro:CR-LPC:e2-0}, we assume $r$ and $k$ are sufficient large so that the coefficient of the first term in \eqref{coro:CR-LPC:e2} satisfies
\[\frac{4 C_1^2 C_0^4}{(1 - \rho_G)^2}\left(\frac{h(r)^2 \rho_G^2}{(1 - \rho_T)(1 - \rho_G^2 \rho_T)}\cdot \rho_S^{2r} + C_3(r)^2 \cdot \rho_T^{2(k-1)} \cdot \rho_G^{2k}\right) \leq \frac{1}{2}.\]
Substituting this bound into \eqref{coro:CR-LPC:e2} gives that
\begin{align}\label{coro:CR-LPC:e3}
    \sum_{t=1}^\horizonlength \norm{x_t - x_t^*}^2 \leq \frac{16 C_0^2 C_1^2}{\mu (1 - \rho_G)^2} \left(\frac{h(r)^2}{(1 - \rho_T)^2}\cdot \rho_S^{2r} + C_3(r)^2 \cdot \rho_T^{2(k-1)}\right)\cdot \sum_{t=1}^\horizonlength f_t(x_t^*).
\end{align}
By \Cref{lemma:smooth-difference}, since $f_t$ is $(\ell_f + \Delta \ell_S)$-smooth, convex, and non-negative on $\mathbb{R}^n$, and $c_t$ is $\ell_T$-smooth, convex, and non-negative on $\mathbb{R}^n \times \mathbb{R}^n$, we know that
\begin{align}\label{coro:CR-LPC:e4-0}
    f_t(x_t) &\leq (1 + \eta) f_t(x_t^*) + \frac{\ell_f + \Delta \ell_S}{2}\left(1 + \frac{1}{\eta}\right)\norm{x_t - x_t^*}^2\nonumber\\
    c_t(x_t, x_{t-1}) & \leq (1 + \eta) c_t(x_t^*, x_{t-1}^*) + \frac{\ell_T}{2}\left(1 + \frac{1}{\eta}\right)\left(\norm{x_t - x_t^*}^2 + \norm{x_{t-1} - x_{t-1}^*}^2\right)
\end{align}
holds for any $\eta > 0$. Summing the above inequality over $t$ gives
\begin{align}\label{coro:CR-LPC:e4}
    &\sum_{t=1}^\horizonlength \left(f_t(x_t) + c_t(x_t, x_{t-1})\right)\nonumber\\
    \leq{}& (1+\eta)\sum_{t=1}^\horizonlength \left(f_t(x_t^*) + c_t(x_{t}^*, x_{t-1}^*)\right) + \frac{(\ell_f + \Delta \ell_S + 2\ell_T)}{2}\left(1 + \frac{1}{\eta}\right)\sum_{t=1}^\horizonlength \norm{x_t - x_t^*}^2\nonumber\\
    \leq{}& (1+\eta) cost(OPT) + \left(1 + \frac{1}{\eta}\right) \frac{16 C_0^2 C_1^2 (\ell_f + \Delta \ell_S + 2\ell_T)}{\mu (1 - \rho_G)^2} \left(\frac{h(r)^2}{(1 - \rho_T)^2}\cdot \rho_S^{2r} + C_3(r)^2 \cdot \rho_T^{2(k-1)}\right)\cdot cost(OPT),
\end{align}
where we used \eqref{coro:CR-LPC:e3} and $\sum_{t=1}^\horizonlength f_t(x_t^*) \leq cost(OPT)$ in the last inequality. Setting $\eta = \rho_S^r + \rho_T^{k-1}$ in \eqref{coro:CR-LPC:e4} finishes the proof of \eqref{coro:CR-LPC:e1}.

As a remark, we require the local cost function $\left(f_t^v, c_t^v, s_t^e\right)$ to be non-negative, convex, and smooth in the whole Euclidean spaces $\left(\bR^n, \bR^n \times \bR^n, \bR^n \times \bR^n\right)$ in Assumption \ref{assump:costs-and-feasible-sets} because we want to apply \Cref{lemma:smooth-difference} in \eqref{coro:CR-LPC:e4-0}.

\subsection{Proof of Corollary \ref{coro:pure-exp-decay}}\label{apx:coro:pure-exp-decay}
We first show $\Delta^2 \rho_S \leq \sqrt{\rho_S}$ holds under \Cref{assump:costs-and-feasible-sets} and the assumptions that $\frac{\ell_S}{\mu} \leq \frac{1}{\Delta^{7}}, \frac{\ell_T}{\mu} \leq \frac{1}{16}$. To see this, note that as we discussed in Section \ref{sec:main:perturb}, by setting $b_1 = 2\Delta -1$ and $b_2 = 4\Delta^2 - 2\Delta$, \Cref{coro:CR-LPC} holds with
\begin{align*}
    \rho_S = \frac{4\Delta^2 (\sqrt{1 + \Delta \ell_S/\mu} - 1)}{\sqrt{1 + \Delta \ell_S/\mu} + 1}.
\end{align*}
Hence we see that
\[\Delta^2 \sqrt{\rho_S} = 2\Delta^3 \left(\frac{\sqrt{1 + (\Delta \ell_S/\mu)} - 1}{\sqrt{1 + (\Delta \ell_S/\mu)} + 1}\right)^{\frac{1}{2}} \leq 2\Delta^3 \left(\frac{\sqrt{1 + \Delta^{-6}} - 1}{2}\right)^{\frac{1}{2}} \leq 1,\]
which implies that
\begin{align}\label{coro:pure-exp-decay:e1}
    \Delta^2 \rho_S \leq \sqrt{\rho_S}.
\end{align}

Recall that function $C_3(r) := \sum_{\gamma = 0}^{r} h(\gamma) \cdot \rho_S^\gamma$. Hence we see that
\begin{align}\label{coro:pure-exp-decay:e2}
    C_3(r) \leq \sum_{\gamma = 0}^{r} \Delta^\gamma \cdot \rho_S^\gamma \leq \sum_{\gamma = 0}^{r} \left(\frac{\sqrt{\rho_S}}{\Delta}\right)^\gamma \leq \frac{\Delta}{\Delta - \sqrt{\rho_S}}.
\end{align}

Substituting \eqref{coro:pure-exp-decay:e1} and \eqref{coro:pure-exp-decay:e2} into the competitive ratio bound in \eqref{coro:CR-LPC:e1} shows that the competitive ratio of LPC is upper bound by
\[1 + \left(1 + \frac{32 C_0^2 C_1^2 (\ell_f + \Delta \ell_S + 2\ell_c)}{\mu (1 - \rho_G)^2 (1 - \rho_T)^2}\right) \cdot \rho_S^{\frac{r}{2}} + \left(1 + \frac{32 C_0^2 C_1^2 (\ell_f + \Delta \ell_S + 2\ell_c)\Delta^2}{\mu (1 - \rho_G)^2 (\Delta - \sqrt{\rho_S})^2}\right) \rho_T^{k-1}.\]

\section{Proof of Theorem \ref{thm:bottleneck}}\label{apx:thm:bottleneck}

In this appendix we prove a lower bound on the competitive ratio of any online algorithm.  Our proof focuses on temporal and spatial lower bounds separately first, and then combines them.  

\subsubsection*{Step 1: Temporal Lower Bounds}\label{apx:lower-bound:temporal}
We first show that the competitive ratio of any online algorithm with $k$ steps of future predictions is lower bounded by $1 + \Omega(\lambda_T^k)$. To show this, we consider the special case when there are no spatial interaction costs (i.e., $s_t^e \equiv 0$ for all $t$ and $e$). In this case, since all agents are independent with each other, it suffices to assume there is only one agent in the network $\mathcal{G}$. Thus we will drop the agent index in the following analysis. To further simplify the problem, we assume dimension $n = 1$, $c_t(x_t, x_{t-1}) = \frac{\ell_T}{2}(x_t - x_{t-1})^2$, and the feasible set is $D_t \equiv D = [0, 1]$ for all $t$. Let $R$ denote the diameter of $D$, i.e., $R = \sup_{x, y \in D}\abs{x - y} = 1$.

By Theorem 2 in \citet{li2020online} and Case 1 in its proof, we know that for any online algorithm $ALG$ with $k$ steps of future predictions and $L_T \in (2R, R\horizonlength)$, there exists a problem instance with quadratic functions $f_1, f_2, \ldots, f_\horizonlength$ that have the form $f_t(x_t) = \frac{\mu}{2}(x_t - \minimizer_t)^2, \minimizer_t \in D$ such that
\begin{align}\label{thm:bottleneck:e1}
    cost(ALG) - cost(OPT) \ge \frac{\mu^3(1 - \sqrt{\lambda_T})^2}{96(\mu + 1)^2}\cdot \lambda_T^k\cdot R\cdot L_\horizonlength,
\end{align}
where $L_\horizonlength \geq \sum_{t=1}^\horizonlength \abs{\minimizer_t - \minimizer_{t-1}}$. Note that
\begin{align*}
    R \cdot L_T \geq{}& \sum_{t=1}^\horizonlength \abs{v_t - v_{t-1}}^2\\
    ={}& \frac{2}{\ell_T} \cdot \sum_{t=1}^\horizonlength \left(f_t(v_t) + c_t(v_t, v_{t-1})\right)\\
    \geq{}& \frac{2}{\ell_T} \cdot cost(OPT).
\end{align*}
Substituting this into \eqref{thm:bottleneck:e1} gives
\begin{align}\label{thm:bottleneck:e2}
    cost(ALG) \geq \left(1 + \frac{\mu^3(1 - \sqrt{\lambda_T})^2}{48(\mu + 1)^2 \ell_T}\cdot \lambda_T^k\right)\cdot cost(OPT).
\end{align}
Note that \eqref{thm:bottleneck:e1} implies $cost(ALG) > 0$, hence the competitive ratio can be unbounded if $cost(OPT) = 0$.

\subsubsection*{Step 2: Spatial Lower Bounds}\label{apx:lower-bound:spatial}
We next show that the competitive ratio of any online algorithm that can communicate within $r$-hop neighborhood according to the scheme defined in Section \ref{sec:look-ahead-and-communication} is lower bounded by $1 + \Omega(\lambda_S^r)$. To show this, we will construct a special Networked OCO instance with random cost functions and show there exists a realization that achieves the lower bound by probabilistic methods.

\begin{theorem}\label{thm:spatial_lower_bound}
Under the assumption that $\Delta \geq 3$, the competitive ratio of any decentralized online algorithm $ALG$ with communication radius $r$ is lower bounded by $1 + \Omega(\lambda_S^r)$, where $\Omega(\cdot)$ notation hides factors that depend polynomially on $1/\mu, \ell_T, \ell_S,$ and $\Delta$, and
\begin{align}\label{thm:spatial_lower_bound:decay_factor}
    \lambda_S = \begin{cases} 
    \frac{(\Delta \ell_S/\mu)}{3 + 3 (\Delta \ell_S/\mu)} & \text{ if } \Delta \ell_S/\mu < 48,\\
    \max\left(\frac{(\Delta \ell_S/\mu)}{3 + 3 (\Delta \ell_S/\mu)}, \left(1 - 4 \sqrt{3} \cdot (\Delta \ell_S/\mu)^{-\frac{1}{2}}\right)^2\right) & \text{ otherwise.}
    \end{cases}
\end{align}
\end{theorem}

\begin{proof}[Proof of \Cref{thm:spatial_lower_bound}]
In the proof, we assume the online game only lasts one time step before it ends, i.e., $H = 1$. Note that when $H > 1$, the same counterexample can be constructed repeatedly by letting the temporal interaction costs $c_t^v \equiv 0$ for every agent $v$ and time step $t$. To simplify the notation, we define $\ell := \ell_S/\mu$ and $d := [\Delta/2]$. Without the loss of generality, we assume $\mathcal{V} = \{1, 2, \cdots, n\}$ so that each agent has a positive integer index.

We consider the case where the node cost function for each agent $i$ is $(x_i + w_i)^2$ and the spatial interaction cost between two neighboring agents $i$ and $j$ is $\ell(x_i - x_j)^2$. Here, $x_i \in \mathbb{R}$ is the scalar action of agent $i$, and parameter $w_i \in \mathbb{R}$ is a local information that corresponds to agent $i$. The parameters $\{w_i\}_{i=1}^n$ are sampled i.i.d. from some distribution $\mathcal{D}$, which we will discuss later.

For a general graph $\mathcal{G} = (\mathcal{V}, \mathcal{E})$ of agents, let $L$ denote its graph Laplacian matrix. Recall that the graph Laplacian matrix $L \in \mathcal{V} \times \mathcal{V}$ is a symmetric $n \times n$ matrix and it is defined as
\[L_{i, j} = \begin{cases}
deg(i) & \text{ if }i = j,\\
-1 & \text{ if } i \not= j \text{ and } (i, j) \in \mathcal{E},\\
0 & \text{ otherwise,}
\end{cases}\]
for agents $i, j \in \mathcal{V}$. Here $deg(\cdot)$ denotes the degree of an agent in graph $\mathcal{G}$. We know that $L$ is a symmetric semi-definite positive semi-definite and has bandwidth $1$ w.r.t. to $\mathcal{G}$. The centralized optimization problem can be expressed as
\begin{align*}
    cost(OPT) &= \min_{x \in \mathbb{R}^n} (x + w)^\top (x + w) + \ell \cdot x^\top L x\\
    &= \min_{x \in \mathbb{R}^n} \norm{(I + \ell \cdot L)^{\frac{1}{2}} x + (I + \ell\cdot L)^{-\frac{1}{2}} w}^2 + w^\top (I - (I + \ell\cdot L)^{-1}) w\\
    &= w^\top (I - (I + \ell\cdot L)^{-1}) w,
\end{align*}
where the minimum is attained at $x^* = (I + \ell \cdot L)^{-1} w$.

When each agent $i$ only has communication radius $r$, it can only observe the part of $w$ that is within $N_i^r$. To simplify the notation, we define the mask operator $\phi_S: \mathbb{R}^n \to \mathbb{R}^n$ w.r.t. a set $S \subseteq \mathcal{V}$ as
\[\phi_S(w)_i = \begin{cases}
w_i & \text{ if } i \in S,\\
0 & \text{ otherwise,}
\end{cases}\]
for $i \in \mathcal{V}$. The local policy of agent $i$ (denote as $\pi_i$) is a mapping from $w_{N_i^r}$ to the local decision $x_i$.

Suppose the distribution $\mathcal{D}$ of each local parameters $w_i$ is a mean-zero distribution with support on $\mathbb{R}$. For every agent $i \in \mathcal{V}$, we see that
\begin{subequations}\label{thm:spatial_lower_bound:e1}
\begin{align}
    \mathbb{E}_w \abs{x_i(w) - x_i^*(w)}^2 &={} \min_{\pi_i} \mathbb{E}_w \abs{\pi_i(w_{N_i^r}) - x_i^*(w)}^2\nonumber\\
    &\geq{} \mathbb{E}_w \abs{\mathbb{E}[x_i^*(w)\mid w_{N_i^r}] - x_i^*(w)}^2\label{thm:spatial_lower_bound:e1:s1}\\
    &={} \mathbb{E}_w \abs{\mathbb{E}[\left((I + \ell \cdot L)^{-1} w\right)_i \mid w_{N_i^r}] - \left((I + \ell\cdot L)^{-1} w\right)_i}^2\nonumber\\
    &={} \mathbb{E}_w \abs{\left((I + \ell \cdot L)^{-1} \phi_{N_i^r}(w)\right)_i - \left((I + \ell \cdot L)^{-1} w\right)_i}^2\label{thm:spatial_lower_bound:e1:s2}\\
    &={} \mathbb{E}_w \abs{\left((I + \ell \cdot L)^{-1} \phi_{N_{-i}^r}(w)\right)_i}^2, \label{thm:spatial_lower_bound:e1:s3}
\end{align}
\end{subequations}
where we use the fact that conditional expectations minimize the mean square prediction error in \eqref{thm:spatial_lower_bound:e1:s1}; we use the requirement that the distribution of $w$ is mean-zero in \eqref{thm:spatial_lower_bound:e1:s2}.

To bound the variance term in \eqref{thm:spatial_lower_bound:e1:s3}, we need the following lemma to lower bound the magnitude of every entry in the exponential decaying matrix $(I + \ell \cdot L)^{-1}$:

\begin{lemma}\label{lemma:inverse-graph-lower-bound}
There exists a finite graph $\mathcal{G}$ with maximum degree $2d$ that satisfies the following conditions: For any two vertices $i, j$ such that $d_\mathcal{G}(i, j) \geq 3$, the following inequality holds:
\[\left((I + \ell \cdot L)^{-1}\right)_{i j} \geq \frac{d_\mathcal{G}(i, j)}{d^2 (2 d\ell + 1)}\cdot \left(\frac{d\ell}{2 d \ell + 1}\right)^{d_\mathcal{G}(i, j)}.\]
If we make the additional assumption that $\ell > \frac{16}{d}$, we have that
\[\left((I + \ell \cdot L)^{-1}\right)_{i j} \geq \frac{1}{4 \sqrt{\pi \cdot d_\mathcal{G}(i, j) \cdot \sqrt{d\ell}} \cdot d^2 (2d\ell + 1)} \cdot \left(1 - 4 (d \ell)^{-\frac{1}{2}}\right)^{d_\mathcal{G}(i, j)}.\]
\end{lemma}

We defer the proof of \Cref{lemma:inverse-graph-lower-bound} to the end of this section. Note that \Cref{lemma:inverse-graph-lower-bound} implies that there exists a graph $\mathcal{G}$ that satisfies $\left((I + \ell \cdot L)^{-1}\right)_{i, j} = \Omega\left(\lambda_S^r\right)$, where $\Omega(\cdot)$ notation hides factors that depend polynomially on $1/\mu, \ell_T, \ell_S,$ and $\Delta$, and $\lambda_S$ is as defined in \eqref{thm:spatial_lower_bound:decay_factor}. We assume the agents are located in this graph $\mathcal{G}$ for the rest of the proof.

Using \Cref{lemma:inverse-graph-lower-bound}, we can derive the following lower bound of the variance term in \eqref{thm:spatial_lower_bound:e1:s2}:
\begin{align}\label{thm:spatial_lower_bound:e2}
    \mathbb{E}_w \abs{\left((I + \ell \cdot L)^{-1} \phi_{N_{-i}^r}(w)\right)_i}^2
    ={}& \mathbb{E}_w \left(\sum_{j \in N_{-i}^r} \left((I + \ell\cdot L)^{-1}\right)_{i j} w_j\right)^2\nonumber\\
    ={}& \sum_{j \in N_{-i}^r} \left((I + \ell\cdot L)^{-1}\right)_{i j}^2 Var(w_j)\nonumber\\
    \geq{}& \sum_{j \in \partial N_{i}^{r+1}} \left((I + \ell\cdot L)^{-1}\right)_{i j}^2 Var(w_j)\nonumber\\
    \geq{}& \Theta\left(\lambda_S^r \cdot Var(w_i)\right).
\end{align}
Substituting \eqref{thm:spatial_lower_bound:e2} into \eqref{thm:spatial_lower_bound:e1} and summing over all vertices $i$, we obtain that
\begin{align*}
    \mathbb{E}_w\norm{x(w) - x^*(w)}^2 \geq \sum_{i=1}^n \mathbb{E}_w \abs{x_i(w) - x_i^*(w)}^2 \geq \Theta\left(n \cdot \lambda_S^r \cdot Var(w_i)\right).
\end{align*}
We also see that
\begin{align}\label{thm:spatial_lower_bound:e3}
    \mathbb{E}_w [cost(OPT)] = \mathbb{E}_w \left[w^\top (I - (I + \ell \cdot L)^{-1}) w\right] = O(n \cdot Var(w_i)).
\end{align}
Note that the global objective function $(x + w)^\top (x + w) + \ell \cdot x^\top L x$ is $1$-strongly convex, and $x^*(w)$ is minimizer of this function. Thus, we have that for any outcome of $w$,
\[cost(ALG) - cost(OPT) \geq \frac{1}{2}\norm{x(w) - x^*(w)}^2.\]
Taking expectations on both sides w.r.t. $w$ gives that
\begin{align}\label{thm:spatial_lower_bound:e4}
    \mathbb{E}_w cost(ALG) - \mathbb{E}_w cost(OPT) \geq \frac{1}{2}\mathbb{E}_w\norm{x(w) - x^*(w)}^2
    \geq \Theta\left(n \cdot \lambda_S^r \cdot Var(w_i)\right).
\end{align}
Dividing \eqref{thm:spatial_lower_bound:e4} by \eqref{thm:spatial_lower_bound:e3}, we obtain that
\[\frac{\mathbb{E}_w cost(ALG)}{\mathbb{E}_w cost(OPT)} \geq 1 + \Omega\left(\lambda_S^r\right).\]
Note that $\mathbb{P}_w\left[cost(OPT) = 0\right] = 0$. Thus, there must exist an instance of $w$ such that $cost(OPT) > 0$ and
\[\frac{cost(ALG)}{cost(OPT)} \geq 1 + \Omega\left(\lambda_S^r\right).\]

\end{proof}

Before we present the proof of \Cref{lemma:inverse-graph-lower-bound}, we first need to introduce two technical lemmas that will be used in the proof of \Cref{lemma:inverse-graph-lower-bound}. The first lemma (\Cref{lemma:combinatorial-number-near-center}) provides a lower bound for binomial coefficient $\binom{(2+\epsilon)m}{m}$.

\begin{lemma}\label{lemma:combinatorial-number-near-center}
For any positive integer $m$ and $\epsilon \in \mathbb{R}_{\geq 0}$ such that $\epsilon m$ is an integer, the following inequality holds:
\[\binom{(2+\epsilon)m}{m} > \frac{1}{\sqrt{2\pi}} m^{-\frac{1}{2}} \cdot \frac{(2 + \epsilon)^{(2+\epsilon)m+\frac{1}{2}}}{(1 + \epsilon)^{(1+\epsilon)m + \frac{1}{2}}} \cdot  e^{-\frac{1}{6m}}.\]
\end{lemma}
\begin{proof}[Proof of \Cref{lemma:combinatorial-number-near-center}]
By Lemma 2.1 in \cite{stanica2001good}, we know for any $n \in \mathbb{Z}_+$,
\[n! = \sqrt{2\pi} n^{n + \frac{1}{2}} e^{-n+r(n)},\]
where $r(n)$ satisfies $\frac{1}{12 n + 1} < r(n) < \frac{1}{12 n}$. Thus we see that
\[\sqrt{2 \pi} n^{n + \frac{1}{2}} e^{-n+\frac{1}{12n+1}} < n! < \sqrt{2\pi} n^{n + \frac{1}{2}} e^{-n + \frac{1}{12n}}, \forall n \in \mathbb{Z}_+.\]
Therefore, we can lower bound $\binom{(2+\epsilon)m}{m}$ by
\begin{align*}
    \binom{(2+\epsilon)m}{m} ={}& \frac{((2+\epsilon)m)!}{m! \cdot ((1 + \epsilon)m)!}\\
    >{}& \frac{\sqrt{2\pi}((2+\epsilon)m)^{(2+\epsilon)m + \frac{1}{2}} e^{-(2+\epsilon)m + \frac{1}{12(2+\epsilon)m + 1}}}{\sqrt{2\pi}m^{m+\frac{1}{2}} e^{-m+\frac{1}{12m}}\cdot \sqrt{2\pi}((1+\epsilon)m)^{(1+\epsilon)m + \frac{1}{2}} e^{-(1+\epsilon)m + \frac{1}{12(1+\epsilon)m}}}\\
    ={}& \frac{1}{\sqrt{2\pi}} m^{-\frac{1}{2}} \cdot \frac{(2 + \epsilon)^{(2+\epsilon)m+\frac{1}{2}}}{(1 + \epsilon)^{(1+\epsilon)m + \frac{1}{2}}} \cdot  e^{\frac{1}{12(2+\epsilon)m + 1} - \frac{1}{12m} - \frac{1}{12(1+\epsilon)m}}\\
    >{}& \frac{1}{\sqrt{2\pi}} m^{-\frac{1}{2}} \cdot \frac{(2 + \epsilon)^{(2+\epsilon)m+\frac{1}{2}}}{(1 + \epsilon)^{(1+\epsilon)m + \frac{1}{2}}} \cdot  e^{-\frac{1}{6m}}.
\end{align*}
\end{proof}

The second technical lemma (\Cref{lemma:numerical-lower-bound}) will be used to simplify the decay factor in the proof of \Cref{lemma:inverse-graph-lower-bound}.

\begin{lemma}\label{lemma:numerical-lower-bound}
For all $\epsilon \in [0, \sqrt{2})$, the following inequality holds
\[\frac{2 + \epsilon}{2 \cdot (1 + \epsilon)^{\frac{1+\epsilon}{2+\epsilon}}} \geq 1 - \frac{\epsilon^2}{2}.\]
\end{lemma}
\begin{proof}[Proof of \Cref{lemma:numerical-lower-bound}]
By taking logarithm on both sides, we see the original inequality is equivalent to
\begin{equation*}
    \ln\left(1 + \frac{\epsilon}{2}\right) - \frac{1 + \epsilon}{2 + \epsilon}\ln(1+\epsilon) \geq \ln\left(1 - \frac{1}{2}\epsilon^2\right),
\end{equation*}
which is further equivalent to
\begin{equation}\label{lemma:numerical-lower-bound:e1}
    \ln\left(1 + \frac{\epsilon}{2}\right) - \frac{1 + \epsilon}{2 + \epsilon}\ln(1+\epsilon) - \ln\left(1 - \frac{1}{2}\epsilon^2\right) \geq 0.
\end{equation}
Note that the LHS can be lower bounded by
\begin{equation*}
    \ln\left(1 + \frac{\epsilon}{2}\right) - \frac{1 + \epsilon}{2 + \epsilon}\ln(1+\epsilon) - \ln\left(1 - \frac{1}{2}\epsilon^2\right) \geq \ln\left(1 + \frac{\epsilon}{2}\right) - \frac{1 + \epsilon}{2}\ln(1+\epsilon) - \ln\left(1 - \frac{1}{2}\epsilon^2\right) =: g(\epsilon).
\end{equation*}
Function $g$ satisfies that $g(0) = 0$, and its derivative is
\begin{align*}
    g'(\epsilon) ={}& \frac{1}{2 + \epsilon} - \frac{1}{2} - \frac{1}{2}\ln(1 + \epsilon) + \frac{\epsilon}{1 - \frac{1}{2}\epsilon^2}\\
    \geq{}& \frac{1}{2 + \epsilon} - \frac{1}{2} - \frac{\epsilon}{2} + \epsilon\\
    ={}& \frac{2 - (2+\epsilon)(1-\epsilon)}{2(2+\epsilon)}\\
    ={}& \frac{\epsilon + \epsilon^2}{2(2 + \epsilon)} \geq 0.
\end{align*}
Thus, $g(\epsilon) \geq 0$ for all $\epsilon \in [0, \sqrt{2})$. Hence \eqref{lemma:numerical-lower-bound:e1} holds for all $\epsilon \in [0, \sqrt{2})$.
\end{proof}

Now we are ready to present the proof of \Cref{lemma:inverse-graph-lower-bound}.

\begin{proof}[Proof of \Cref{lemma:inverse-graph-lower-bound}]
\begin{figure}[H]
    \centering
    \includegraphics[width=.9\textwidth]{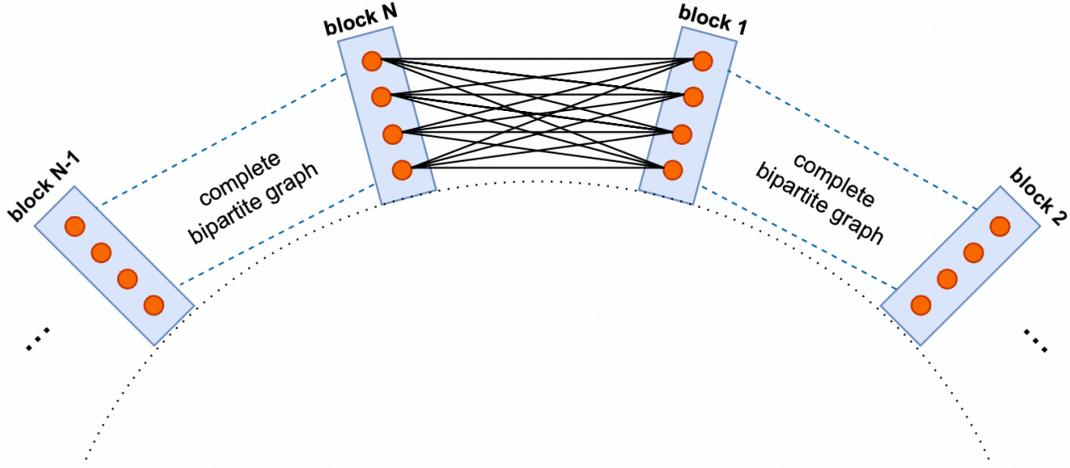}
    \caption{Graph structure of $\mathcal{G}$ to obtain the lower bound: $N$ blocks form a ring. Each block contains $d$ vertices.}
    \label{fig:G_exp_decay_lower}
\end{figure}

Consider the graph $\mathcal{G}$ constructed as Figure \ref{fig:G_exp_decay_lower}: Let $N$ be a positive integer that is sufficiently large. $N$ blocks form a ring, where each block contains $d$ nodes. Every pair of blocks are connected by a complete bipartite graph.
The graph Laplacian of $\mathcal{G}$ can be decomposed as $L = 2d I - M$, where $M$ is the adjacency matrix of $\mathcal{G}$. We see that
\begin{align*}
    (I + \ell \cdot L)^{-1} &= \left((2d \ell + 1)I - \ell \cdot M\right)^{-1}\\
    &= \frac{1}{2d \ell + 1}\left(I - \frac{\ell}{2d \ell + 1}M\right)^{-1}\\
    &= \frac{1}{2d \ell + 1}\sum_{t=0}^\infty \frac{\ell^t}{(2d \ell + 1)^t} M^t
\end{align*}
Fix two vertices $i$ and $j$ and denote $\kappa := d_\mathcal{G}(i, j)$ and assume $\kappa \geq 3$. Without the loss of generality, we can assume $j$ is on the clockwise direction of $i$. We see that
\begin{align}\label{equ:graph-Laplacian-inverse:e1}
    \left((I + \ell \cdot L)^{-1}\right)_{i j} ={}& \frac{1}{2d\ell + 1}\sum_{t=0}^\infty \frac{\ell^t}{(2d\ell + 1)^t} (M^t)_{i j}\nonumber\\
    ={}& \frac{\ell^\kappa}{(2d\ell+1)^{\kappa+1}}\sum_{m=0}^\infty \frac{\ell^{2m}}{(2d\ell + 1)^{2m}} (M^{\kappa + 2m})_{i j}.
\end{align}
Note that $(M^{\kappa + 2m})_{i j}$ denotes the number of paths from $i$ to $j$ with length $\kappa + 2m$ in graph $\mathcal{G}$. Note that the shortest paths from $i$ to $j$ have length $\kappa$. To pick a path with length $(\kappa+2m)$ from $i$ to $j$, we can first pick a path on the level of blocks: The number of possible block-level paths is lower bounded by $\binom{\kappa+2m}{m}$ because we can choose $m$ in $(\kappa + 2m)$ steps to go in the counter clockwise direction. After a block-level path is fixed, we can choose which specific vertices in the blocks we want to land at, and there are $d^{\kappa+2m-2}$ choices. Thus we see that
\[(M^{\kappa + 2m})_{i j} \geq \binom{\kappa+2m}{m} d^{\kappa+2m-2}.\]
Substituting this into \eqref{equ:graph-Laplacian-inverse:e1} gives
\begin{align}\label{equ:graph-Laplacian-inverse:e2}
    \left((I + \ell \cdot L)^{-1}\right)_{i j} \geq{}& \frac{\ell^\kappa d^{\kappa-2}}{(2 d \ell+1)^{\kappa+1}}\sum_{m=0}^\infty \frac{\ell^{2m} d^{2m}}{(2 d \ell + 1)^{2m}} \binom{\kappa+2m}{m}.
\end{align}
Let $m = 0$ will give that $\left((I + \ell \cdot L)^{-1}\right)_{i j} \geq \frac{\kappa}{d^2 (2d\ell + 1)} \cdot \left(\frac{d\ell}{2d\ell+1}\right)^\kappa$, which shows the first claim of \Cref{lemma:inverse-graph-lower-bound}. Now we proceed to show the second claim of \Cref{lemma:inverse-graph-lower-bound}.

By \Cref{lemma:combinatorial-number-near-center}, we know that when $\kappa = \epsilon m$, we have that
\begin{align*}
    \binom{\kappa + 2m}{m} ={}& \binom{(2+\epsilon)m}{m}\\
    >{}& \frac{1}{\sqrt{2\pi}} e^{-\frac{1}{6m}} m^{-\frac{1}{2}} \frac{(2 + \epsilon)^{(2 + \epsilon)m}}{(1 + \epsilon)^{(1 + \epsilon)m}} \cdot \left(\frac{2+\epsilon}{1+\epsilon}\right)^{\frac{1}{2}}\\
    \geq{}& \frac{1}{2\sqrt{2\pi}} m^{-\frac{1}{2}} \left(\frac{2 + \epsilon}{(1 + \epsilon)^{\frac{1+\epsilon}{2+\epsilon}}}\right)^{(2 + \epsilon)m}.
\end{align*}
For any $m > \kappa$, the inequality we just showed can help us bound a term in the summation of \eqref{equ:graph-Laplacian-inverse:e2} below:
\begin{align*}
    &\frac{\ell^\kappa d^{\kappa-2}}{(2d\ell+1)^{\kappa+1}}\cdot \frac{\ell^{2m} d^{2m}}{(2d\ell + 1)^{2m}} \binom{\kappa+2m}{m}\\
    \geq{}& \frac{1}{d^2(2d\ell+1)} \cdot \frac{1}{2^{\kappa + 2m}} \binom{\kappa+2m}{m} \cdot \left(1 - \frac{1}{2d\ell+1}\right)^{\kappa + 2m}\\
    \geq{}& \frac{1}{2\sqrt{2\pi}\cdot \sqrt{\frac{\kappa}{\epsilon}} \cdot d^2 (2d\ell + 1)} \cdot \left(\frac{2 + \epsilon}{2 \cdot (1 + \epsilon)^{\frac{1+\epsilon}{2+\epsilon}}}\right)^{(2 + \epsilon)\kappa/\epsilon} \cdot \left(1 - \frac{1}{2d\ell+1}\right)^{(1 + \frac{2}{\epsilon})\kappa}\\
    \geq{}& \frac{1}{2\sqrt{2\pi}\cdot \sqrt{\frac{\kappa}{\epsilon}} \cdot d^2 (2d\ell + 1)} \cdot \left(\left(1 - \frac{\epsilon^2}{2}\right)^{\frac{1}{\epsilon}} \cdot \left(1 - \frac{1}{2d\ell+1}\right)^{\frac{1}{\epsilon}}\right)^{(2 + \epsilon) \kappa},
\end{align*}
where the last line follows from \Cref{lemma:numerical-lower-bound}.

Thus, we obtain that the following inequality holds for arbitrary $\epsilon \in (0, 1)$:
\begin{align}\label{equ:graph-Laplacian-inverse:e3}
    \left((I + L)^{-1}\right)_{i j} \geq \frac{1}{2\sqrt{2\pi}\cdot \sqrt{\frac{\kappa}{\epsilon}} \cdot d^2 (2d\ell + 1)} \cdot \left(\left(1 - \frac{\epsilon^2}{2}\right)^{\frac{2}{\epsilon}+1} \cdot \left(1 - \frac{1}{2d\ell+1}\right)^{\frac{2}{\epsilon}+1}\right)^{\kappa}.
\end{align}

By setting $\epsilon$ such that $1/\epsilon = \left[2(d\ell)^{\frac{1}{2}}\right]$ in \eqref{equ:graph-Laplacian-inverse:e3}, we obtain that:
\begin{align}\label{equ:graph-Laplacian-inverse:e4}
    \left((I + \ell \cdot L)^{-1}\right)_{i j} \geq{}& \frac{1}{2\sqrt{2\pi}\cdot \sqrt{2\kappa \cdot \sqrt{d\ell}} \cdot d^2 (2d\ell + 1)} \cdot \left(\left(1 - \frac{1}{2d\ell}\right)^{4\sqrt{d\ell}+1} \cdot \left(1 - \frac{1}{2d\ell+1}\right)^{4\sqrt{d\ell}+1}\right)^{\kappa}\nonumber\\
    \geq{}& \frac{1}{4\sqrt{\pi \cdot \kappa \cdot \sqrt{d\ell}} \cdot d^2 (2d\ell + 1)} \cdot \left(\left(1 - \frac{4\sqrt{d\ell} + 1}{2d\ell} \right) \cdot \left(1 - \frac{4\sqrt{d\ell} + 1}{2d\ell+1}\right)\right)^{\kappa}\nonumber\\
    \geq{}& \frac{1}{4\sqrt{\pi \cdot \kappa \cdot \sqrt{d\ell}} \cdot d^2 (2d\ell + 1)} \cdot \left(1 - \frac{4}{\sqrt{d\ell}}\right)^{\kappa}.
\end{align}
\end{proof}

\subsubsection*{Step 3: Combine Temporal and Spatial Lower Bounds}
Combining the results of Steps 1 and 2 together, we know that the competitive ratio of any decentralized online algorithm is lower bounded by
\[\max\{1 + \frac{\mu^3(1 - \sqrt{\lambda_T})^2}{48(\mu + 1)^2 \ell_T}\cdot \lambda_T^k, 1 + \Omega(\lambda_S^{r})\} = 1 + \Omega(\lambda^k) + \Omega(\lambda_S^r).\]

\section{Proof of Corollary \ref{coro:resource-augmentation}}\label{apx:coro:resource-augmentation}

In this appendix we prove a resource augmentation bound for LPC.  To simplify the notation, we define the shorthand  $a_T := \ell_T/\mu$ and $a_S := \ell_S/\mu$. $a_T$ and $a_S$ are positive real numbers. We first show two lemmas about the relationships between the decay factors $\rho_T$ and $\lambda_T$, and $\rho_S$ and $\lambda_S$.

\begin{lemma}\label{lemma:temporal-upper-and-lower}
Under the assumptions of \Cref{thm:networked-exp-decay}, we have $\rho_T^4 \leq \lambda_T \leq \rho_T^2.$
\end{lemma}
\begin{proof}[Proof of Lemma \ref{lemma:temporal-upper-and-lower}]
Recall that $\rho_T$ is given by
\[\rho_T = \sqrt{1 - \frac{2}{\sqrt{1 + 2a_T}+1}}\]
in \Cref{thm:networked-exp-decay}. Thus we see that
\begin{align*}
    \rho_T^4 = \left(1 - \frac{2}{\sqrt{1 + 2a_T}+1}\right)^2 \leq \left(1 - \frac{2}{\sqrt{1 + 4a_T}+1}\right)^2 = \lambda_T.
\end{align*}
On the other hand, we have that
\[\lambda_T - \rho_T^2 = \left(1 - \frac{2}{\sqrt{1 + 4a_T}+1}\right)^2 - 1 + \frac{2}{\sqrt{1 + 2a_T}+1} = \frac{4\sqrt{(1+2a_T)}\left(\sqrt{1 + 2a_T} - \sqrt{1+4a_T}\right)}{\left(\sqrt{1+2a_T} + 1\right)\left(\sqrt{1 + 4a_T}+1\right)^2} \leq 0.\]
\end{proof}

\begin{lemma}\label{lemma:spatial-upper-and-lower}
Under the assumptions of \Cref{thm:networked-exp-decay}, we have $\rho_S^{32} \leq \lambda_S$.
\end{lemma}
\begin{proof}[Proof of \Cref{lemma:spatial-upper-and-lower}]
Recall that $\rho_T$ is given by
\[\rho_S = \sqrt{1 - \frac{2}{\sqrt{1 + \Delta a_S}+1}}\]
in \Cref{thm:networked-exp-decay}. We consider the following three cases separately.

\textbf{Case 1}: $\Delta a_S \geq 224$. We have $\rho_S^{32} \leq \lambda_S$ in this case.

We first show that the following inequality holds for any positive integer $n_0$ and $x \in [0, 1/(2n_0)]$:
\begin{align}\label{lemma:spatial-upper-and-lower:e1}
    (1 - x)^{2n_0} \leq 1 - n_0 x.
\end{align}
To see this, define function $g(x) = (1 - x)^{2n_0} + n_0 x - 1$. Note that $g$ is a convex function with $g(0) = 0$ and
\[g\left(\frac{1}{2n_0}\right) = \left(1 - \frac{1}{2n_0}\right)^{2n_0} - \frac{1}{2} \leq e^{-1} - \frac{1}{2} < 0.\]
Thus, we see that $g(x) \leq 0$ holds for all $x \in [0, 1/(2n_0)]$. Hence \eqref{lemma:spatial-upper-and-lower:e1} holds.

By \eqref{lemma:spatial-upper-and-lower:e1}, we see that
\begin{align*}
    \rho_S^{16} &= \left(1 - \frac{2}{\sqrt{1 + \Delta a_S} + 1}\right)^8\\
    &\leq 1 - \frac{8}{\sqrt{1 + \Delta a_S} + 1}\\
    &\leq 1 - \frac{4\sqrt{3}}{\sqrt{\Delta a_S}}\\
    &= \sqrt{\lambda_S}.
\end{align*}

\textbf{Case 2}: $1 \leq \Delta a_S < 224$. We have $\rho_S^{28} \leq \lambda_S$ in this case.

To see this, note that $\rho_S^2 \leq 1 - \frac{2}{\sqrt{1 + 224} + 1} = \frac{7}{8}$, and by \Cref{thm:bottleneck}, $\lambda_S \geq \frac{1}{3 + 3\cdot 1} = \frac{1}{6}$. Therefore, we see that
\[\rho_S^{28} \leq \left(\frac{7}{8}\right)^{14} \leq \frac{1}{6} \leq \lambda_S.\]

\textbf{Case 3}: $\Delta a_S < 1$. We have $\rho_S^4 \leq \lambda_S$ in this case.

To see this, note that
\begin{align*}
    \rho_S^2 = \frac{\sqrt{1 + \Delta a_S} - 1}{\sqrt{1 + \Delta a_S} + 1} \leq \frac{\sqrt{\Delta a_S}}{4}.
\end{align*}
Thus we see that
\begin{align*}
    \rho_S^4 \leq \frac{\Delta a_S}{16} \leq \frac{\Delta a_S}{3 + 3 \Delta a_S} = \lambda_S.
\end{align*}
\end{proof}

Now we come back to the proof of Corollary \ref{coro:resource-augmentation}. By \Cref{coro:CR-LPC} and \Cref{thm:bottleneck}, we know that the optimal competitive ratio is lower bounded by
\begin{align*}
    c(k^*, r^*) \geq 1 + C_\lambda \left(\lambda_T^{k^*} + \lambda_S^{r^*}\right)
\end{align*}
and LPC's competitive ratio is upper bounded by
\[c_{LPC}(k, r) := 1 + C_\rho \left(C_3(r)^2 \cdot \rho_T^{k} + h(r)^2 \cdot \rho_S^{r}\right),\]
where $C_\lambda$ and $C_\rho$ are some positive constants. To achieve $c_{LPC}(k, r) \leq c(k^*, r^*)$, it suffices to guarantee that
\[C_\rho \cdot C_3(r)^2 \cdot \rho_T^k \leq C_\lambda \lambda_T^{k^*} \text{ and } C_\rho \cdot h(r)^2 \cdot \rho_S^{r} \leq C_\lambda \lambda_S^{r^*}.\]
Note that $C_3(r)$ can be upper bounded by some constant and $h(r)^2 \leq poly(r) \cdot \rho_S^{-\frac{r}{2}}$ under our assumptions. Applying Lemma \ref{lemma:temporal-upper-and-lower} and Lemma \ref{lemma:spatial-upper-and-lower} finishes the proof.

\end{document}